\documentclass[psamsfonts]{amsart}

\usepackage{hyperref}
\hypersetup{
    colorlinks,
    citecolor=black,
    filecolor=black,
    linkcolor=black,
    urlcolor=black
}
\usepackage{geometry}
\usepackage{pxfonts}
\usepackage[fontsize=11pt]{scrextend}
\usepackage{amssymb,amsfonts}
\usepackage[all,arc]{xy}
\usepackage{enumerate}
\usepackage{mathrsfs}
\usepackage[utf8]{inputenc}
\usepackage{url}
\usepackage{mathtools}
\usepackage{enumerate}
\usepackage[shortlabels]{enumitem}
\usepackage{tikz}
\usetikzlibrary{trees, snakes}
\usepackage{color}
\usepackage{bbm}
\usepackage{graphicx}
\usepackage{enumitem}

\calclayout

\DeclareMathOperator\Ft{Ft}
\DeclareMathOperator\arctanh{arctanh}
\DeclareMathOperator\Fr{Fr}
\DeclareMathOperator\Gr{Gr}
\DeclareMathOperator\PSL{PSL}

\DeclareMathOperator\Id{Id}

\DeclareMathOperator\vol{vol}

\DeclareMathOperator\diag{diag}

\DeclareMathOperator\SL{SL}
\DeclareMathOperator\supp{supp}

\DeclareMathOperator\SO{SO}

\DeclareMathOperator\Aut{Aut}

\DeclareMathOperator\dist{dist}

\DeclareMathOperator\area{area}

\DeclareMathOperator\Prob{\bP\br\bo\bb}
\DeclareMathOperator\ft{\mathbf{f}\bt}
\DeclareMathOperator\bary{\mathbf{b}\ba\br}
\DeclareMathOperator\ab{\ba\mathbf{b}\ba\br}
\DeclareMathOperator\core{core}
\DeclareMathOperator\T{T}
\DeclareMathOperator\N{N}

\DeclareMathOperator\proj{proj}

\newcommand{\lef}{\left}
\newcommand{\ri}{\right}

\newcommand{\wk}{\rightharpoonup}
\newcommand{\wkstar}{\stackrel{\star}\wk}

\newcommand{\iinfty}{\stackrel{i\to\infty}\arr}
\newcommand{\epszero}{\stackrel{\epsilon\to 0}\arr}

\newcommand{\aand}{\quad\text{and}\quad}
\newcommand{\eps}{\epsilon}

\newcommand{\la}{\langle}
\newcommand{\ra}{\rangle}
\newcommand{\til}{\tilde}
\newcommand{\se}{\subset}

\newcommand{\wt}{\widetilde}

\newcommand{\cH}{\mathcal{H}}

\newcommand{\fa}{\mathfrak{a}}
\newcommand{\fb}{\mathfrak{b}}

\newcommand{\fm}{\mathfrak{m}}

\newcommand{\fs}{\mathfrak{s}}
\newcommand{\frt}{\mathfrak{t}}

\newcommand{\ov}{\overline}

\newcommand{\pants}{\mathbf{\Pi}_{\epsilon, R}}
\newcommand{\tpants}{\mathbf{\wt{\Pi}}_{\epsilon, R}}

\newcommand{\curves}{\mathbf{\Gamma}_{\epsilon, R}}

\newcommand{\bC}{\mathbf{C}}

\newcommand{\bH}{\mathbf{H}}

\newcommand{\bL}{\mathbf{L}}

\newcommand{\bP}{\mathbf{P}}
\newcommand{\bPi}{\mathbf{\Pi}}

\newcommand{\bR}{\mathbf{R}}

\newcommand{\bU}{\mathbf{U}}

\newcommand{\bZ}{\mathbf{Z}}

\newcommand{\ba}{\mathbf{a}}
\newcommand{\bb}{\mathbf{b}}

\newcommand{\bd}{\mathbf{d}}

\newcommand{\bh}{\mathbf{h}}

\newcommand{\bl}{\mathbf{l}}

\newcommand{\bo}{\mathbf{o}}

\newcommand{\br}{\mathbf{r}}

\newcommand{\bt}{\mathbf{t}}

\newcommand{\sF}{\mathscr{F}}
\newcommand{\sG}{\mathscr{G}}

\newcommand{\sM}{\mathscr{M}}

\newcommand{\sS}{\mathscr{S}}
\newcommand{\sT}{\mathscr{T}}

\providecommand{\arr}{\longrightarrow}
\providecommand{\seq}{\subseteq}

\newtheorem{thm}{Theorem}[section]
\newtheorem{cor}[thm]{Corollary}
\newtheorem{prop}[thm]{Proposition}
\newtheorem{lem}[thm]{Lemma}

\newtheorem{claim}[thm]{Claim}
\newtheorem*{claim*}{Claim}
\newtheorem*{thm*}{Theorem}
\newtheorem*{appl*}{Application}
\newtheorem*{prop*}{Proposition}
\newtheorem*{lem*}{Lemma}
\newtheorem*{cor*}{Corollary}
\newtheorem*{conj*}{Conjecture}

\theoremstyle{definition}

\newtheorem*{ques*}{Question}
\newtheorem*{exmp*}{Example}
\newtheorem*{defn*}{Definition}

\def\XXint#1#2#3{{\setbox0=\hbox{$#1{#2#3}{\int}$ }
\vcenter{\hbox{$#2#3$ }}\kern-.6\wd0}}

\theoremstyle{remark}

\newtheorem*{rem*}{Remark}

\makeatletter
\let\c@equation\c@thm
\makeatother
\numberwithin{thm}{section}
\numberwithin{equation}{section}

\title{limits of asymptotically Fuchsian surfaces in a closed hyperbolic 3-manifold}
\author{Fernando Al Assal}
\address{University of Wisconsin-Madison, 480 Lincoln Drive Madison, WI 53703}
\email{alassal@wisc.edu}

\begin{document}

\begin{abstract}
Let $M$ be a closed hyperbolic 3-manifold. Let $\nu_{\Gr M}$ denote the probability volume (Haar) measure of the 2-plane Grassmann bundle $\Gr M$ of $M$ and let $\nu_T$ denote the area measure on $\Gr M$ of an immersed closed totally geodesic surface $T\se M$. We say a sequence of $\pi_1$-injective maps $f_i:S_i\to M$ of surfaces $S_i$ is \emph{asymptotically Fuchsian} if $f_i$ is $K_i$-quasifuchsian with $K_i\to 1$ as $i\to \infty$. We show that the set of weak-* limits of the probability area measures induced on $\Gr M$ by asymptotically Fuchsian minimal or pleated maps $f_i:S_i\to M$ of closed connected surfaces $S_i$ consists of all convex combinations of $\nu_{\Gr M}$ and the $\nu_T$.
\end{abstract}

\maketitle

\section{Introduction}
Let $M = \Gamma\backslash \bH^3$ be a closed hyperbolic 3-manifold, where $\Gamma\leq \PSL_2\bC$ is a cocompact lattice. We say a sequence of $\pi_1$-injective (essential) maps $f_i:S_i\to M$ of surfaces $S_i$ is \emph{asymptotically Fuchsian} if $f_i$ is $K_i$-quasifuchsian with $K_i\to 1$ as $i\to\infty$. For an almost-everywhere differentiable map $f:S\to M$ of a surface into $M$, we let $\nu(f)$ denote the probability area measure induced by $f$ on the oriented 2-plane Grassmann bundle $\Gr M$ of $M$. (Precisely, if we let $\ov{f}:S\to \Gr M$ be given by $\ov{f}(p) = (f(p),T_{f(p)} f(S))$, then $\nu(f)$ is the pushforward via $\ov{f}$ of the pullback via $f$ of the volume measure of $M$, normalized to have mass 1.) We let $\sG$ denote the set of commensurability classes of immersed closed totally geodesic surfaces in $M$. For $T\in \sG$, we let $\nu_T$ denote the area measure of $T$ on $\Gr M$. We let $\nu_{\Gr M}$ denote the probability volume (Haar) measure of $\Gr M$. The main theorem of the article is

\begin{thm}\label{main}
The set of weak-* limits of $\nu(f_i)$, where $f_i:S_i\to  M$ are asymptotically Fuchsian minimal or pleated maps of closed connected surfaces, consists of all measures of the form
\[\tag{$\bullet$}
\nu= \alpha_M\nu_{\Gr M} + \sum_{T\in\sG} \alpha_T \nu_T
\]
where $\alpha_M + \sum_{T\in\sG} \alpha_T = 1$.
\end{thm}

In particular, for any measure $\nu$ of the form $(\bullet)$,
we construct asymptotically Fuchsian minimal or pleated connected surfaces $f_i:S_i\to M$ so $\nu(f_i)$ weak-* converges to $\nu$.

An important part of the proof of Theorem 1.1 is showing that the weak-* limits of convergent subsequences of $\nu(f_i)$ do not depend on whether $f_i$ is minimal or pleated, or in particular on the choice of pleated map. This is despite the fact that, in the pleated case, the universal covers of $f_i(S_i)$ typically do not converge to a geodesic plane in the $C^1$ sense.

\begin{thm}\label{samelimit_intro}
Suppose $f_i : S_i \to M$ are essential asymptotically Fuchsian maps of a closed connected surface. Let $f_i^p$ and $f_i^m$ be, respectively, pleated and minimal maps homotopic to $f_i$. Then, the probability area measures $\nu(f_i^p)$ and $\nu(f_i^m)$ have the same weak-* limit along any convergent subsequence.
\end{thm}

\begin{figure}
\includegraphics[scale=0.13]{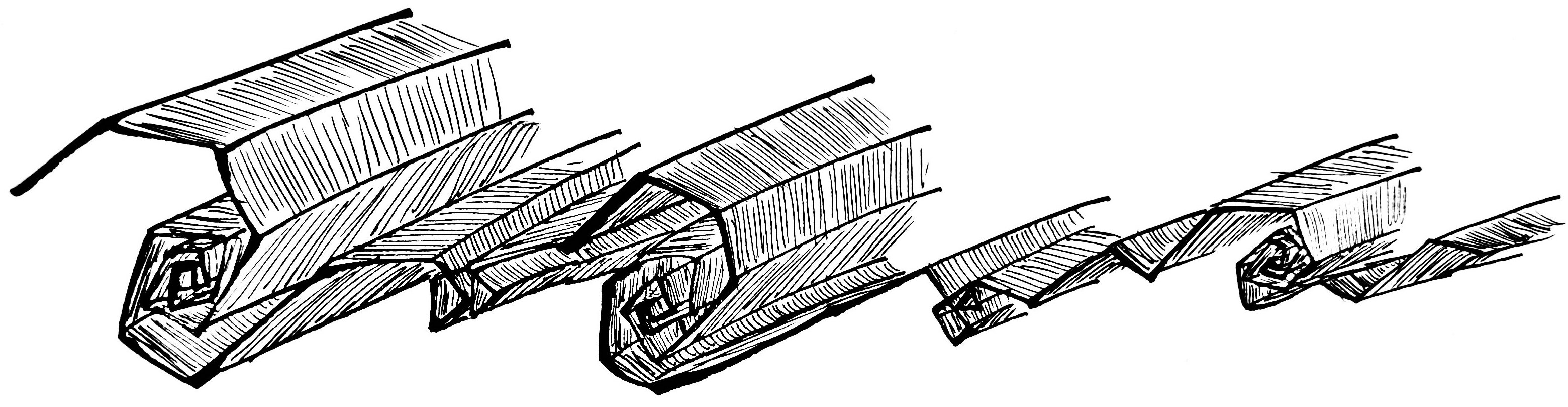}
\caption{The universal covers of asymptotically Fuchsian pleated surfaces are not necessarily embedded in $\bH^3$ and may develop wrinkles as above, so they are never $C^1$-close to a totally geodesic plane}
\label{introwrinkle}
\end{figure}

Theorem \ref{main} is in contrast with the case in which the maps $f_i:S_i\to M$ are all Fuchsian and the $S_i$ are all distinct. Then, the surfaces $f_i(S_i)$ equidistribute in $\Gr M$, namely

\begin{thm*}[Mozes-Shah]
$\nu(f_i) \wkstar \nu_{\Gr M}$ as $i\to \infty$.
\end{thm*}

This follows from a more general theorem of Mozes and Shah \cite{SM} on unipotent dynamics, which builds on work of Dani, Margulis and Ratner. A special case of the main theorem in \cite{SM} is that a sequence of infinitely many distinct orbit closures of the unipotent flow in $\Gr M$ equidistributes. Due to Ratner \cite{R}, these orbit closures are either totally geodesic surfaces or all of $\Gr M$.

%It should be noted that the theorems of Shah, Mozes and Ratner are phrased for measures on the frame bundle $\Fr M$ of $M$, which may be identified with $\PSL_2
%\bC$. For us, the Grassmann bundle is a more natural setting to work with, as the measures $\hat{\nu}(f_i)$ induced by the $f_i$ on $\Fr M$ are invariant by the right action of the copy of $\SO_2$ that preserves normal bundle of the surfaces $f_i(S_i)$. In particular, Theorems 1.1 and 1.2 are valid for $\hat{\nu}(f_i)$ as well.

More recently, Margulis-Mohammadi \cite{MM} and Bader-Fisher-Miller-Stover \cite{BFMS} showed that if $M$ contains infinitely many distinct (noncommensurable) totally geodesic surfaces, then $M$ is arithmetic. (On the other hand, it was already known, due to Reid \cite{Re} and Maclachlan-Reid \cite{MR} that if $M$ is arithmetic, then it contains either zero or infinitely many distinct totally geodesic surfaces.) This rigid behavior of totally geodesic surfaces, however, is not shared by the nearly Fuchsian surfaces of $M$. Due to the surface subgroup theorem of Kahn and Marković \cite{KM}, any closed hyperbolic 3-manifold $M$ has infinitely many commensurability classes of $K$-quasifuchsian surfaces, for any $K>1$.

The Kahn-Marković construction of surface subgroups has a probabilistic flavor. The building blocks from which the nearly Fuchsian surfaces are assembled are the \emph{$(\eps,R)$}-good pants, which are the maps $f:P\to M$ from a pair of pants $P$ taking the cuffs of $P$ to $(\eps,R)$-good curves in $M$ -- the closed geodesics with complex translation length $2\eps$-close to $2R$.  We say two $(\eps,R)$-good pants $f$ and $g$ are equivalent if $f$ is homotopic to $g\circ \phi$, for some orientation-preserving homeomorphism $\phi:P\to P$. For more detailed and precise definitions, see Section 2.

A crucial reason why this construction works is that the good pants incident to a given good curve $\gamma$ come from a well-distributed set of directions. Precisely, the \emph{feet} of the good pants are well-distributed in the unit normal bundle $\N^1(\gamma)$ of $\gamma$. The feet of a good pants $\pi = f:P\to M$ are the derivatives of the unit speed geodesic segments connecting a cuff of $f(P)$ to another, meeting both cuffs orthogonally. Each cuff has two feet, and it turns out that they define the same point, the \emph{foot}, denoted $\ft(\pi)$, in the quotient $\N^1(\sqrt{\gamma})$ of $\N^1(\gamma)$ by $n\mapsto n + \bh\bl(\gamma)$, where $\bh\bl(\gamma)$ is half of the translation length of $\gamma$.

The precise statement of the equidistribution of the feet follows below, from the article of Kahn and Wright \cite{KW} with proof in the supplement \cite{KW2}. In \cite{KW} Kahn and Wright extend the surface subgroup theorem to the case where $M$ has finite volume, while simplifying some elements of the original proof of Kahn-Marković. The proof of the well-distribution of feet in \cite{KW2} follows a different approach than the original Kahn-Marković argument in \cite{KM}. In the latter, the pants are constructed by flowing tripods via the frame flow. In the former, pants with a given cuff are constructed from geodesic segments meeting the cuff orthogonally (the \emph{orthogeodesic connections}). Denote the space of $(\eps,R)$-good curves in $M$ as $\curves$ and the space of $(\eps,R)$-good pants having $\gamma$ as a cuff as $\pants(\gamma)$.

\begin{thm}[Kahn-Wright: Equidistribution of feet]\label{kw}
There is $q=q(M)>0$ so that if $\eps>0$ is small enough and $R > R_0(\eps)$, the following holds. Let $\gamma\in \curves$. If $B\se \N^1 (\sqrt{\gamma})$, then
\[
(1-\delta) \lambda (N_{-\delta} B) \leq
\frac{\#\{ \pi \in \pants(\gamma) : \ft_{\gamma} \pi \in B\}}
{C_{\eps,R,\gamma}}
\leq 
(1+\delta) \lambda(N_{\delta} B),
\]
where $\lambda=\lambda_{\gamma}$ is the probability Lebesgue measure on $\N^1(\sqrt{\gamma})$, $\delta = e^{-qR}$, $N_{\delta}(B)$ is the $\delta$-neighborhood of $B$, $N_{-\delta}(B)$ is the complement of $N_{\delta} (\N^1(\sqrt{\gamma}) - B)$ and $C_{\eps,R,\gamma}$ is a constant depending only on $\eps$, $R$ and $\bl(\gamma)$.
\end{thm}
This theorem will be used in many ways in the article. It implies that a nearly geodesic surface $S(\eps,R)$ may be built using one representative of each equivalence class of $(\eps,R)$-good pants. The equidistribution of feet (in a slight generalization explained in Section 5) will also be used to show that these surfaces equidistribute in $\Gr M$ as $\eps\to 0$. It will also be important in the construction of \emph{non-equidistributing} asymptotically Fuchsian surfaces. 

The surface $S(\eps,R)$ built out of a representative of each equivalence class may not be connected, however. And we do need, for our main theorem, a \emph{connected} surface that goes through every good pants, meeting every cuff in a well-distributed set of directions. This can be achieved using the work of Liu and Marković \cite{LM}. Using their ideas, we can reglue the pants used to build $N = N(\eps,R)$ copies of $S(\eps,R)$ and obtain a connected surface that goes through every cuff in many directions.

%\subsection*{Applications of the nonequidistributing surfaces}

%In \cite{Lo}, Lowe constructed a one parameter family $g_t$ of negatively curved metrics for $M$, with $t\in [0,T)$ for $T\in (0,1]\cup \infty$ so that $\Gr M$ admits a foliation $\sF_t$ whose leaves are immersed minimal surfaces in $(M,g_t)$. The author showed (Theorem 5.2), that there are area-minimizing leaves $S_i$ of this foliation that equidistribute in $\Gr M$ as $i\to \infty$. The same reasoning, together with our result that there are families of increasingly Fuchsian surfaces that accumulate on totally geodesic surface, produces a family $S_i$ of area-minimizing minimal surfaces in $(M,g_t)$ that accumulates in any stable minimal surface of $(M,g_t)$. In a similar note, Song and Zhou produced in \cite{SZ} a sequence of higher-index minimal hypersurfaces that concentrates along a stable minimal hypersurface $S$ for a $C^{\infty}$-generic metric on a closed manifold $M^n$ of dimension $3\leq n \leq 7$.

\subsection*{Further directions} 

One can ask the same questions for finite-volume hyperbolic 3-manifolds $M$. Crucially, Kahn and Wright \cite{KW} extended the surface subgroup theorem of Kahn and Marković to this context by building a nearly Fuchsian surface in $M$ from the good pants that do not go too far into the cusps as well as new building blocks called umbrellas. To execute our construction, we would need to reglue these pieces in order to get a sequence of \emph{connected} closed asymptotically Fuchsian surfaces. This would likely use the work of Sun \cite{Sun} that generalizes ideas of Liu and Marković from \cite{LM} to finite-volume 3-manifolds. Another difference in this setting is that it is not clear whether the mass of a sequence of asymptotically Fuchsian minimal or pleated surfaces may escape to the cusps or not.

Another direction is to extend these results to other homogeneous spaces $\Gamma\backslash G$, where $G$ is a semisimple Lie group and $\Gamma < G$ a cocompact lattice. It has been shown that $\Gamma$ has many surface subgroups, in the style of the Kahn-Marković theorem, when $\Gamma$ is a uniform lattice in a rank one simple Lie group of noncompact type distinct from $\SO_{2m,1}$ by Hämenstadt \cite{H} and when $\Gamma$ is a uniform lattice in a center-free complex semisimple Lie group by Kahn, Labourie and Mozes \cite{KLM}. In the latter article, the authors show that their surface groups are \emph{$\zeta$-Sullivan} for any $\zeta>0$, which is a generalization of $K$-quasifuchsian for the higher rank setting. Again, it would be necessary to extend ideas of Liu and Marković from \cite{LM} to those contexts. Moreover, when $G$ is a higher rank Lie group, there are more $\SL_2\bR$-invariant measures (depending on a choice of representation $\SL_2\bR\to G$) in the homogeneous space $\Gamma\backslash G$ in which a sequence of asymptotically Fuchsian (or $\zeta$-Sullivan with $\zeta\to 0$) could perhaps limit to.

\subsection*{Outline} 
The large-scale structure of the article is the following. In Section 2, we show that that the weak-* limits of the probability area measures of asymptotically Fuchsian surfaces in $M$ is a convex combination of the volume measure of $\Gr M$ and the area measures supported on closed geodesic surfaces. This is one of the directions of Theorem \ref{main}. The other direction of this equality will be proved in Sections 3, 4, 5, 6 and 7. Details follow below.

In Section 2, we prove Theorem \ref{samelimit_intro}. Namely, we describe how nearly Fuchsian surfaces may be realized geometrically inside $M$ as pleated or as minimal surfaces. We argue that, as these surfaces $f_i: S_i\to M$ become closer to Fuchsian, the weak-* limits of their area measures in $\Gr M$ do not depend on the choice of geometric structure. We do this by mapping the universal covers of our surfaces to a component of the convex core of $(f_i)_*(\pi_1(S_i))$ via the normal flow, and arguing that this map has small derivatives in most of its domain. This is despite the fact that the universal covers of the pleated surfaces do not converge to a geodesic disc in the $C^1$ sense. For the case of minimal surfaces, we use the fact that their principal curvatures are uniformly small, as shown by Seppi \cite{Se}. 

Using a theorem of Lowe for minimal surfaces \cite{Lo}, we conclude that the limiting measures are $\PSL_2 \bR$-invariant. Thus, due to the Ratner measure classification, they are a convex combination of the volume measure of $\Gr M$ and area measures of the totally geodesic surfaces of $M$. This shows one direction of Theorem \ref{main}.

In Section 3, we explain how to construct nearly geodesic closed essential surfaces in $M$, following Kahn, Marković and Wright (\cite{KM} and \cite{KW}). We define their building blocks, the $(\eps,R)$-good pants, and the correct ($(\eps,R)$-\emph{good}) way to glue them together so the result is nearly Fuchsian. Finally, we explain how to use the equidistribution of feet (Theorem \ref{kw}), together with the Hall marriage theorem from combinatorics, to show that a copy of each good pants may be glued via good gluings to form a closed surface $S(\eps,R)$.

In Section 4, we follow ideas of Liu and Marković \cite{LM} to explain how to reassemble a \emph{connected} nearly Fuchsian closed essential surface $\hat{S}(\eps,R)$ from $N = N(\eps,R)$ copies of the surface built in Section 2. In particular, this connected surface defines the same area measure $\nu(\eps,R)$ in $\Gr (M)$ as $S(\eps,R)$.

In Section 5, we endow the connected surface built out of the same number of copies of each good pants $\hat{S}(\eps,R)$ with the pleated structure in which every good pants is glued from two ideal triangles. We show that the barycenters of these triangles equidistribute in the frame bundle $\Fr M$ of $M$ as the surfaces become more Fuchsian (namely, as $\eps\to 0$ and $R(\eps)\to\infty$). To do so, we use a generalization of the equidistribution of feet (Theorem \ref{kw}), in which a continuous function $g\in C(\N^1 (\sqrt{\gamma}))$ plays the role of the set $B$ in the statement above. We also use the fact, from a formulation due to Lalley \cite{L} (building on the work of Bowen \cite{B}), that asymptotically almost surely, the cuffs of the pants equidistribute in the unit tangent bundle $\T^1 M$.

In Section 6, from the equidistribution of the barycenters of the triangles, we conclude that the surfaces $\hat{S}(\eps,R)$ built from the triangles equidistribute as $\eps\to 0$ and $R(\eps)\to \infty$. This is because each triangle can be obtained from the right action of a subset $\Delta\se \PSL_2\bR$ on the barycenter. The approach we take in Sections 5 and 6 is similar to the one used by Labourie \cite{La} to show that certain perhaps disconnected asymptotically Fuchsian surfaces equidistribute in $M$. A difference is that the surfaces in \cite{La} are built from a different multiset of good pants that comes from the original Kahn-Marković construction. It is not clear, for example, how many copies of each pants are used to build those asymptotically Fuchsian surfaces.

In Section 7, we build a family of nearly Fuchsian surfaces by gluing the equidistributing surfaces $\hat{S}(\eps,R)$ of Sections 5 and 6 to high degree covers of totally geodesic surfaces in $M$. To do so, we need the fact that a high degree cover of the totally geodesic surfaces of $M$ may be built from good gluings of good pants. This was shown by Kahn and Marković \cite{KME} in order to prove the Ehrenpreis conjecture. We show that as these hybrid surfaces become asymptotically Fuchsian, they may accumulate on any of the totally geodesic surfaces.

\medskip 
 
\subsection*{Acknowledgements}
I would like to thank Danny Calegari, James Farre, Ben Lowe and Franco Vargas-Pallete for useful discussions and correspondence. I thank Rich Schwartz for suggesting the adverb ``asymptotically'' instead of ``increasingly'' for the title. \iffalse I would also like to thank Natalie Rose Schwartz for help drawing Figure \ref{spun}. \fi I especially thank Jeremy Kahn for suggesting the proof of Theorem \ref{equidftfm} and my advisor Yair Minsky for all the help. Finally, I thank the referee for a careful reading and thoughtful comments.

\bigskip

\tableofcontents

\section{Geometric realizations of nearly Fuchsian surfaces}

Suppose $f:S\to M$ is an essential nearly Fuchsian immersion of a closed connected orientable surface. Then, $f$ is homotopic to maps with interesting geometric properties, namely a unique minimal map and many pleated maps. In this section, we will describe these geometric realizations and show that their area measures in $\Gr M$ have the same limit as they become asymptotically Fuchsian.

Precisely, suppose $f_i: S_i \to M$ are asymptotically Fuchsian maps of closed connected surfaces $S_i$. Let $f_i^p$ and $f_i^m$ be, respectively, pleated and minimal maps homotopic to $f_i$. Let $\nu(f^p_i) = p_i$ and $\nu(f^m_i) = m_i$ be the probability area measures induced by these maps on the 2-plane Grassmann bundle $\Gr M$. The main theorem of this section is the following, which was labeled as Theorem 1.2 in the introduction.

\begin{thm}\label{samelimit}
A subsequence $m_{i_j}$ satisfies $m_{i_j}\wkstar \nu$ as $j\to\infty$ if and only if $p_{i_j}\wkstar \nu$.
\end{thm}

Let $\hat{m}_i = \hat{\nu}(f_i^m)$ be the probability measure induced by $f_i^m$ on the frame bundle $\Fr M$. By the weak-* compactness of the probability measures on $\Fr M$, the $\hat{m}_i$ converge to a measure $\hat{\nu}$ along a subsequence. As shown by Lowe in Proposition 5.2 of \cite{Lo} and, following ideas of Labourie \cite{La}, by Lowe-Neves in Lemma 3.2 of \cite{LN}, the measure $\hat{\nu}$ is invariant under the right action of $\PSL_2 \bR$. Thus, from the Ratner measure classification theorem \cite{R}, it follows that the weak-* subsequential limits of $m_i$ are of the form
\[\tag{$\star$}
\nu = \alpha_M \nu_{\Gr M} + \sum_{T\in\sG} \alpha_T \nu_T.
\]
As before, $\sG$ is a set containing a representative of each commensurability class of closed immersed totally geodesic surfaces in $M$, $\nu_{\Gr M}$ is the probability Haar measure on $\Gr M$, and $\nu_T$ is the probability area measure of an immersed closed totally geodesic surface $T\se M$. The coefficients $\alpha_M$ and $\alpha_T$ sum to 1.

This, combined with Theorem \ref{samelimit}, shows one of the directions of the main theorem of the article, Theorem \ref{main}. In Sections 5, 6 and 7, we will show that given any $\nu$ of the form ($\star$), we may find asymptotically Fuchsian connected closed surfaces in $M$ with limiting measure $\nu$.

%\begin{figure}
%\includegraphics[scale=0.1]{wrinkle.jpg}
%\caption{Asymptotically Fuchsian pleated surfaces in $\bH^3$ are not necessarily embedded and develop wrinkles so they are never $C^1$-close to a totally geodesic disc}
%\label{wrinke}
%\end{figure}

Let $H_i^+$ be a (the top) component of the boundary of the convex core of the quasifuchsian group $Q_i = (f_i)_*(\pi_1 S_i)$.
Let $f_i^h$ be the pleated map homotopic to $f_i$ whose lift to the universal cover maps $\hat{S}_i$ into $H_i^+$ and say $h_i = \nu(f_i^h)$. To prove Theorem 4.1, we show that each of $p_i$ and $m_i$ has the same weak-* subsequential limits as $h_i$.

\begin{thm}\label{sl1}
A subsequence $p_{i_j}$ satisfies $p_{i_j}\wkstar \nu$ as $j\to\infty$ if and only if $h_{i_j}\wkstar \nu$.
\end{thm}

\begin{thm}\label{sl2}
A subsequence $m_{i_j}$ satisfies $m_{i_j}\wkstar \nu$ as $j\to\infty$ if and only if $h_{i_j}\wkstar \nu$.
\end{thm}

Theorems \ref{sl1} and \ref{sl2} are in turn proven by flowing the universal covers $\wt{f}^m_i(\wt{S}_i)$ and $\wt{f}^p_i(\wt{S}_i)$ normally into $H_i^+$. We argue that this process has uniformly small area distortion. In the pleated case of Theorem \ref{sl1}, we need to argue a definite distance $\eta>0$ away from the bending lamination of $\wt{f}^p_i(\wt{S}_i)$ to avoid complicated wrinkles as in  Figure \ref{introwrinkle}. Then, we take $\eta\to 0$. In the minimal case of Theorem \ref{sl2}, we use the result of Seppi \cite{Se} that says that the principal curvatures of $\wt{f}^m_i(\wt{S}_i)$ go uniformly to zero as the quasiconformal constant $K_i$ tends to 1.

\begin{figure}
\includegraphics[scale=0.13]{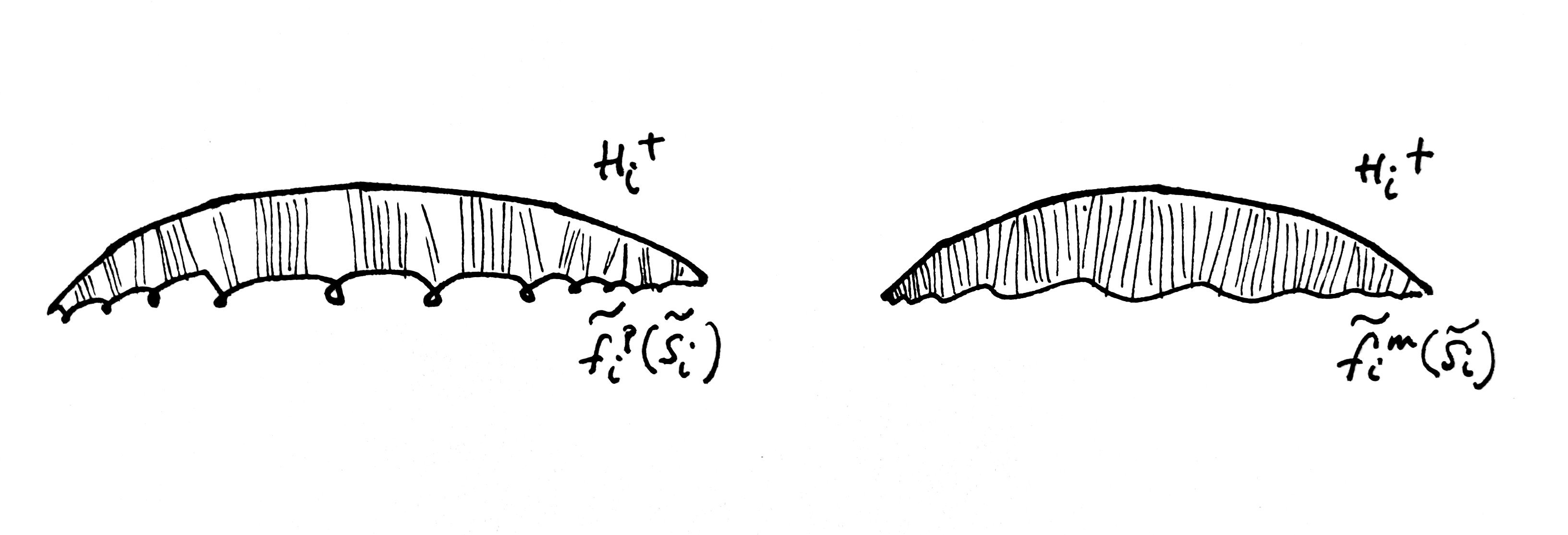}
\caption{\small{Visual outline of the proof of Theorem \ref{samelimit}. We will flow the universal covers $\wt{f}^m_i(\wt{S}_i)$ and $\wt{f}^p_i(\wt{S}_i)$ of the asymptotically Fuchsian minimal and pleated surfaces normally till they hit a component $H_i^+$ of the boundary of the convex core. We will argue this process has a uniformly small area distortion (away from the pleating lamination, in the pleated case).}}
\end{figure}

\subsection{Quasiconformal maps and quasifuchsian groups}
Let $\Omega\se \hat{\bC}$ be a domain. A continuous map $h:\Omega\to \hat{\bC}$ is quasiconformal if its weak derivatives are locally in $L^2(\Omega)$ and it satisfies the Beltrami equation
\[
\partial_z h(z) = \mu(z) \partial_{\bar{z}} h(z)
\]
for almost every $z\in \Omega$ for some $\mu \in L^{\infty} (\Omega)$ with $\|\mu\|_{L^{\infty}(\Omega)} < 1$. The derivatives $\partial_z = (\partial_x - i\partial_y)/2$ and $\partial_{\bar{z}} = (\partial_x + i\partial_y)/2$ are understood in the distributional sense.

We say that $h:\Omega\to\hat{\bC}$ is $K$-quasiconformal if $\mu$, which is called the Beltrami differential of $h$, satisfies
\[
K(h):= \frac{1+\|\mu\|_{\infty}}{1-\|\mu\|_{\infty}} \leq K.
\]

In general, $\mu$ is a Beltrami differential in a domain $\Omega\se\hat{\bC}$ if it is an element of the open unit ball around the origin $B_1(0)$ of $L^{\infty} (\Omega)$.
The measurable Riemann mapping theorem says that given a Beltrami differential in $\hat{\bC}$, we may find a unique quasiconformal mapping $h:\hat{\bC}\to\hat{\bC}$ fixing $0$, $1$ and $\infty$ with $\partial_z h = \mu \partial_{\bar{z}}h$.

Quasiconformal maps enjoy the following compactness property that will be useful to us. (It is Lemma 6 on page 21 of \cite{G}.)

\begin{lem*}[Compactness]\label{qc}
Let $h_i: \hat{\bC} \to \hat{\bC}$ be a sequence of $K$-quasiconformal maps fixing $0$, $1$ and $\infty$. Then the $h_i$ converge uniformly to $h$ as $i\to\infty$, where $h$ is a $K$-quasiconformal map.
\end{lem*}

It turns out that $1$-quasiconformal maps are  conformal, which is a  regularity theorem for the solutions of the Beltrami equation. Thus, it follows that if the $h_i$ are $K_i$-quasiconformal fixing $0$, $1$ and $\infty$ with $K_i\to 1$ as $i\to \infty$, then they converge uniformly to the identity.

Let $\bU \se \hat{\bC}$ denote the upper half plane, and let $\bL = \hat{\bC} - \bar{\bU}$. We define the universal Teichmüller space of $\bU$ as
\[
\sT (\bU) = \{ h:\hat{\bC}\to\hat{\bC}\text{ quasiconformal fixing 0, 1 and }\infty\,:\, h|_{\bL}\text{ is conformal}\}.
\]
To obtain elements of $\sT(\bU)$, let $\mu$ be a Beltrami differential in $\bU$. We may extend it to a Beltrami differential also denoted $\mu$ in $\hat{\bC}$ by setting $\mu|_{\bL} = 0$. By the measurable Riemann mapping theorem, there is a unique quasiconformal mapping $h$ of $\hat{\bC}$ that fixes 0, 1 and $\infty$ and satisfies $\partial_z h = \mu \partial_{\bar{z}} h$. Moreover, $\partial_{\bar{z}} h= 0$ in $\bL$, so $h|_{\bL}$ is conformal.

A Jordan curve $\Lambda\se \hat{\bC}$ is a $K$-quasicircle if
\[
K = \inf \{ K(h) \,:\,h\in \sT(\bU)\text{ and }\Lambda=h(\partial\bU) \}.
\]
Note that this infimum is achieved: if $h_i$ are elements of $\sT(\bU)$ with $K(h_i) \to K$, then by the compactness lemma, the $h_i$ converge uniformly to a $K$-quasiconformal mapping of $\hat{\bC}$ fixing 0, 1 and $\infty$ with $\Lambda = h(\partial\bU)$.

A group $Q\leq \PSL_2 \bC$ is $K$-\emph{quasifuchsian} if $F=hQ h^{-1}$ is a Fuchsian group for some $K$-quasiconformal map $h:\hat{\bC}\to\hat{\bC}$. Up to conjugating $Q$ by a $g\in \PSL_2 \bC$, we can say that its limit set $\Lambda_Q$ contains $0$, $1$ and $\infty$. Thus, there is a $K$-quasiconformal mapping $h\in \sT(\bU)$ so that $\Lambda_Q = f(\partial\bU)$. In particular, we see that $\Lambda_Q$ is a $K$-\emph{quasicircle} -- the image of a circle under a $K$-quasiconformal map. These are nowhere differentiable Hölder curves.

A continuous, $\pi_1$-injective map $f:S\to M$ of a hyperbolic surface $S$ into a hyperbolic 3-manifold $M$ is $K$-\emph{quasifuchsian} if $f_*(\pi_1 S) \leq \Gamma \cong \pi_1 M \leq \PSL_2\bC$ is a $K$-quasifuchsian group. Given a $K$-quasifuchsian subgroup $Q$ of the Kleinian group $\Gamma \cong \pi_1 M$, we may recover a $K$-quasifuchsian map $f:S\to M$ in the following way. As described above, $Q$ gives rise to a $K$-quasiconformal map $h\in \sT(\bU)$, whose restriction to $\partial \bU \cong \partial_{\infty} \bH^2$ may be extended to a $Q$-equivariant map $\tilde{f} : \bH^2 \to \bH^3$. The map $\tilde{f}$ in turn descends to $f:S\to M$. (We will describe examples of this extensions as minimal or pleated maps in detail below.)

A sequence of maps $f:S_i\to M$ of hyperbolic surfaces $S_i$ into a hyperbolic 3-manifold is \emph{asymptotically Fuchsian} if the $f_i$ are $K_i$-quasifuchsian for $K_i\to 1$ as $i\to\infty$. Given such a sequence, we may find a sequence of $K_i$-quasiconformal maps $h_i\in \sT(\bU)$ that conjugate $Q_i = (f_*)(\pi_1 S_i)$ into $\PSL_2\bR$. From the compactness theorem of quasiconformal maps, it follows that the $h_i$ converge uniformly to the identity. In particular, the limit sets $\Lambda_{Q_i}$ are sandwiched between two circles at an Euclidean distance going to zero as $i\to\infty$.

\subsection{The Schwarzian derivative and the Bers norm} The \emph{Schwarzian derivative} of a holomorphic function $f$ with nonvanishing derivative is given by
\[
S_f = \lef( \frac{f''}{f'}\ri)' - \frac{1}{2}\lef( \frac{f''}{f'} \ri)^2.
\]
This vanishes precisely at the Möbius transformations and it can be shown that if $f_i$ converges uniformly to a Möbius transformation as $i\to\infty$, then $S_{f_i}\to 0$ as $i\to \infty$.

The \emph{Bers norm} of $f\in\sT(\bU)$ is given by
\[
\|f\|_B := \sup_{z\in\bL} |S_f(z)| \rho^2(z),
\]
where $\rho$ is the Poincaré metric of curvature -1 on $\bL$. As the quasiconformal constant of $f$ goes to 1, $f$ converges uniformly to the identity on $\hat{\bC}$, and so $\|f\|_B \to 0$.

\subsection{Pleated surfaces and the convex core}

A $\pi_1$-injective isometric map $f:S\to M$ of a surface $S$ is \emph{pleated} or \emph{uncrumpled} if every $p\in S$ is inside a geodesic arc of $S$ that is mapped to a geodesic arc of $M$. It turns out (see Proposition 8.8.2 of \cite{Th}) that the set $\lambda \se S$ of points that lie in a single geodesic segment that gets mapped to a geodesic is a lamination on $S$, and that $f$ is totally geodesic outside $\lambda$. The lamination $\lambda$ is called the \emph{pleating} or \emph{bending} lamination.

A $K$-quasifuchsian map $f:S\to M$ is homotopic to many pleated surfaces -- given any geodesic lamination $\lambda \se S$, it is possible to find a pleated map homotopic to $f$ whose pleating locus is $\lambda$. One such pleated map of note comes from the boundary of the convex core of the quasifuchsian group $Q = f_* (\pi_1 S)$. Let $\Lambda$ be the limit set of $Q$. The convex core of $Q$ is the smallest set $\core Q\se \bH^3$ containing the geodesics with endpoints in $\Lambda$. Thurston showed that its boundary $\partial \core Q$ has two components $H^-$ and $H^+$ that are the image of $\bH^2$ under a $Q$-equivariant pleated map \cite{EM}. In particular, $f:S\to M$ is homotopic to a pleated map $f^h:S\to M$ so that $\wt{f^h}(\tilde{S}) = H^+$.

The pleated discs $H^-$ and $H^+$ inherit an orientation from $f$, and in particular normal vector fields $n^-$ and $n^+$ away from their bending loci. We will follow the convention that $H^-$ is the component so that the trajectory from flowing a vector $n^-$ via the geodesic flow will meet $H^+$ at some positive time.

Another pleated map homotopic to $f:S\to M$ of importance in this article is the one where the bending lamination consists of a pants decomposition of $S$ as well as three spiraling geodesics per pants that divide the pants into two ideal triangles. We will keep track of these triangles to show that the surface built out of one copy of each $(\eps,R)$-good pants equidistributes as $\eps\to 0$ in Section \ref{equid}.

In what follows we will repeatedly use the fact, proved by Birman and Series \cite{BS}, that geodesic laminations on a surface have Lebesgue measure zero. Without loss of generality we will also assume that all our geodesic laminations are \emph{maximal}, i.e., their complement is an union of ideal triangles. (See Theorem I.4.2.8 in \cite{EM} -- it is always possible to make a geodesic lamination maximal by adding a finite number of leaves.)

\subsection{Proving Theorem \ref{sl1}}

We are now ready to restate and prove Theorem \ref{sl1}. Let $f_i:S_i\to M$ be asymptotically Fuchsian maps, with $Q_i = (f_i)_* (\pi_1 S_i)$. Let $H_i^-$ and $H_i^+$ be the components of $\partial\core Q_i$ (again, chosen so flowing normally from $H_i^-$ gets you to $H_i^+$). Let $f_i^p$ and $f_i^h$ be pleated maps homotopic to $f_i$, where $f_i^h$ has a lift to the universal cover $\wt{f_i^h} : \wt{S_i} \to \bH^3$ satisfying $\wt{f_i^h} (\wt{S_i}) = H_i^+$. Let $p_i = \nu(f_i^p)$ and $h_i = \nu(f_i^h)$ be the area measures induced on $\Gr M$ by $f_i^p$ and $f_i^h$, respectively.

\begin{thm*}
A subsequence $p_{i_j}$ satisfies $p_{i_j}\wkstar \nu$ as $j\to\infty$ if and only if $h_{i_j}\wkstar \nu$.
\end{thm*}

Let $\Lambda_i \se \hat{\bC}$ be the limit set of $Q_i$ 
Let $\wt{f_i^p}:\wt{S_i}\to \bH^3$ be the lift of $f_i^p$ to the universal cover so $\partial_{\infty} \wt{f_i^p} (\wt{S_i}) = \Lambda_i$. We define $P_i:= \wt{f_i^p} (\wt{S_i})$. We let $\tilde{p}_i$ and $\til{h}_i$ be, respectively, the area measures induced by $\wt{f_i^p}$ and $\wt{f_i^h}$ on $\Gr \bH^3$. We denote the pleating laminations of $P_i$ and $H_i^+$ by $\lambda_i$ and $\beta_i$, respectively. Finally, we define $\Sigma_i := \Gamma\backslash P_i$ and $R_i := \Gamma\backslash H_i^+$.

We let $n_t: \Gr \bH^3\to \bH^3$ be the map taking $(p,\Pi)\in \Gr \bH^3$ to the point $q\in \bH^3$ obtained by flowing $p$ in the direction normal to $\Pi$ (from the orientation of $P$) for time $t$ via the geodesic flow. We will usually apply $n_t$ to points in $p\in P_i$, so to loosen the notation $n_t(p)$ will be shorthand for $n_t(p,T_p P_i)$.
 
For $\eta>0$, let $P_i^{\eta}$ denote the set of points in $P_i$ that are at a distance greater than $\eta$ from the pleating lamination $\lambda_i$. We define a map
\[
F_i^{\eta} : P_i^{\eta} \arr H_i^+
\]
by flowing $p\in P_i^{\eta}$ normally for the time $\tau_i(p)$ it takes to hit $H_i^+$. In other words, $F_i^{\eta}(p) = n_{\tau_i(p)}(p)$.

We also let $\det(dF_i^{\eta})$ be the Radon-Nikodym derivative
\[
\det (dF_i^{\eta}) := \frac{d (F_i^{\eta})^* \til{h}_i}{d \til{p}_i},
\]
which is defined due to parts i and ii of the Proposition \ref{map} below.

\begin{figure}
\includegraphics[scale=0.14]{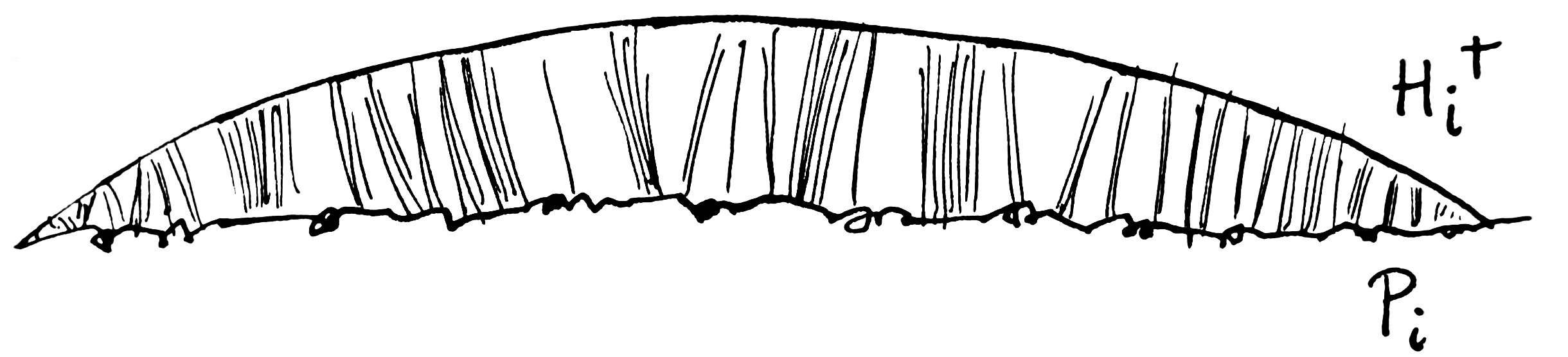}
\caption{A visualization of the map $F^{\eta}_i$, flowing normally from $P_i^{\eta}$ till $H_i^+$. Lemma \ref{box} below shows that these lines indeed do not meet for $i$ large enough.}
\end{figure}

\begin{prop}\label{map}
For $i\geq I_0(\eta)$, these maps $F_i^{\eta}$ satisfy
\begin{enumerate}[i.]
\item $F_i^{\eta}$ is differentiable outside of $(F^{\eta}_i)^{-1}(\beta_i) \cup \lambda_i,$
\item $\til{p}_i\lef( (F^{\eta}_i)^{-1} (\beta_i) \ri) = 0,$
\item $\|\det(dF^{\eta}_i) - 1\|_{L^{\infty} (P^{\eta}_i)} \to 0 \text{ as }i\to\infty$.
\end{enumerate}
\end{prop}

\begin{proof}
(We will drop the superscript $\eta$ when convenient and unambiguous.)

{\bf i.} Let $p\in P^{\eta}_i - (F_i^{-1}(\beta_i)\cup\lambda_i)$. Then, $F_i$ maps a small disc around $p$ to a piece of a totally geodesic plane in $H^+_i$ via the normal flow. This is a differentiable map.

\medskip

{\bf ii.} Let $H_i^{\eta}:= F_i^{\eta}(P_i^{\eta})$. We will prove the statement by showing that the inverse map
\[
(F_i^{\eta})^{-1} : H_i^{\eta} \to P_i^{\eta}
\]
is well defined and sends sets of measure zero (such as the bending lamination $\beta_i$) to sets of measure zero.

For a parameter $t(\eta)>0$ to be picked later, we define $E_i^{\eta}$ to be the three-dimensional submanifold of $\bH^3$ obtained by flowing $P_i^{\eta}$ normally for times $s\in [0,t(\eta)]$. Precisely, $E_i^{\eta} := \bigcup_{s\in [0,t(\eta)]} n_s (P^{\eta}_i)$.

As the pleating lamination $\lambda_i$ is maximal, $P_i^{\eta}$ is a disjoint union of triangles. Thus, the set $E_i^{\eta}$ is a union of thickened triangles, possibly not disjoint. Lemma \ref{box} below shows that for $i$ sufficiently large depending on $\eta$, we can choose $t(\eta)$ so $E_i^{\eta}$ is in fact a disjoint union of thickened triangles. (See Figure 4.)

For an ideal triangle $T\se \bH^3$, we will denote by $T^{\eta}$ the set of $p\in T$ at a distance greater than $\eta$ from the edges of $T$.

\begin{lem}\label{box}
There exists $I_0(\eta)$ satisfying the following. Supose $i\geq I_0(\eta)$. Then, there is $t(\eta)>0$ so that for any two ideal triangles $S$ and $T$ in $P_i$, we have
\[
n_{s_1} (S^{\eta}) \cap n_{s_2} (T^{\eta}) = \emptyset
\]
for all $0\leq s_1, s_2 \leq t(\eta)$.
\end{lem}

\begin{figure}\label{boxfig}
\includegraphics[scale=0.12]{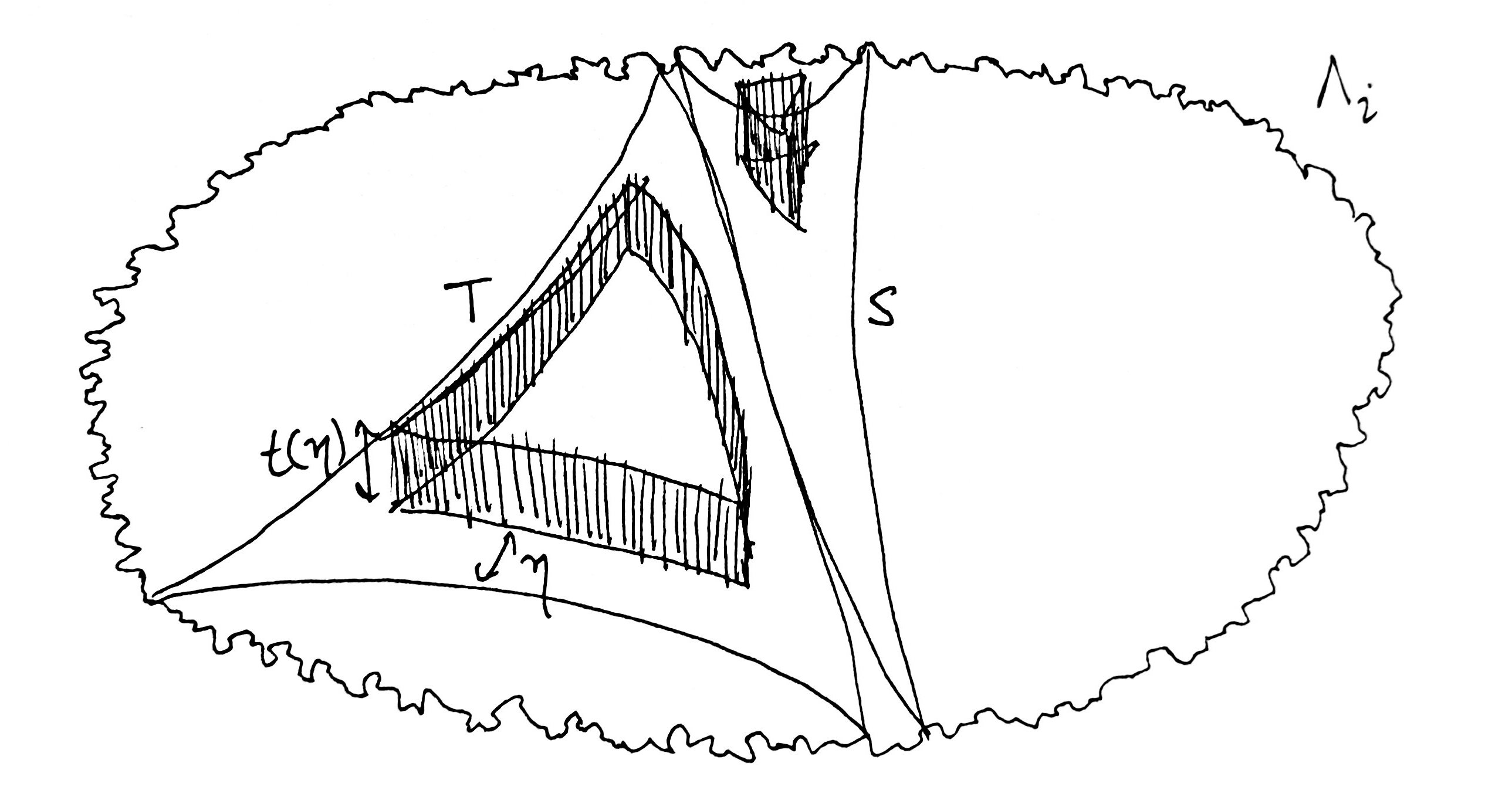}
\caption{Lemma \ref{box} says that the boxes made out of flowing $S^{\eta}$ and $T^{\eta}$ for time $t(\eta)$ never meet.}
\end{figure}

\begin{proof}
Without loss of generality, up to conjugating everything by Möbius transformations, we may take $S= \Delta$.

Recall that $\wt{f^p_i}: \bH^2 \arr \bH^3$ is the pleated map so that $\wt{f^p_i} (\bH^2) = P_i$. We know that $\partial_\infty p_i$ is the $K_i$-quasiconformal homeomorphism $h_i:\hat{\bC} \to \hat{\bC}$ fixing $0$, $1$ and $\infty$ so that $h_i(\hat{\bR}) = \Lambda_i$ (where $K_i\to 1$ as $i\to\infty$). In particular, $\wt{f^p_i}$ is the identity on $\Delta$. Moreover, as discussed previously, $h_i$ converges uniformly to the identity map as $i\to\infty.$ Denote this modulus of uniform convergence as $\omega_i$.

Define the following closed intervals in $\hat{\bR}\se \hat{\bC}$:
\[
I^1 = [0,1],\quad I^2 = [1,\infty]\aand I^3 = [\infty,0].
\]

Let $T$ be a triangle arising as a component of $P_i - \lambda_i$, distinct from $\Delta$. Then, $\wt{f_i^p}^{-1} (T)$ and $\wt{f_i^p}^{-1} (\Delta) = \Delta$ are triangles in the ideal triangulation $\wt{f_i^p}^{-1}(\lambda_i)$ of $\bH^2$. In particular, they do not intersect, so the vertices of $\wt{f_i^p}^{-1}(T)$ all lie in $I^{\ell}$ for some $\ell\in \{1,2,3\}$. Thus, as $f_i$ is uniformly $\omega_i$-close to the identity, the vertices of $T$ are contained in $N_{\omega_i}(I^{\ell})$, the $\omega_i$-neighborhood of $I^{\ell}$ in $\hat{\bC}$.

As the vertices of $T$ are trapped in a shrinking neighborhood of $I^{\ell}$, we have that
\[
T^{\eta} \cap N_{\eta}(\Delta^{\eta}) = \emptyset,
\]
which in turn implies that
\[
\sup_{p\in T^{\eta}} \dist_{\bH^3} (p,\Delta^{\eta}) \geq \eta.
\]
It follows that $n_{s_1} (\Delta^{\eta})\cap n_{s_2} (T^{\eta}) = \emptyset$ for $0\leq s_1,s_2 \leq \eta/2$. We conclude the theorem holds for $t(\eta) = \eta/2$.
\end{proof}

For $i$ sufficiently large depending on $\eta$, the lemma allows us to define a map
\[
G_i : E^{\eta}_i \arr P^{\eta}_i
\]
which takes $q\in E^{\eta}_i$ back to the unique $p\in P^{\eta}_i$ so that $n_t(p) = q$. This is well defined as the components of $E_i^{\eta}$ given by normal flow starting at some triangle of $P^{\eta}_i$ never intersect. We are also using the fact that the normal flow is injective when restricted to a geodesic plane. Moreover, $G_i$ is smooth.

As we will explain in the beginning of the proof of Lemma \ref{c1} below, it is possible to take $i$ large enough so $H^{\eta}_i$ lies inside $E^{\eta}_i$. We can thus restrict $G_i$ to $H^{\eta}_i$ and obtain the inverse of $F_i^{\eta}$. This restriction is Lipschitz and hence sends sets of measure zero to sets of measure zero.

\medskip

{\bf iii.} It suffices to show that $\det (dF^{\eta}_i)$ converges uniformly to 1 in the fixed triangle $\Delta^{\eta}$, namely

\begin{prop}\label{derivdelta}
$\|\det(dF_i) - 1 \|_{L^{\infty}(\Delta^{\eta})} \to 0$ as $i\to\infty$.
\end{prop}

Indeed, if $T_{i_j} \se P_i$ is a sequence of triangles, there are Möbius transformations $f_{i_j}$ so that $f_{i_j} T_{i_j} (f_{i_j})^{-1} = \Delta$ while $f_{i_j} \Lambda_i f_{i_j}^{-1}$ is still trapped in a $\delta(i)$ neighborhood of $\bR$, where $\delta(i) = o_i(1)$ and does not depend on $j$. Therefore, conjugating by $f_{i_j}$ does not affect the following analysis and in particular $\det dF_i$ being uniformly close to 1 in $\Delta^{\eta}$ implies $\det dF_i$ is uniformly close to 1 in all of $P_i^{\eta}$.

Recall that $\tau_i$ is the time it takes for a point $p\in \Delta^{\eta}$ to hit $H_i^+$ via the normal flow, i.e.,
\[
\tau_i(x) = \inf \{t>0 : n_t(x) \in H_i^+\}.
\]
In order to prove Proposition \ref{derivdelta}, we will need the following

\begin{lem}\label{c1}
$\|\tau_i\|_{C^1(\Delta^{\eta})} \to 0$ as $i\to\infty$.
\end{lem}

\begin{proof}
We begin by showing that 

\begin{equation}\label{cpt} \|\tau_i\|_{L^{\infty}(\Delta^{\eta})} \to 0 \text{ as }i\to\infty.\end{equation}

Recall that $\Lambda_i= f_i (\bR)$, where the $f_i$ are $K_i$-quasiconformal maps with $K_i\to 1$ that converge uniformly to the identity on $\hat{\bC}$. In particular, we can find a function $\delta(i)\to 0$ with $i\to\infty$ so that $\Lambda_i \se N_{\delta(i)}(\bR)$. Let $\Pi_i^+$ and $\Pi_i^-$ be the totally geodesic planes satisfying
\[
\partial_{\infty} \Pi_i^+ \cup \partial_{\infty} \Pi_i^- = \partial N_{\delta(i)} (\bR),
\]
with $\Pi_i^+$ in the same side of the plane containing $\Delta$ as $H_i^+$. Let $T_i(x)$ be the time it takes for a point $x\in \Delta^{\eta}$ to hit $\Pi_i^+$ via the normal flow, i.e.,
\[
T_i(x) = \inf \{t>0 : n_t(x) \in \Pi_i^+\}.
\]
By construction, $\tau_i(x) \leq T_i(x)$ for $x\in \Delta^{\eta}$. In addition, as $\delta (i) \to 0$, we also have that $T_i(x) \to 0$ uniformly in $x$, with $i\to\infty$. This shows \ref{cpt}.

Let $p\in \Delta^{\eta}$ be a point outside of $F_i^{-1}(\beta_i)$ and let $v\in T_p \Delta^{\eta}$ be a unit vector. Now, we show that
\begin{equation}\label{cptder} d\tau_i(p)(v) \iinfty 0 \end{equation}
uniformly in $(p,v)\in \T^1 \Delta^{\eta}.$

Let $\theta_i(p,v)$
be the angle in $(0,\pi/2]$ that the geodesic normal to $\Delta$ through $p$ makes with the curve $s\mapsto F_i(\exp_p sv)$, for $s\geq 0$. It suffices to show that this angle is uniformly close to $\pi/2$ as $(p,v)$ varies in $\T^1 \Delta^{\eta}$.

\begin{figure}
\includegraphics[scale=0.08]{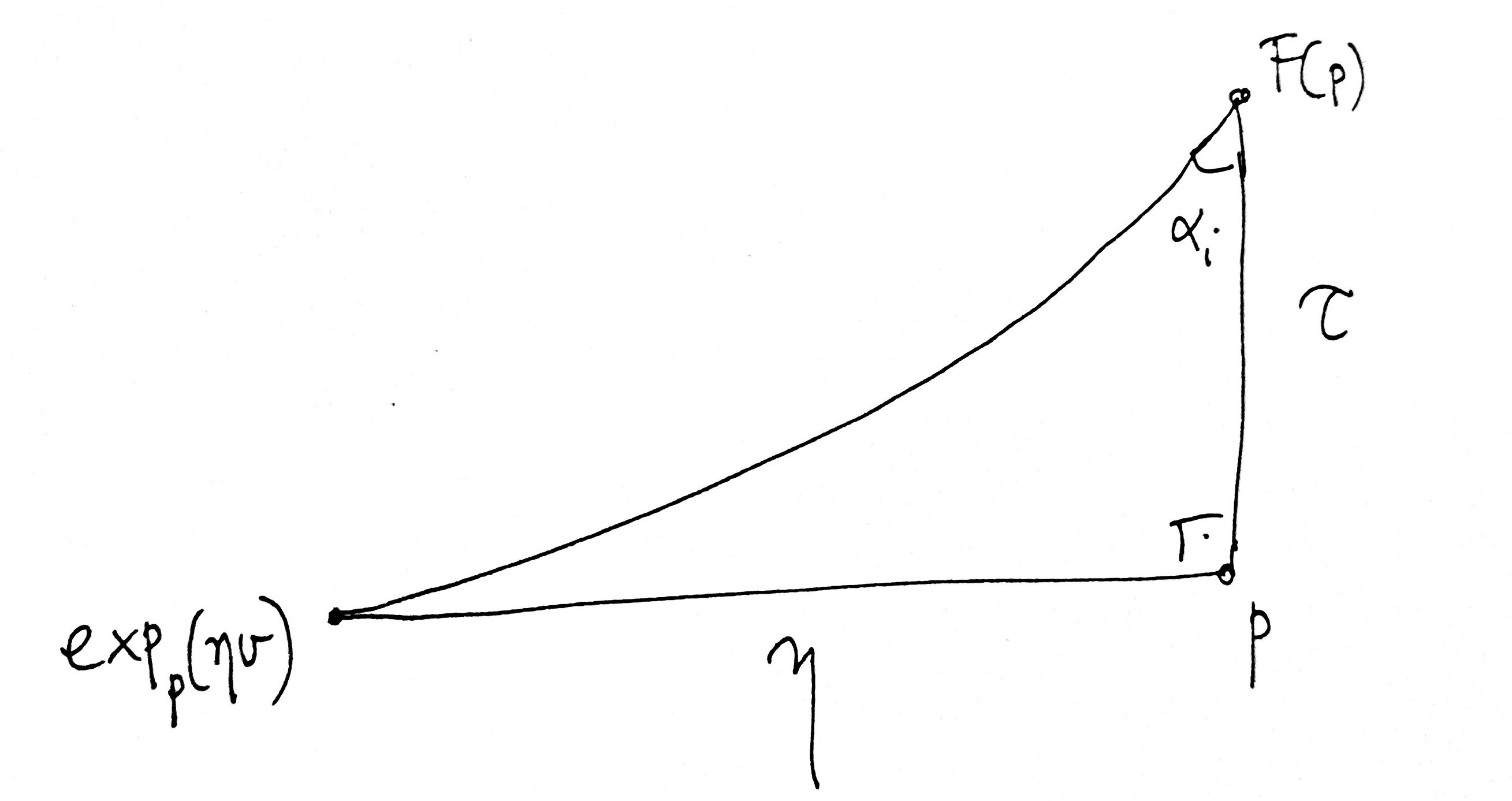}
\caption{If $\theta_i$ was smaller or equal to $\til{\theta}_i = \cos^{-1}(\tanh \tau_i(p)/\tanh \eta)$ as in the figure above, then the supporting plane to $H_i^+$ containing $F_i(p)$ would intersect $\Delta$, in a violation of convexity.}
\label{angle}
\end{figure}

Let $\alpha_i$ be the angle based at $F_i(p)$ between the normal geodesic $t\mapsto n_t(p)$ to $\Delta$ at $p$ and the geodesic segment from $F_i(p)$ to $\exp_p (\eta v)$. (See Figure \ref{angle}.) 

Note that $\alpha_i < \theta_i$. If that was not the case and $\theta_i$ was any smaller, the supporting plane of $H_i^+$ containing $F_i(p)$ would intersect the intrinsic disc $B_{\eta}(p)\se \Delta$ of radius $\eta$, in a contradiction of convexity.

On the other hand, from trigonometry,
\[
\cos \theta_i < \cos \alpha_i = \frac{\tanh \tau_i(p)}{\tanh \eta}.
\]{}
Thus, as $\|\tau_i\|_{L^{\infty}(\Delta^{\eta})}\to 0$ as $i\to \infty$, it follows that $\cos \theta_i\to 0$ uniformly in $(p,v)$. This finishes the argument.
\end{proof}

\begin{proof}[Proof of proposition]

For $p\in\bH^3$, let $v,w \in \T_p\bH^3$. We let $\area_p(v,w)$ denote the area spanned by $v$ and $w$ in $\T_p \bH^3$, with respect to the hyperbolic metric $G = \la \cdot,\cdot\ra$ of $\bH^3$. In formulas,
\[
\area_p (v,w) = \lef| \det \begin{bmatrix} \la v,v\ra & \la v, w\ra \\ \la v,w\ra & \la w,w \ra \end{bmatrix} \ri|^{1/2}.
\]

The Radon-Nikodym derivative $\det dF_i$ measures the area distortion caused by $F_i$. If $p\in \Delta_i^{\eta}$, we have
\[
\det dF_i (p) = \frac{\area(dF_i(p)v,dF_i(p)w)}{\area(v,w)},
\]
where $v$ and $w$ are distinct vectors in $\T^1_p\bH^3$.

Consider the coordinates in $\bH^3$ given by $(x,y,t) = n_t(x,y)$, where we choose $(x,y)\in\bH^2$ so that $\partial_x$ and $\partial_y$ form an orthonormal basis for $\T_p \bH^2$, where $\bH^2$ denotes the geodesic plane containing $\Delta$. In these coordinates, the metric on $n_t(\bH^2)$ is given by \[G_t = \cosh^2 t \, g_{\bH^2} + dt^2,\] where $g_{\bH^2}$ is the hyperbolic metric of $\bH^2$. We also have
\[
F_i(x,y,0) = (x,y,\tau_i(x,y)),
\]
and for $v\in T_p\bH^2$,  
\begin{align*}
dF_i(p)(v) &= v + d\tau_i(p) v \,\partial_t.
\end{align*}

With these explicit formulae for $dF_i$ and $G_t$, we can calculate the area distortion $\det dF_i(p)$, which turns out to be
\begin{align*}
\det dF_i(p) &= \frac{\area(dF_i(p)\, \partial_x,dF_i(p)\,\partial_y)}{\area(\partial_x,\partial_y)} \\
&= \lef| \det \begin{bmatrix} \cosh^2 \tau_i(p) + (\partial_x \tau_i(p))^2 & \partial_x \tau_i(p)\,\partial_y \tau_i(p) \\ \partial_x \tau_i(p)\,\partial_y \tau_i(p) & \cosh^2 \tau_i(p) + (\partial_y \tau_i(p))^2  \end{bmatrix} \ri|^{1/2} \\
&= \lef( \cosh^4 \tau_i(p) + |\nabla \tau_i (p)|^2 \cosh^2 \tau_i(p) \ri)^{1/2}.
\end{align*}

Using Lemma \ref{c1} above, we see that
\[
\|\det dF_i - 1\|_{L^{\infty}(\Delta^{\eta})} \to 0 \text{ as }i\to\infty.
\]
\end{proof}

As argued above, it follows that\[
\|\det dF_i - 1\|_{L^{\infty}(P_i^{\eta})} \to 0 \text{ as }i\to\infty.
\]
This concludes the proof of item {\bf iii} of Proposition \ref{map}.
\end{proof}

Let $H^{\eta}_i := F_i(P^{\eta}_i)$ and let $R^{\eta}_i$ be its projection to $M$. Let $\tilde{h}^{\eta}_i$ be the area measure of $H^{\eta}_i$ and let $h^{\eta}_i$ be the restriction of $h_i$ (the probability area measure of $R_i = \Gamma\backslash H_i^+$) to $R^{\eta}_i$.

\begin{cor}\label{complementsmall}
$h^{\eta}_i (M - R^{\eta}_i) \to 0$ as $\eta\to 0$.
\end{cor}

\begin{proof}
Given an ideal triangle $T\se P_i - \lambda_i$, the area of $F_i^{\eta} (T^{\eta})$ is larger than that of $T^{\eta}$. Since $R_i$ and $\Sigma_i = \Gamma \backslash P_i$ have the same area (as they are pleated and homotopic to each other), the corollary follows from the fact that $p_i(M - \Sigma^{\eta}_i)\to 0$ as $\eta\to 0$.
\end{proof}

\begin{claim}\label{pleatedupstairs}
Let $(g_{\alpha}) \se C(\Gr \bH^3)$ be a bounded and equicontinuous family of functions, namely
\begin{enumerate}[i.]
\item $\sup_{\alpha} \|g_{\alpha}\|_{\infty} < \infty$ and the support of the $(g_{\alpha})$ has uniformly bounded diameter.
\item There is a function $w:(0,\infty)\to\bR$ satisfying $w(\delta)\to 0$ as $\delta\to 0$ such that
\[
|g_{\alpha} (x) - g_{\alpha} (y)| \leq w\lef(\dist_{\Gr \bH^3} (x,y)\ri).
\]
\end{enumerate}
Then,
\[
\sup_{\alpha} \lef| 
\int_{\Gr \bH^3} g_{\alpha} \,d\tilde{p}^{\eta}_i - \int_{\Gr \bH^3} g_{\alpha} \,d \tilde{h}^{\eta}_i \ri| \to 0 \text{ as }i\to\infty.
\]
\end{claim}

\begin{proof}
Since $\det dF_i = d(F_i^* \tilde{h}^{\eta}_i)/d\til{p}^{\eta}_i$, we have
\[
\int_{\Gr \bH^3} g_{\alpha} \, d\tilde{h}^{\eta}_i
= \int_{P^{\eta}_i} g_{\alpha} (F_i(x)) \, \det dF_i (x) \,d\tilde{p}^{\eta}_i(x).
\]
Thus,
\[
\lef| \int_{\Gr \bH^3} g_{\alpha} \,d\tilde{p}^{\eta}_i -  \int_{\Gr \bH^3} g_{\alpha} \,d\tilde{h}^{\eta}_i \ri| \leq
\int_{P^{\eta}_i} |g_{\alpha} \circ F_i | \, |\det(dF_i) - 1| \, d\tilde{p}^{\eta}_i + \int_{P^{\eta}_i} | g_{\alpha}\circ F_i - g_{\alpha}|\, d\tilde{p}^{\eta}_i
\]
From the boundedness and equicontinuity of $g_{\alpha}$ (including the fact that the supports of $g_{\alpha}$ and hence $g_{\alpha}\circ F_i$ have uniformly bounded diameter) and the fact that $\det dF_i$ converges uniformly to 1 (Claim \ref{map}),  we see that the right hand side of this inequality goes to zero.
\end{proof}

\begin{claim}\label{upseta}
Let $g\in C(\Gr M)$. Then,
\[
\lim_{i\to\infty} \int_{\Gr M} g \, dp^{\eta}_i = 
\lim_{i\to\infty} \int_{\Gr M} g \, dh^{\eta}_i.
\]
\end{claim}

\begin{proof}
It suffices to take a $g$ supported in a small geodesic ball $B\se \Gr M$. Let $\tilde{B}$ be a lift of this ball to $\Gr \bH^3$. Then, there is $\tilde{g}\in C(\Gr \bH^3)$ and a finite set $K_i \se \Gamma$ so that
\[
\int_{\Gr M} g \,dp^{\eta}_i = \frac{1}{\area (\Sigma_i)} \sum_{\gamma\in K_i} \int_{\Gr \bH^3} \tilde{g}\circ \gamma \, d\tilde{p}^{\eta}_i
\]
and similarly,
\[
\int_{\Gr M} g \,dh^{\eta}_i = \frac{1}{\area (R_i)} \sum_{\gamma\in K_i} \int_{\Gr \bH^3} \tilde{g}\circ \gamma \, d\tilde{h}^{\eta}_i.
\]

Note that $(\tilde{g}\circ\gamma)_{\gamma\in\Gamma}$ is a bounded and equicontinuous family of functions in $C(\Gr \bH^3)$.

We claim moreover that the $K_i$ can be chosen so that 
\[\frac{\sup_i \#K_i}{\area(\Sigma_i)} < \infty.\] 
Indeed, let $2B\se \Gr M$ be a ball of twice the radius as $B$, centered at the same point. Then, $\#K_i$ is no larger than the number of connected components of $\Sigma_i\cap 2B$ that meet $B$. Each such component $C$ satisfies $\area(C) \geq c(B)$, where $c(B)$ is a constant depending only on $B$. Thus, we have
\[
\#K_i \cdot c(B) \leq \area(\Sigma_i),
\]
which shows that $\#K_i/\area(\Sigma_i)\leq c(B)^{-1}$.

Using the fact that $\area(\Sigma_i) = \area(R_i)$, we estimate
\begin{align*}
&\lef|\int_{\Gr M} g\,dp^{\eta}_i - \int_{\Gr M} g\,dh^{\eta}_i \ri| \leq
 \frac{1}{\area(\Sigma_i)} \sum_{\gamma\in K_i} \lef| \int_{\Gr \bH^3} \tilde{g}\circ \gamma \, d\tilde{p}^{\eta}_i -  \int_{\Gr \bH^3} \tilde{g}\circ \gamma \, d\tilde{h}^{\eta}_i \ri|\\
 &\leq \frac{\# K_i}{\area(\Sigma_i)} \sup_{\gamma\in\Gamma} 
 \lef| \int_{\Gr \bH^3} \tilde{g}\circ\gamma\,d\tilde{p}^{\eta}_i - 
 \int_{\Gr \bH^3} \tilde{g}\circ \gamma\, d\tilde{h}^{\eta}_i \ri|.
\end{align*}

The term above goes to zero as $i\to \infty$ due to the boundedness of $\#K_i/\area(\Sigma_i)$ and the boundedness and equicontinuity of $(\tilde{g}\circ\gamma)_{\gamma\in\Gamma}$.
\end{proof}

Finally, let $g\in C(\Gr M)$. Then,

\begin{align*}
&\lef| \int_{\Gr M} g\,dp_i - \int_{\Gr M} g\,dh_i \ri|\\
&\leq \lef| \int_{\Gr M} g\,dp^{\eta}_i - \int_{\Gr M} g\,dh^{\eta}_i\ri|
+ \int_{\Gr M} |g|\cdot 1_{P_i - P^{\eta}_i}\,dp_i + 
\int_{\Gr M} |g|\cdot 1_{H^+_i - H^{\eta}_i}\,dh_i\\
&\leq \lef| \int_{\Gr M} g\,dp^{\eta}_i - \int_{\Gr M} g\,dh^{\eta}_i\ri| + \|g\|_{L^{\infty}(\Gr M)} \lef( p_i^{\eta} (M - \Sigma^{\eta}_i) + h^{\eta}_i (M - H_i^{\eta})  \ri).
\end{align*}

Claim \ref{upseta} implies that the first summand in the expression above goes to zero as $i\to \infty$. The second summand, in turn, goes to zero as $\eta\to 0$, from Corollary \ref{complementsmall}. Since $\eta$ was arbitrary, we have shown
\[
\lef| \int_{\Gr M} g\,dp_i - \int_{\Gr M} g\,dh_i \ri| \to 0 \text{ as }i\to\infty.
\]
In particular, if a subsequence $p_{i_j}$ converges to $\nu$, then so does $h_{i_j}$ and vice-versa. This completes the proof of Theorem \ref{sl1}.

\subsection{Minimal surfaces}

A map $f:S\to M$ of a surface $S$ into $M$ is \emph{minimal} if the principal curvatures of $f(S)$ (a \emph{minimal surface}) sum to zero at every point. These surfaces turn out to be locally area-minimizing.

Let $f:S\to M$ be a $\pi_1$-injective map of a hyperbolic surface $S$ into $M$. Schoen-Yau \cite{SY} and Sacks-Uhlenbeck \cite{SU} show that $f$ is homotopic to a minimal map $f^m$. In addition, Uhlenbeck shows that if the principal curvatures $\pm \lambda(p)$ of $f^m(S)$ satisfy $\lambda(p)\in (-1,1)$ for every $p\in f^m(S)$, then $f^m$ is quasifuchsian and it is the unique minimal map in its homotopy class. Finally, Seppi \cite{Se} shows that for a minimal $K$-quasifuchsian map $f^m:S\to M$ with $K$ small enough, 

\begin{thm}[Seppi]\label{seppimain}The principal curvatures $\pm \lambda$ of $f^m(S)$ satisfy
\[
\|\lambda\|_{L^{\infty}(f^m(S))} \leq C\log K,
\] for an universal constant $C$.
\end{thm}

Combining these theorems, we see that if $f_i:S_i\to M$ are asymptotically Fuchsian maps, then for $i$ large enough $f_i$ is homotopic to a unique minimal map $f_i^m$. In addition, the principal curvatures of $f_i^m(S_i)$ go to zero uniformly as $i\to\infty$.

\subsection{Proving Theorem \ref{sl2}}

We will now restate and prove Theorem \ref{sl2}. As before, $f_i:S_i\to M$ are asymptotically Fuchsian maps and $f_i^h$ is the pleated map homotopic to $f_i$ coming from the top component $H_i^+$ of $Q_i = (f_i)_* (\pi_1 S_i)$. We let $f_i^m$ be the minimal maps homotopic to $f_i$. We denote the probability area measure induced by $f_i^m$ and $f_i^h$ as $m_i = \nu(f_i^m)$ and $h_i = \nu(f_i^h)$.

\begin{thm*}\label{hm} A subsequence $m_{i_j}$ satisfies $m_{i_j}\wkstar \nu$ if and only if $h_{i_j}\wkstar \nu$.
\end{thm*}

We let $\wt{f_i^m}$ be the lift of $f_i^m$ to $\bH^2$ so that $\partial_{\infty} \wt{f_i^m}$ is the limit set $\Lambda_i$ of $Q_i$. We let $\til{m}_i$ and $\til{h}_i$ be the area measures induced by $\wt{f^m_i}$ and $\wt{f^h_i}$ on $\Gr\bH^3$. As before, $\beta_i$ is the bending lamination of $H_i^+$, $R_i = \Gamma\backslash H_i^+$ and we put $N_i:= \Gamma\backslash D_i = f_i^m(S_i)$.

As in the proof of Theorem \ref{sl1}, we define a map
\[
F_i : D_i \arr H_i^+
\]
where $F_i(p)$ is given by flowing $p$ in the direction normal to $D_i$ for the time $\tau_i(p)$ it takes to hit $H_i^+$. Concisely, $F_i(p) = n_{\tau_i(p)}(p)$.

\begin{claim}\label{Fnice}
The map $F_i$ satisfies the following properties:
\begin{enumerate}[i.]
\item $F_i$ is differentiable outside of $F_i^{-1}(\beta_i)$
\item $\tilde{m}_i (F_i^{-1}(\beta_i)) = 0$
\item $\| \det(dF_i) - 1 \|_{L^{\infty} (D_i)} \to 0$ as $i\to\infty$.
\end{enumerate}
\end{claim}

To prove this claim, we will use the following rephrasing of Proposition 4.1 of Seppi in \cite{Se}:

\begin{prop}\label{seppiprop}
Suppose $i$ is large enough that the uniformizing map $f_i$ has Bers norm $\|f_i\|_B<1/2$. Then, we may find surfaces $D^-_i$ and $D^+_i$ that are equidistant from $D_i$ so that the region between $D^-_i$ and $D^+_i$ is convex and thus contains $\core Q_i$.

Moreover, given $x\in D_i$, there is a geodesic segment $\alpha$ from $D_i^-$ to $D_i^+$, meeting $D_i$ and $D_i^{\pm}$ orthogonally, whose length satisfies
\begin{equation}\label{seppi}
\ell(\alpha) \leq \arctanh (2\|f_i\|_B).
\end{equation}
\end{prop}

\begin{figure}
\includegraphics[scale=0.09]{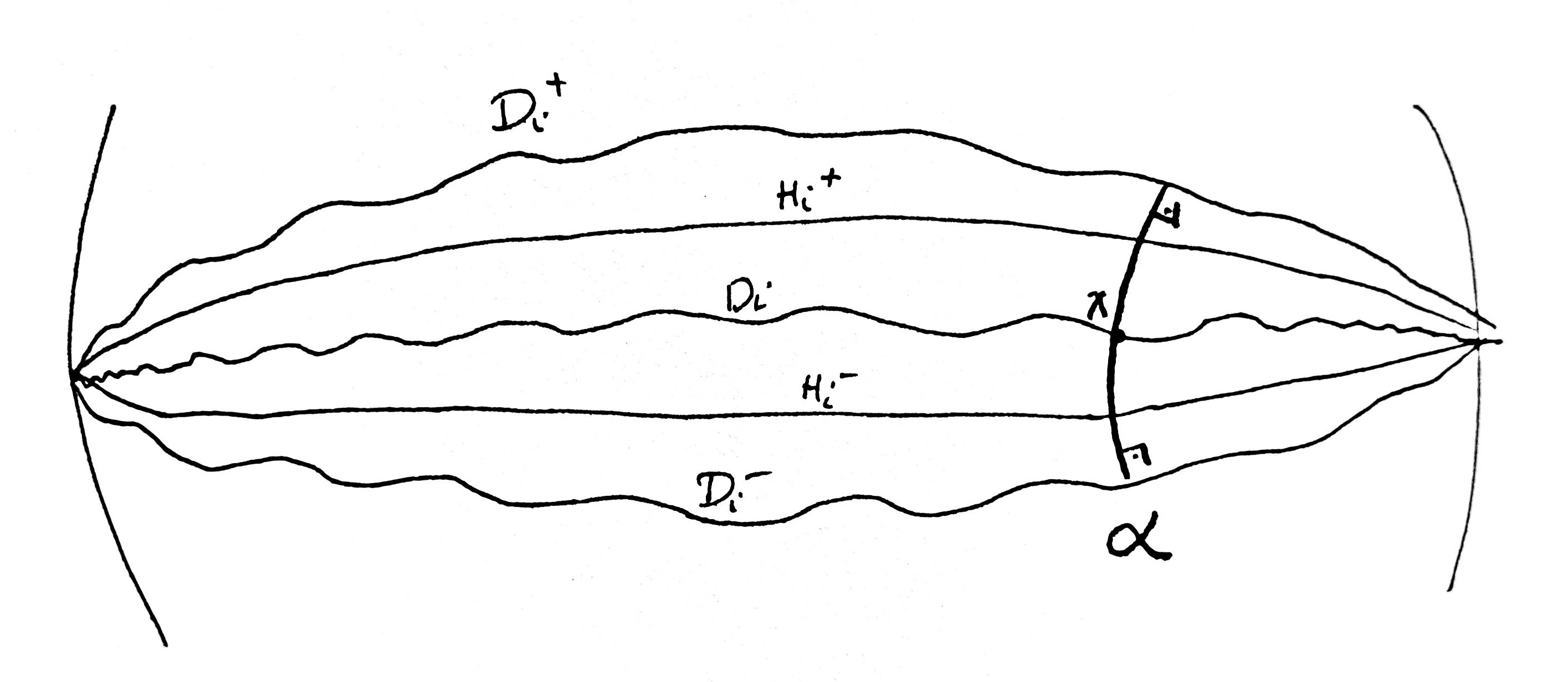}
\caption{Illustrating Proposition \ref{seppiprop}}
\end{figure}

In particular, given $x_i\in D_i$, let $P^+_i$ and $P^-_i$ be the geodesic planes tangent to $D^+_i$ and $D^-_i$ at the endpoints of the segment $\alpha_i$. From \ref{seppi}, we see that the distance between $P^+_i$ and $P^-_i$ goes to zero as $i\to\infty$ and does not depend on the chosen point $x_i \in D_i$.

\begin{proof}[Proof of Claim \ref{Fnice}]

{\bf i.} Let $x\in D_i - F_i^{-1} (\beta_i)$. Then, $F_i$ maps a disc around $x$ to a piece of a totally geodesic plane in $H_i^+$ by the geodesic flow in the normal direction. This is a smooth map.

\medskip

{\bf ii.} Let $E$ be the set containing all the points above $D_i^-$ and below $D_i^+$. This set is foliated by surfaces $D_i^t$ equidistant to $D_i$, for $t\in [-\dist(D_i^-,D_i),\dist(D_i,D_i^+)]$. The set $E$ is also foliated by the orbits of the geodesic flow going through points in $D_i$ and their normal vector. These flow lines never meet. If they did, the pullback metrics of the $D_i^t$ on $D_i$ would be degenerate, which is the not the case, as their principal curvatures are given by
\[
\frac{\lambda-\tanh t}{1-\lambda\tanh t} \aand \frac{-\lambda-\tanh t}{1+\lambda\tanh t}
\]
and $\lambda\in (-1,1).$

In particular, we can define a map
\[
G_i : E \arr D_i,
\]
which takes $y \in D_i^t \se E$ back to the point $x\in D_i$ so that $g_t(x,n) = y$. This map is smooth and in particular, its restriction to $H_i^+$ is Lipschitz and hence takes sets of measure zero to sets of measure zero. Thus,
\[
\tilde{m}_i (G_i(\beta_i)) = \tilde{m}_i (F_i^{-1}(\beta_i)) = 0.
\]

\medskip

{\bf iii.} As before, first we show that the hitting times $\tau_i$ converge to zero uniformly on $D_i$ in the $C^1$ sense.

\begin{lem}\label{c12} $\|\tau_i\|_{C^1(D_i)}\to 0$ as $i\to\infty$.
\end{lem}

\begin{proof}
The fact that $\|\tau_i\|_{L^{\infty}(D_i)}\to 0$ as $i\to\infty$ follows readily from Seppi's Proposition \ref{seppiprop}. It remains to show that $d\tau_i(p)v\to 0$ as $i\to\infty$ uniformly in $\T^1 (D_i - F_i^{-1}(\beta_i))$. This in turn will follow from Seppi's Theorem \ref{seppimain} that says that the principal curvatures of $D_i$ converge uniformly from zero.

For $(p,v)\in \T^1 D_i$, we let $\theta_i(p,v)$ be the angle in $(0,\pi/2]$ that the geodesic normal to $D_i$ at $p$ makes with the curve $s\mapsto F_i(\exp_p sv)$. As in the pleated case (Lemma \ref{c1}), it suffices to show that $\theta_i(p,v)\to \pi/2$ uniformly in $\T^1D_i$.

Fix $\eta > 0$.
Again, we let $\alpha_i$ be the angle based at $F_i(p)$ between the normal geodesic $t\mapsto n_t(p)$ to $\Delta$ at $p$ and the geodesic segment from $F_i(p)$ to $\exp_p (\eta v)$. Let $\exp^{D_i}_p : \T^1_p D_i \to D_i$ denote the exponential map intrinsic to $D_i$. We let  $\alpha'_i$ be the angle based at $F_i(p)$ between the normal geodesic $t\mapsto n_t(p)$ to $\Delta$ at $p$ and the \emph{intrinsic} geodesic segment of $D_i$ from $F_i(p)$ to $\exp^{D_i}_p (\eta v)$.

\begin{figure}
\includegraphics[scale=0.06]{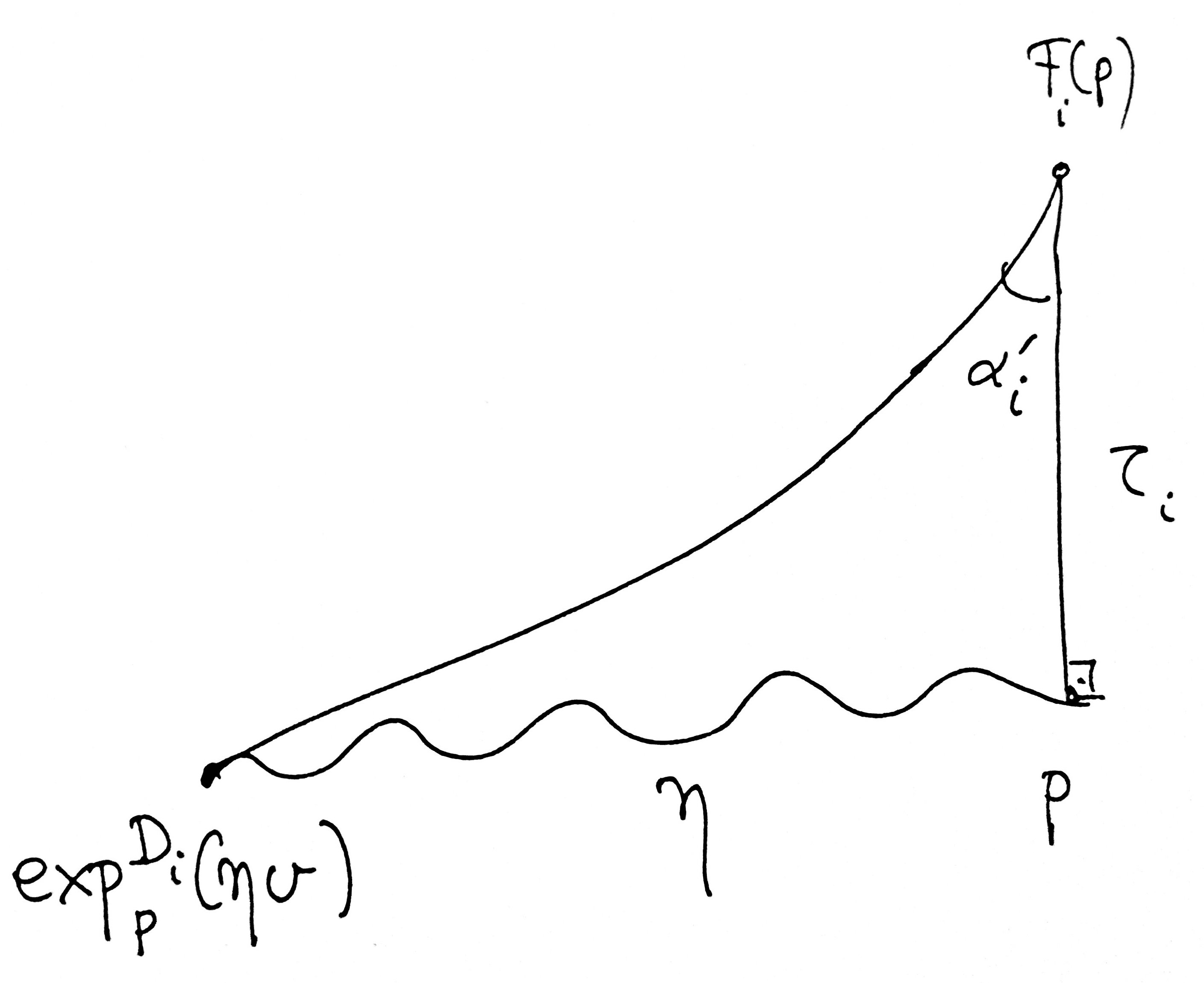}
\caption{The angle $\alpha'_i$ is defined in a similar way to $\alpha_i$, except that it is opposite to an intrinsic geodesic of $D_i$ of length $\eta$, rather than a geodesic of $\bH^3$.}
\end{figure}

Note that $\alpha'_i < \theta_i$. If that was not the case, a supporting plane to $H_i^+$ based at $F_i(p)$ would meet $D_i$, in a violation of convexity.

Due to the principal curvatures of $D_i$ going to zero as $i\to \infty$, it follows that the difference between $\alpha_i$ and $\alpha'_i$ also goes to zero unformly as $i\to \infty$. In other words, there is a quantity $\omega_i \to 0$ as $i\to \infty$ (which depends on the choice of $\eta>0$, but not of $(p,v)\in \T^1 D_i$), so that
\[
|\cos \alpha_i - \cos \alpha'_i| \leq \omega_i.
\]

But as before, $\cos \alpha_i = \tanh(\tau_i(p))/\tanh\eta$. Thus,
\[
\cos \theta_i \leq \frac{\tanh\|\tau_i\|_{L^{\infty} (D_i)}}{\tanh\eta} + \omega_i \iinfty 0.
\]
\end{proof}

For each $i$, we choose coordinates on $\bH^3$ given by $(x,y,t) = n_t(x,y)$, where $(x,y)$ are coordinates for $D_i$ chosen so that $\partial_x$ and $\partial_y$ form an orthonormal basis for $T_{p_i} D_i$. (The points $p_i$ are chosen in the full measure set $D_i - F_i^{-1}(\beta_i)$.) In these coordinates, the metric $G_t$ on $n_t(D_i)$ is given by
\[
G_t = g_t + dt^2,
\]
where at $n_t(p_i)$, the matrix entries of $g_t$ corresponding to the basis $\partial_x$, $\partial_y$ are given by
\[\tag{$\star$}
g_t = (\cosh t \Id + \sinh t A_i)^2
\]
and $A_i$ is the second fundamental form of $D_i$. (See Section 5 of Uhlenbeck \cite{U} for details.)

As before, in these coordinates we have $F_i(x,y,0) = (x,y,\tau_i(x,y))$ and so $dF_i(p)v = v + d\tau_i(p) v \, \partial/\partial t$. Thus,
\begin{align*}
\det dF_i(p_i) &= {\area(dF_i(p_i)\,\partial_x,dF_i(p_i)\,\partial_y)} \\
&= \lef| \det \begin{bmatrix} g_{\tau_i(p_i)} (\partial_x,\partial_x) + (\partial_x \tau_i(p_i))^2 & \partial_x \tau_i(p)\,\partial_y \tau_i(p_i) \\ \partial_x \tau_i(p_i)\,\partial_y \tau_i(p) & g_{\tau_i(p_i)} (\partial_y,\partial_y) + (\partial_y \tau_i(p_i))^2  \end{bmatrix} \ri|^{1/2} \\
&= \lef( |\partial_x|^2_{\tau_i(p_i)}|\partial_y|^2_{\tau_i(p_i)} + (\partial_x\tau_i(p_i))^2|\partial_y|^2_{\tau_i(p_i)} + (\partial_y\tau_i(p_i))^2 |\partial_x|^2_{\tau_i(p_i)} \ri)^{1/2},
\end{align*}

Above, $|\partial_x|^2_{\tau_i(p_i)}$ and $|\partial_y|^2_{\tau_i(p_i)}$ denote, respectively, the first and second diagonal entries of $g_{\tau_i(p_i)}$. From Seppi's result (Theorem \ref{seppimain}), we know that the second fundamental forms $A_i$ converge uniformly to the zero matrix as $i\to\infty$. In view of the formula ($\star$), it follows that $|\partial_x|^2_{\tau_i(p_i)}$ and $|\partial_y|^2_{\tau_i(p_i)}$ both converge uniformly to 1 as $i\to\infty$.

In addition, from Lemma \ref{c12}, we know that the derivatives of $\tau_i(p_i)$ converge to zero as $i\to\infty$. We can thus conclude that

\[
\|\det dF_i(p) - 1 \|_{L^{\infty} (D_i)} \to 0 \text{ as }i\to\infty.
\]
\end{proof}

Let $\tilde{m}_i$ and $\tilde{h}_i$ be the area measures of $D_i$ and $H^+_i$ in $\Gr \bH^3$.

\begin{claim}\label{upstairs}
Let $(g_{\alpha}) \se C(\Gr \bH^3)$ a bounded and equicontinuous family of functions.
Then,
\[
\sup_{\alpha} \lef| 
\int_{\Gr \bH^3} g_{\alpha} \,d\tilde{m}_i - \int_{\Gr \bH^3} g_{\alpha} \,d \tilde{h}_i \ri| \to 0 \text{ as }i\to\infty.
\]
\end{claim}

\begin{proof}
The proof is the same as the proof of Claim \ref{pleatedupstairs}, substituting $\tilde{h}_i$ for $\tilde{h}^{\eta}_i$ and $\tilde{m}_i$ for $\tilde{p}^{\eta}_i$.\end{proof}

Now, let $g\in C(\Gr M)$. In a similar fashion to the proof of Claim \ref{upseta}, we proceed to show that
\[
\lim_{i\to\infty} \int_{\Gr M} g\,dm_i = \lim_{i\to\infty}
\int_{\Gr M} g\,dh_i.
\]
As before, we may choose $g$ to be supported in a small geodesic ball $B\se \Gr M$. For a lift $\tilde{B}$ of $B$ to $\Gr \bH^3$, there is $\tilde{g}\in C(\Gr \bH^3)$ and a finite set $K_i \se \Gamma$ so that
\[
\int_{\Gr M} g \,dm_i = \frac{1}{\area (N_i)} \sum_{\gamma\in K_i} \int_{\Gr \bH^3} \tilde{g}\circ \gamma \, d\tilde{m}_i
\]
and similarly,
\[
\int_{\Gr M} g \,dh_i = \frac{1}{\area (\Sigma_i)} \sum_{\gamma\in K_i} \int_{\Gr \bH^3} \tilde{g}\circ \gamma \, d\tilde{h}_i.
\]

As before, $(\tilde{g}\circ\gamma)_{\gamma\in\Gamma}$ is a bounded and equicontinuous family of functions in $C(\Gr M)$, and the $K_i$ can be chosen so that 
$\sup_i \#K_i/\area(\Sigma_i) < \infty.$

Now we estimate
\begin{align*}
&\lef|\int_{\Gr M} g\,dm_i - \int_{\Gr M} g\,dh_i \ri| \leq
 \frac{1}{\area(\Sigma_i)} \sum_{\gamma\in K_i} \lef| \frac{\area(\Sigma_i)}{\area(N_i)}\int_{\Gr \bH^3} \tilde{g}\circ \gamma \, d\tilde{m}_i -  \int_{\Gr \bH^3} \tilde{g}\circ \gamma \, d\tilde{h}_i \ri|\\
 &\leq \frac{\# K_i}{\area(\Sigma_i)} \sup_{\gamma\in\Gamma} \lef[ 
 \lef| \int_{\Gr \bH^3} \tilde{g}\circ\gamma\,d\tilde{m}_i - 
 \int_{\Gr \bH^3} \tilde{g}\circ \gamma\, d\tilde{h}_i \ri| + 
 \lef| 1 - \frac{\area(\Sigma_i)}{\area(N_i)} \ri|\, \| \tilde{g}\circ \gamma\|_{L^1(\tilde{h}_i)} \ri].
\end{align*}

The upper bound above goes to zero as $i\to \infty$ due to the boundedness of $\#K_i/\area(\Sigma_i)$; the equicontinuity and boundedness of $(\tilde{g}\circ\gamma)_{\gamma\in\Gamma}$ and the fact that $\area(\Sigma_i)/\area(N_i)$ goes to 1. 

To see the latter, say $\pm \lambda_i$ are the principal curvatures of $N_i$ and $g_i$ is the genus of $S_i$. Using the Gauss-Bonnet formula, we have
\begin{align*}
\lef| 1 - \frac{\area(\Sigma_i)}{\area (N_i)} \ri| &= \frac{1}{\area (N_i)} \lef| \int_{N_i} 1 \,d\area - 2\pi(2g_i - 2) \ri| \\
&= \frac{1}{\area(N_i)} \lef| \int_{N_i} 1 \,d\area - \int_{N_i} \lambda_i^2 \,d\area \ri| \\
&\leq \frac{1}{\area(N_i)} \| 1 - \lambda_i^2 \|_{L^{\infty} (N_i)} \area (N_i).
\end{align*}
We know that $\|1 - \lambda_i^2\|_{L^{\infty}(N_i)}\to 0$ as $i\to\infty$ due to Seppi's Theorem \ref{seppimain}.

\section{Building surfaces out of good pants}

In this section, we will outline how to construct a $\pi_1$-injective closed oriented nearly Fuchsian surface in $M$ out of good pants. This is the Kahn-Marković surface subgroup theorem \cite{KM}, though our exposition will line up with that of Kahn and Wright \cite{KW} as well as use some notation from Liu and Marković \cite{LM}.

\subsection{Building blocks}

The following paragraphs define the many terms related to the building blocks of this construction.

An orthogeodesic $\gamma$ between two closed geodesics $\alpha_0, \alpha_1\se M$ is a geodesic segment parametrized with unit speed going from $\alpha_0$ to $\alpha_1$ and meeting both curves orthogonally.

We denote by $\curves$ the space of $(\epsilon,R)$-\emph{good curves}. Those are the closed oriented geodesics whose complex translation length $\bl(\gamma)$ is $2\eps$-close to $2R$.

Let $P_R$ be the planar oriented hyperbolic pair of pants whose cuffs $C_i$ have length exactly $2R$ for $i\in \bZ/3$. We define the space $\pants$ of $(\eps,R)$-good pants to be the space of equivalence classes of maps $f:P_R\to M$ so that $f(C_i)$ is homotopic to an element of $\curves$, for all $i\in\bZ/3$. Two representatives $f$ and $g$ of elements of $\pants$ are equivalent if $f$ is homotopic to $g\circ \psi$ for some orientation-preserving homeomorphism $\psi:P_R\to P_R$. 

We let $\tpants$ be the space of \emph{ends of $(\eps,R)$ good pants}, which can be thought as good pants with a marked cuff. Precisely $\tpants$ is the space of equivalence classes of pairs $[(f,C_i)]$, where $f\in\pants$ and $C_i\se \partial P_R$ is a cuff. We say two representatives $(f,C_i)$ and $(g,C_j)$ of elements of $\tpants$ are equivalent if $f$ is homotopic to $g\circ \psi$, where $\psi:P_R\to P_R$ is an orientation-preserving homeomorphism $\psi: P_R\to P_R$ so that $\psi(C_i) = C_j$. Note that forgetting the cuff of $[(f,C_i)]\in \tpants$ defines a three-to-one surjection from $e:\tpants\to\pants$. For $\pi\in\pants$, we call $e^{-1}(\pi)$ the \emph{ends} of $\pi$.

For $\gamma\in\curves$, we let $\tpants(\gamma)$ be the $[(f,C_i)]\in \tpants$ so that $f(C_i)$ is homotopic to $\gamma$ or its orientation reversal $\gamma^{-1}$. We can decompose $\tpants(\gamma)$ into $\pants^+(\gamma)\sqcup \pants^-(\gamma)$, where $\pants^+(\gamma)$ consists of the $[(f,C_i)]$ with $f(C_i) \sim \gamma$ and $\pants^-(\gamma)$ consists of the $(f,C_i)$ with $f(C_i)\sim \gamma^{-1}$. There is a bijection
\[
r: \pants^-(\gamma) \arr \pants^+ (\gamma)
\]
given by $r([(f,C_i)]) = ([(f\circ \rho,C_i)])$, where $\rho: P_R\to P_R$ is the reflection along the short orthogeodesics of $P_R$. We let $\pants(\gamma)$ denote the quotient of $\tpants(\gamma)$ by $r$.

\begin{figure}
\includegraphics[scale=0.055]{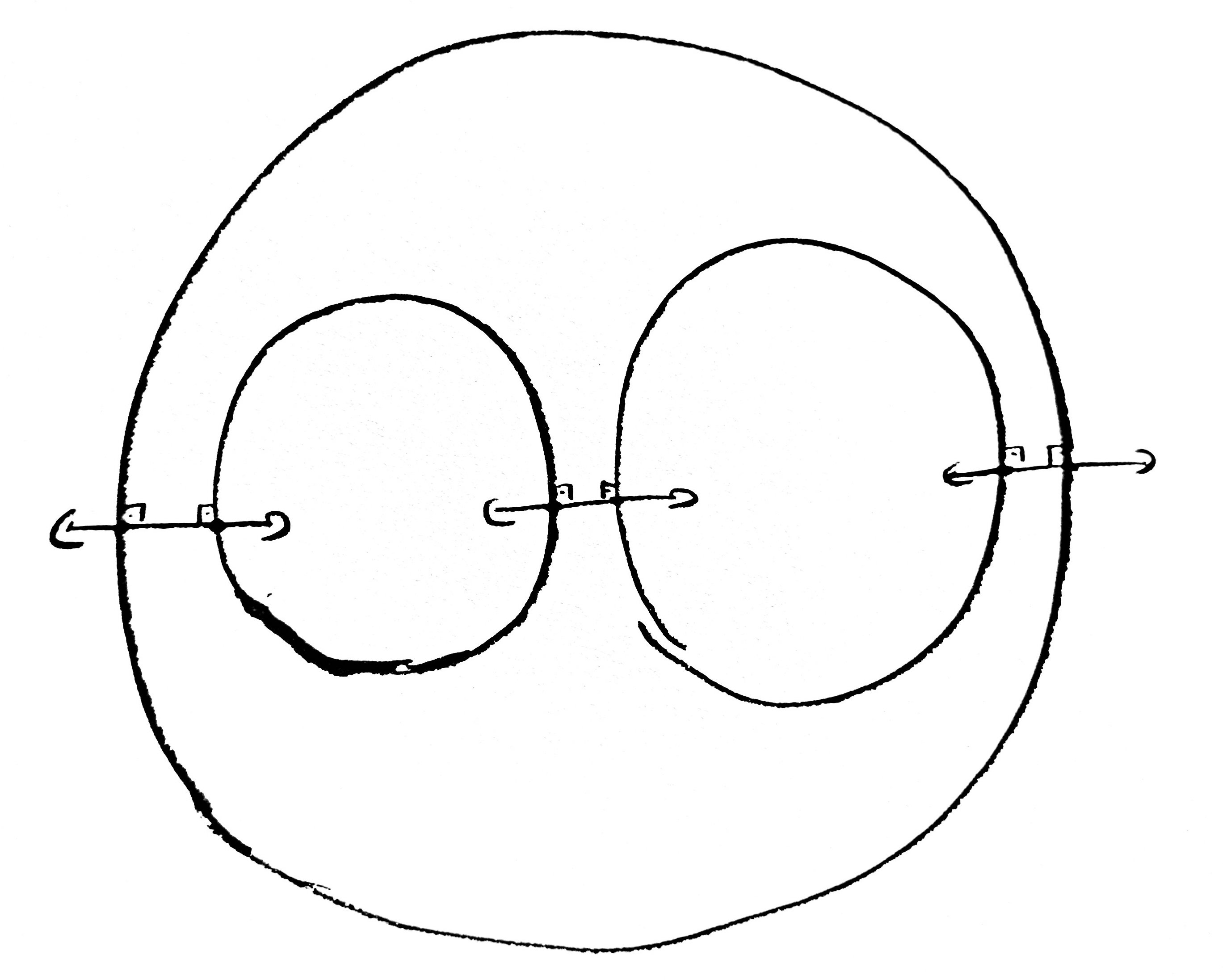}
\caption{Short orthogeodesics and feet of a good pants.}
\end{figure}

The planar pair of pants $P_R$ is equipped with three \emph{short orthogeodesics}, which are the orthogonal geodesic segments  from one cuff to another. The short orthogeodesic from $C_i$ to $C_j$ is denoted $a_{ij}$. A marked pair of pants $\pi\in \tpants$ comes with left and right short orthogeodesics, respectively denoted $\eta^{\ell}(\pi)$ and $\eta^r(\pi)$, which are defined as follows. Choose a representative $(f,C_i)\in\pi$ that sends cuffs $C_j\se \partial P_R$ to geodesics $\gamma_j\se M$. We let $\eta^{\ell}(\pi)$ be the geodesic segment homotopic to $f(a_{i,i+1})$ (through segments from $\gamma_i$ to $\gamma_{i-1}$) meeting $\gamma_i$ and $\gamma_{i-1}$ orthogonally. Similarly, $\eta^{r}(\pi)$ is the geodesic segment homotopic to $f(a_{i,i-1})$ meeting $\gamma_i$ and $\gamma_{i+1}$ orthogonally. Note that these definitions do not depend on the choice of representative in $\pi$.

We endow the short orthogeodesics of $\pi$ with unit speed parametrizations, and from their construction, they are oriented to go from $\gamma_i$ to the other cuffs. The \emph{feet} of a short orthogeodesic $\gamma$ of $\pi$ are the unit vectors $-\eta'(0)$ and $\eta'(\ell(\eta))$. We call $\ft^{\ell} (\pi) = -(\eta^{\ell}(\pi))'(0)$ and $\ft^{r} (\pi) = -(\eta^{r}(\pi))'(0)$ respectively the left and right foot of $\pi$.

We define the \emph{half length} $\bh\bl(\gamma_i)$ of $\gamma_i$ to be the complex distance between lifts of $\eta_{i,i-1}$ and $\eta_{i,i+1}$ to $\bH^3$ that differ by a positively oriented segment of $\gamma$ joining $\eta_{i,i-1}$ to $\eta_{i,i+1}$. It turns out that $\bl(\gamma) = 2\bh\bl(\gamma)$ (see \cite{KW}, Section 2.8).

The unit normal bundle  $\N^1 (\gamma)$ to a oriented closed geodesic $\gamma$ in $M$ is acted upon simply and transitively by the group $\bC / (\bl (\gamma) + 2\pi i \bZ).$ We define $\N^1(\sqrt{\gamma})$ to be the quotient of $\N^1 (\gamma)$ by the involution $n\mapsto n + \bh\bl(\gamma)$. This is acted upon simply and transitively by $\bC/(\bh\bl(\gamma) + 2\pi i \bZ)$.

As $\bh\bl(\gamma) = \bl(\gamma)/2$, the left and right feet of $\pi \in\tpants (\gamma)$ turn out to define the same point in $\N^1(\sqrt{\gamma})$. We thus have a map
\[
\ft : \tpants(\gamma) \arr \N^1\lef(\sqrt{\gamma}\ri)
\]
which assigns the pants in $\pi$ to its foot in $\N^1 (\sqrt{\gamma})$. This map is also well defined on the unoriented version $\pants(\gamma)$.

\begin{figure}
\includegraphics[scale=0.07]{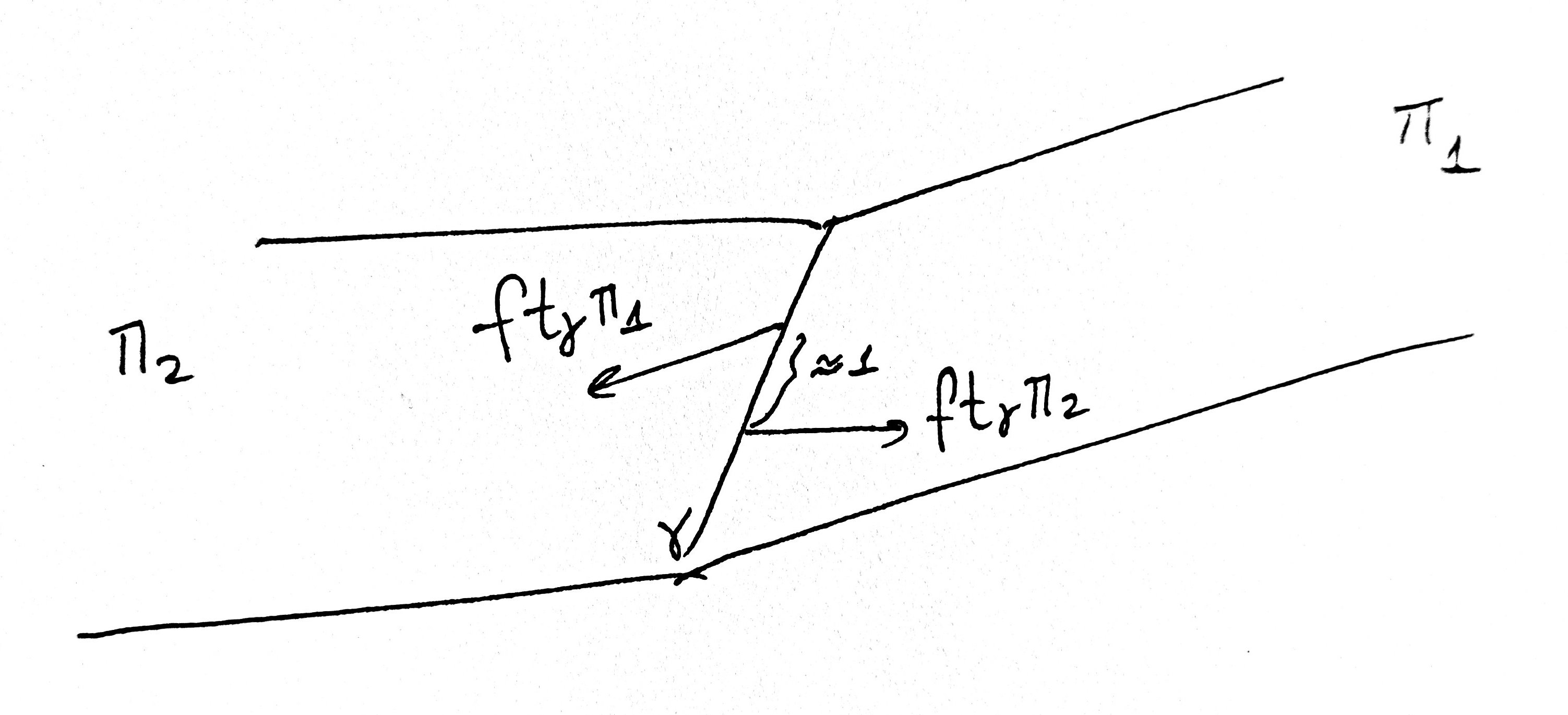}
\caption{A good gluing between $\pi_1$ and $\pi_2$.}
\end{figure}

Two pants $\pi_1, \pi_2 \in \tpants(\gamma)$ that induce opposite orientations on $\gamma$ are $(\epsilon,R)$-\emph{well matched} or \emph{well glued} along $\gamma\in \curves$ if
\[
\dist_{\N^1(\sqrt{\gamma})} (\ft\pi_1, \tau(\ft\pi_2)) < \frac{\eps}{R},
\]
where $\tau$ is the translation of $\N^1(\sqrt{\gamma})$ given by $\tau(x) =x +  1 + i\pi$.
In other words, the shearing between the feet is always approximately one (and the $i\pi$ takes into account that they point toward nearly opposite directions). Heuristically, the nearly constant shearing ensures you are never gluing the thin part of a pants (near the short orthogeodesics) to the thin part of another pants repeatedly.

For a finite set $X$, we let $\sM(X)$ be the space of measures on $X$ that are valued on the nonnegative integers. For $\mu\in\sM( X)$, we let $\sS(\mu)$ be the multiset consisting of $\mu(x)$ copies of each $x\in X$. For $\mu\in \sM(\pants)$, we let $\til{\mu}\in \sM(\tpants)$ denote the measure so that $\til{\mu}(\til{\pi}) = \mu(\pi)$ for any end $\til{\pi}$ of $\pi$. For $\gamma\in \curves$ and $\sF$ a multiset of elements of $\tpants$, we let $\sF_{\gamma}$ consists of the elements of $\sF$ that also lie in $\tpants(\gamma)$. The multiset $\sF_{\gamma}$ decomposes into a disjoint union $\sF^-_{\gamma}\sqcup \sF^+_{\gamma}$ of the ends reversing and preserving the orientation of $\gamma$. There is a map
\[
\partial: \sM(\pants) \arr \sM(\curves)
\]
defined via $\partial \mu (\gamma) = \sum_{\pi\in \tpants(\gamma)} \mu(\pi)$.

We say a surface is \emph{built out of} $\mu$ if it is obtained from gluing the elements of a submultiset of ends $\sF \se \sS (\til{\mu})$ via bijections $\sigma_{\gamma}:\sF^-_{\gamma}\to\sF^+_{\gamma}$ for every cuff $\gamma\in\supp\partial\mu$. A surface is $(\eps,R)$-\emph{well built} out of $\mu$ if all the gluings are $(\eps,R)$-good.

\subsection{Assembling the surface} The first step in the construction is to show that a surface made out of good pants glued via good gluings is essential and nearly Fuchsian. Precisely,

\begin{thm}\label{good}
For $R>0$ large enough and $\eps>0$ small enough, the following holds. Let $\mu\in\sM(\pants)$ be so that a closed surface $S$ may be $(\eps,R)$-well built from $\mu$. Then, $S$ is essential and $(1+O(\eps))$-quasifuchsian.
\end{thm}

\begin{proof}
The proof of this is long and is the content of Section 2 of \cite{KM}. A more concise proof that $\rho$ is $K(\epsilon)$-quasifuchsian, with $K(\epsilon)\to 1$ as $\epsilon\to 0$ (without the quantitative statement that $K(\epsilon) = 1+O(\epsilon)$) can be found in the appendix of \cite{KW}.
\end{proof}

It remains to find such a measure $\mu\in \sM(\pants)$ from which we can build a closed surface with good gluings. The matching theorem below tells us we can take $\mu$ to be the measure $\mu_{\epsilon,R}$ that gives weight 1 to each $\pi\in\pants$.

\begin{thm}\label{matching}
For $\epsilon>0$ sufficiently small, there is $R\geq R_0(\eps)$ so if $\gamma\in\curves$, there is a bijection
\[
\sigma_{\gamma} : \pants^- (\gamma) \arr \pants^+(\gamma)
\]
with the property that $\pi$ is $(\epsilon,R)$-well matched to $\sigma_{\gamma}(\pi)$ for all $\pi\in\pants^-(\gamma)$.
\end{thm}

Gluing the pants in $\tpants(\gamma)$ via $\sigma_{\gamma}$ for every $\gamma$ gives us a closed surface whose components, by Theorem \ref{good}, are essential and $(1+O(\epsilon))$-quasifuchsian.

The crucial ingredient in the proof of the matching theorem is the fact that the feet of pants in $\pants$ are well distributed in $\N^1(\sqrt{\gamma})$ for every $\gamma \in \curves$. This is the content of the equidistribution theorem below.

\begin{figure}
\includegraphics[scale=0.12]{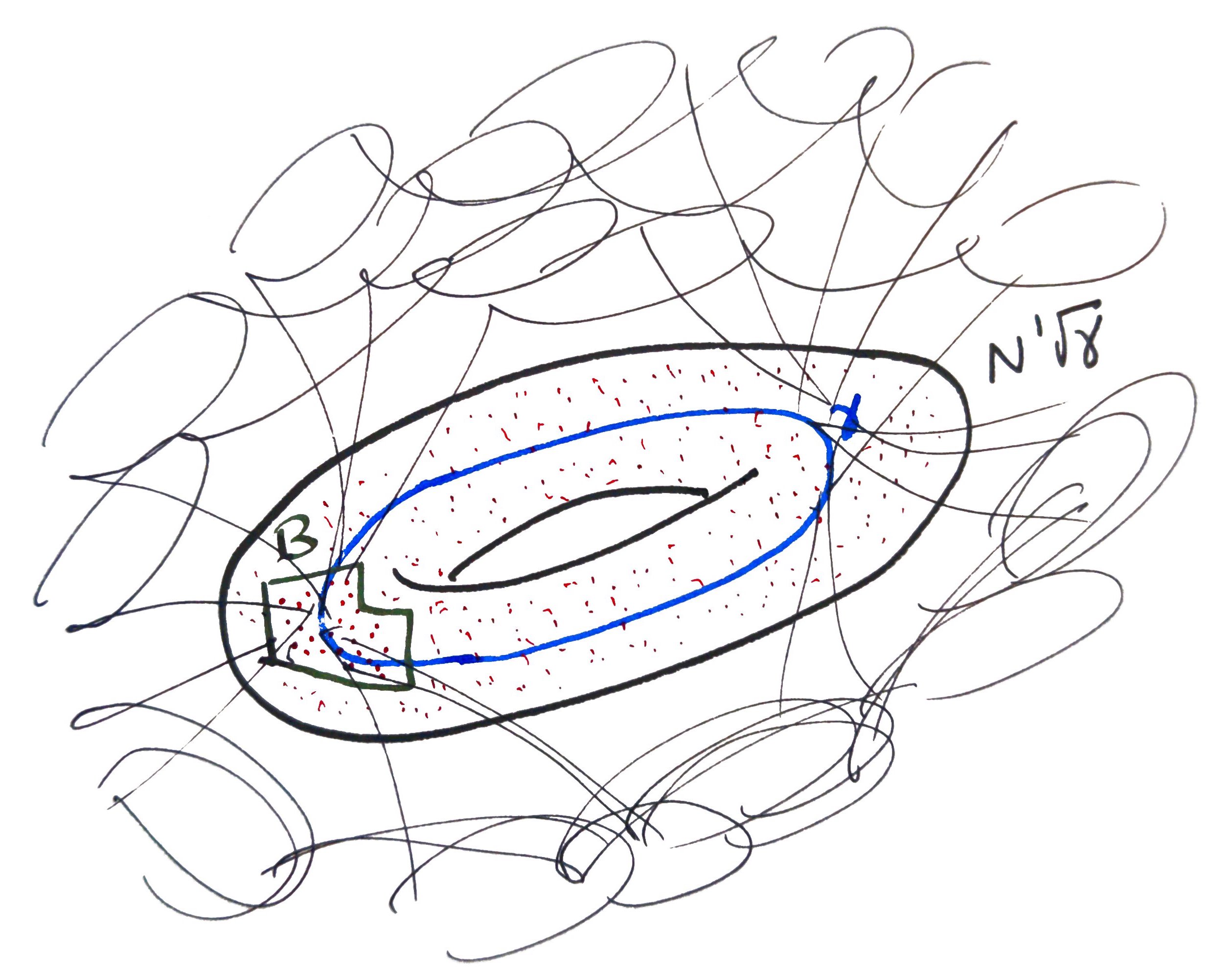}
\caption{The feet of the good pants with cuff $\gamma$ are well distributed in $\N^1(\sqrt{\gamma})$.}
\end{figure}

\begin{thm}[Equidistribution of feet]\label{kw}
There is $q=q(M)>0$ so that if $\eps>0$ is small enough and $R > R_0(\eps)$, the following holds. Let $\gamma\in \curves$. If $B\se \N^1 (\sqrt{\gamma})$ is measurable, then
\[
(1-\delta) \lambda (N_{-\delta} B) \leq
\frac{\#\{ \pi \in \pants(\gamma) : \ft \pi \in B\}}
{C(\eps,R,\gamma)}
\leq 
(1+\delta) \lambda(N_{\delta} B),
\]
where $\lambda=\lambda_{\gamma}$ is the probability Lebesgue measure on $\N^1(\sqrt{\gamma})$, $\delta = e^{-qR}$, $N_{\delta}(B)$ is the $\delta$-neighborhood of $B$, $N_{-\delta}(B)$ is the complement of $N_{\delta} (\N^1(\sqrt{\gamma}) - B)$ and $C(\eps,R,\gamma)$ is a constant depending only on $\eps$, $R$ and $\bl(\gamma)$.
\end{thm}

\begin{proof}
The proof of the equidistribution of feet is the content of \cite{KW2}. The main engine is the mixing of the frame flow in $\Fr M$. We use a slight generalization of this theorem in Sections 5 and 6, which is explained then.
\end{proof}

To complete the exposition, we will include a proof of the matching theorem using the equidistribution of feet along a curve. This is a relatively short argument which uses the Hall marriage theorem of combinatorics. Before stating it, we fix some notation: in a graph $X$, we write $v\sim w$ for two vertices $v$ and $w$ that are connected by an edge. For a set $A$ of vertices, we let $\partial N_1 (A)$ be the vertices $w\notin A$ satisfying $w\sim v$ for some $v\in A.$

\begin{thm}[Hall marriage]
Suppose $X$ is a bipartite graph, i.e., the vertices $V$ of $X$ are the disjoint union of $V_1$ and $V_2$, where no two elements of a given $V_i$ are connected by an edge. Then, there is a \emph{matching} $m:V_1\to V_2$, namely an injection so that $v\sim m(v)$ if and only if
\[
\# A \leq \# \partial N_1 (A)
\]
for any finite $A \se V_1$.
\end{thm}

We will also use the following fact. Let $M^d$ be a $d$-dimensional Riemannian manifold with Riemannian volume measure $|\cdot|$ and $(d-1)$-dimensional Hausdorff measure $\cH^{d-1}$. We define the \emph{Cheeger constant} $h(M)$ of $M$ to be the infimum of $\cH^{d-1}(\partial A)/|A|$, taken over all measurable subsets of $M$ satisfying $|A|\leq |M|/2$.

\begin{prop}
Suppose a measurable $A\se M$ satisfies $|N_{\eta} (A)| \leq |M|/2$ for some $\eta>0$. Then,
\[
\frac{|N_{\eta}(A)|}{|A|}\geq 1 + \eta h(M).
\]
\end{prop}

\begin{proof}
One incarnation of the coarea formula (\cite{N}, Theorem 5.3 or \cite{N}, page 77) says that if $f:M\to \bR$ is a Lipschitz function and $B\se M$ is measurable, then
\[
\int_B |\nabla f| = \int_{-\infty}^{\infty} \cH^{d-1} (B\cap f^{-1} (t))\, dt.
\]

Let us take $f(\cdot) = \dist(\cdot, A)$, which is a 1-Lipschitz function with $|\nabla f| = 1$ almost everywhere in the complement of $A$ and $B = N_{\eta}(A) - A$. The coarea formula then specializes to

\[
|N_{\eta}A - A| = \int_0^{\eta} \cH^{d-1} (\partial N_t A) \,dt.
\]
Since $\cH^{d-1}(\partial N_t A) \geq h(M) |A|$, the above equation tells us $|N_{\eta} A - A| \geq \eta h(M) |A|$.
\end{proof}

The Cheeger constant of the flat torus $\N^1(\sqrt{\gamma})$ satisfies $h\lef(\N^1(\sqrt{\gamma})\ri) > 1/R$ (see \cite{HHM}). Thus, if $A\se \N^1(\sqrt{\gamma})$ satisfies $\lambda(N_{\eta} (A)) \leq 1/2$, we have
\begin{equation}\label{Cheeger}
\frac{\lambda(N_{\eta}(A))}{\lambda(A)}> 1 + \frac{\eta}{R}.
\end{equation}

We also define
\[\Ft B := \#\{\pi\in \pants(\gamma): \ft \pi\in B \}.\]

The inequality \ref{Cheeger}, together with the equidistribution of feet gives us

\begin{lem}\label{matchlem}
Let $B\se \N^1 (\sqrt{\gamma})$ and let $\rho:\N^1(\sqrt{\gamma})\to \N^1(\sqrt{\gamma})$ be a translation. Then,
\[
\Ft B \leq \Ft \rho \lef(N_{\eps/R} B\ri).
\]
\end{lem}

\begin{proof}

From the equidistribution of feet and the fact that $\rho$ is measure preserving, we have that
\[
(1-\delta) \lambda\lef( N_{\epsilon/R - \delta} B \ri) \leq \frac{\Ft \rho (N_{\epsilon/R} B)}{C_{\eps,R,\gamma}}.
\]
Thus, it suffices to show that
\begin{equation}\label{suff}
\frac{\Ft B}{C_{\eps,R,\gamma}} \leq (1-\delta)\, \lambda\lef( N_{\epsilon/R - \delta} B \ri).
\end{equation}

Using the equidistribution of feet again, we have that
\[
\frac{\Ft B}{C_{\eps,R,\gamma}}  \leq (1+\delta) \lambda\lef( N_{\delta} B\ri).
\]
This reduces our goal to showing
\begin{equation}\label{suff2}
\lambda\lef( N_{\delta} B \ri) \leq 
\frac{1-\delta}{1+\delta}\, \lambda\lef(N_{\eps/R - \delta} B \ri).
\end{equation}

Suppose now that $\lambda(N_{\epsilon/2R} B) \leq 1/2$. From equation \ref{Cheeger}, we have that
\[\lambda\lef(N_{\delta} B\ri) <
\frac{1}{1 + \lef(\frac{\epsilon}{2R} - \delta \ri)\frac{1}{R}} \lambda\lef( N_{\epsilon/2R} B\ri).
\]But if $R$ is large enough, as $\delta = e^{-qR}$, we have\footnote{\[\frac{1}{1 + \lef(\frac{\epsilon}{2R} - \delta \ri)\frac{1}{R}} \leq \frac{1}{1 + \frac{\eps}{3R^2}} \leq 1 - \frac{\eps}{6R^2} \leq 1 - 3\delta \leq \frac{1 - \delta}{1+\delta}.\]}

\[
\frac{1}{1 + \lef(\frac{\epsilon}{2R} - \delta \ri)\frac{1}{R}} \leq \frac{1 - \delta}{1+\delta}.
\]

Thus, we conclude that \ref{suff2} holds if $\lambda(N_{\eps/R} B) \leq 1/2.$ In particular,
\[
\Ft B \leq \Ft \rho (N_{\eps/R} B)
\]
holds in this case.

On the other hand, if $\lambda(N_{\epsilon/2R}B)>1/2$, let $C = \N^1(\sqrt{\gamma}) - N_{\epsilon/R} \rho(B)$. Then, $\lambda(N_{\epsilon/2R} C)  \leq 1/2$ and so by the same argument as above, for $C$ instead of $B$ and $\rho^{-1}$ instead of $\rho$, we have
\[
\Ft C \leq \Ft \rho^{-1} (N_{\epsilon/R}  C).
\]
But $\Ft C = \Ft \N^1 (\sqrt{\gamma}) - \Ft \rho(N_{\eps/R} B)$ and $\Ft \rho^{-1} (N_{\epsilon/R} C) = \Ft \N^1(\sqrt{\gamma}) - \Ft B$. Therefore,
\[
\Ft B \leq \Ft \rho (N_{\eps/R} B),
\]
in this case as well. This completes the proof of the lemma.
\end{proof}

\begin{proof}[Proof of the matching theorem]
For $\gamma\in\curves$, we can make $\tpants(\gamma)$ into a graph by saying that $\pi_1\sim\pi_2$ if $\pi_1$ and $\pi_2$ are $(\epsilon,R)$-well matched, namely if they induce opposite orientations on $\gamma$ and $\dist_{\N^1(\sqrt{\gamma})} (\ft\pi_1,\tau(\ft\pi_2)) < \eps/R$, where $\tau:\N^1(\sqrt{\gamma})\to\N^1(\sqrt{\gamma})$ is the translation $\tau(x) = x+1+i\pi$. Since only the pants inducing opposite orientations on $\gamma$ may be matched, $\tpants(\gamma) = \pants^-(\gamma)\sqcup\pants^+(\gamma)$ is a bipartite graph. We wish to show there is a matching
\[
\sigma_{\gamma} : \pants^-(\gamma) \arr \pants^+(\gamma).
\]

By the Hall marriage theorem, it suffices to show that for $A\se \pants^-(\gamma)$,
\begin{align*}
\# A &\leq \partial N_1 (A) \\
&= \# \{\pi^+ \in \pants^+(\gamma) : 
|\ft \pi^+ - \tau(\ft \pi^-) | < \epsilon/R \text{ for some }\pi^-\in A \} \\
&= \Ft \tau(N_{\epsilon/R} \ft A).
\end{align*}
This, in turn, follows from Lemma \ref{matchlem} for $B = \ft A$ and $\rho = \tau$.

Since the sets $\pants^-(\gamma)$ and $\pants^+(\gamma)$ are finite and have the same cardinality, it follows that $\sigma_{\gamma}$ is a bijection, which concludes the proof of the matching Theorem \ref{matching}.
\end{proof}

In summary, we have shown the matching Theorem \ref{matching}, which allows us to build a closed $(1+O(\eps))$-quasifuchsian surface $S_{\eps,R}$ in $M$ by gluing one copy of each pants in $\pants$ via $(\eps,R)$-good gluings.

\section{Connected surfaces going through every good pants}

Recall that $\mu_{\eps,R}\in \sM(\pants)$ is the measure so that $\mu_{\eps,R} (\pi) = 1$ for each $\pi\in\pants$. In the previous section, we have seen that a closed, oriented, essential and $(1+O(\eps))$-quasifuchsian surface $S_{\eps,R}$ may be built from $\mu_{\eps,R}$. We do not know, however, whether $S_{\eps,R}$ is connected, or what its components may look like.  Following ideas of Liu and Marković \cite{LM}, we show that if we take $N=N(\eps,R,M)$ copies of $S_{\eps,R}$, it is possible to perform cut-and-paste surgeries around certain good curves in order to get \emph{connected} closed, oriented, essential and $(1+O(\eps))$-quasifuchsian surfaces $\hat{S}_{\eps,R}$.

\begin{thm}
There is an integer $N = N(\eps,R,M) > 0$ so that a \emph{connected}, closed, oriented, essential and $(1+O(\eps))$-quasifuchsian surface may be built from $N\mu_{\eps,R}$.
\end{thm}

We define a measure $\mu\in\sM(\pants)$ to be \emph{irreducible} if for any nontrivial decomposition $\mu = \mu_1 + \mu_2$, there is a curve $\gamma\in\curves$ so that $\gamma$ lies in $\supp\partial\mu_1$ and its orientation reversal $\gamma^{-1}$ lies in $\supp\partial\mu_2$.

If $\mu$ is \emph{not} irreducible, then no connected surface may be built from $\mu$. In fact, if there is a nontrivial decomposition $\mu=\mu_1+\mu_2$ so that if $\gamma\in \supp\partial\mu_1$, then $\gamma^{-1}\notin\supp\partial\mu_2$, then no pants in $\supp\partial\mu_1$ may be glued to pants in $\supp\partial\mu_2$. Thus, a surface built out of $\mu$ will have at least two components.

On the other hand, if $\mu$ is irreducible, we have the following theorem, which is close to Lemma 3.9 of Liu and Marković \cite{LM}. (They do not assume $\mu$ to be positive on all pants, using a weaker hypothesis instead, but the conclusion is the same.)

\begin{thm}\label{irred}
Suppose $\mu\in \sM(\pants)$ is an irreducible measure so that $\mu(\pi)>0$ for every $\pi\in\pants$ and a closed surface may be $(\eps,R)$-well built from $\mu$. Then, there is an integer $N = N(\mu)$ so that a \emph{connected} closed surface may be $(2\eps,R)$-well built from $N\mu$.
\end{thm}

In view of that, our goal for this section is to prove that $\mu_{\eps,R}$ is irreducible. Fortunately, we have the following theorem, which is Proposition 7.1 of \cite{LM}.

\begin{prop}\label{conn}
Given two curves $\gamma_0,\gamma_1 \in \curves$, we may find a sequence of pants $\pi_0,\ldots,\pi_n$ in $\pants$ so that $\gamma_0$ is a cuff of $\pi_0$, $\gamma_1$ is a cuff of $\pi_n$ and $\pi_i$ may be glued to $\pi_{i+1}$ for $0\leq i <n$.
\end{prop}

The gluings in the proposition above are not necessarily $(\eps,R)$-good.%, as suggested in Figure \ref{connfig}.

\begin{proof}[Proof sketch]
For any curve $\gamma\in \curves$, it is possible to use the mixing of the frame flow to show there is a segment $\alpha$ from $\gamma$ to itself, dividing $\gamma$ into two pieces of approximately equal length and with the property that the curves homotopic to each of the bigons formed by $\gamma$ and $\alpha$ are $(\eps/10000,R)$-good. (Liu and Marković call this construction \emph{splitting} -- \cite{LM}, Construction 4.17) Thus, we can think of $\gamma_0$ and $\gamma_1$ as boundaries of pants $\pi$ and $\pi'$ whose other cuffs are in $\mathbf{\Gamma}_{\eps/10000,R}$ and for the purposes of proving this proposition, we may assume $\gamma_0$ and $\gamma_1$ are $(\eps/10000,R)$-good.

Liu and Marković then use a construciton called \emph{swapping} (\cite{LM}, Construction 4.18) to build a surface $F$ made out of pants in $\pants$ so that $\gamma_0$ and $\gamma_1$ are connected components of $\partial F$.

To explain a bit more about the swapping construction, we will need to introduce a bit of their terminology. A \emph{framed segment} is a geodesic segment in $M$ with orthonormal frames at its endpoints, where the first vectors of both frames are tangent to the segment. An \emph{$(L,\delta)$-tame cycle} is a sequence $\fs_1,\ldots,\fs_m$ of framed segments of length at least $2L$, with the property that co{}nsecutive initial and terminal frames are within $\delta$ of each other in the metric on $\Fr M$ and with $\fs_m$ ending on the initial point of $\fs_1$. The \emph{reduced concatenation} of $[\fs_1 \cdots \fs_m]$ of $\fs_1,\dots,\fs_m$ is the unique geodesic in $M$ homotopic to the concatenation of the segments. An \emph{$(L,\delta)$-tame bigon} is an $(L,\delta)$-tame cycle with two segments and an \emph{$(L,\delta)$-tame swap pair of bigons} is a pair of $(L,\delta)$-tame bigons $[\fa_-\fa_+]$ and $[\fa'_-\fa'_+]$ such that $\fa_s$ and $\fa'_s$ have complex length $\delta$-close to each other for $s = -$ or $+$ and $[\fa_-\fa_+']$ and $[\fa'_-\fa_+]$ are also $(L,\delta)$-tame bigons.

Now let $\eps>0$ be sufficiently small and $L = L(\eps,M)$ be sufficiently large. The swapping construction takes as an input a $(10L,\eps/100)$-tame swap pair of bigons $[\fa_-\fa_+]$ and $[\fa'_-\fa'_+]$ with the property that all of  $[\fa_-\fa_+]$,  $[\fa_-'\fa_+']$,  $[\fa_-'\fa_+]$ and  $[\fa_-\fa_+']$ are in $\curves$. The output is a surface $F$ built out of pants in $\pants$ so that $\partial F$ has exactly four boundary components: $[\fa_-\fa_+]$,  $[\fa_-'\fa_+']$, as well as the orientation reversals of  $[\fa_-'\fa_+]$ and  $[\fa_-\fa_+']$.

A word on the proof of this construction: in the simple case when $\fa_-$ and $\fa_+$ (hence also $\fa'_-$ and $\fa'_+$) have approximately the same length, it is possible to using mixing to draw a segment $\fm$ from the initial to the final point of $\fa_-$ with the property that all of $[\fa_s \fm]$ and $[\fa'_s\fm]$ are in $\curves$, for $s \in \{-,+\}$. These curves, together with $[\fa_-\fa_+]$,  $[\fa_-'\fa_+']$, $[\fa_-'\fa_+]$ and  $[\fa_-\fa_+']$, bound four  pants in $\pants$, which we glue together to get $F$. The general case requires a more intricate construction, and the resulting $F$ is made from 12 pants in $\pants$.

Now let us go back to our curves $\gamma_0$ and $\gamma_1$ in $\mathbf{\Gamma}_{\eps/10000,R}$ and explain how to use swapping to build a bridge of pants in $\pants$ between them. Decompose $\gamma_0$ into segments $\fa_-$ and $\fa_+$, with $\fa_-$ of length $R/2$ and $\gamma_1$ into segments $\fb_-$ and $\fb_+$ with $\fb_+$ of length $R/2$. Using mixing of the frame flow, it is possible to draw segments $\fs$ and $\frt$ of complex length $\eps/10000$-close to $R/2$ so that $\fa_- ,\fs ,\fb_+ ,\frt$ is a $(100,\eps/100)$-tame cycle (see Figure \ref{swapfig}). Moreover, it is possible to use the \emph{length and phase formula} (\cite{LM}, Lemma 4.8) to show that the reduced concatenation $\gamma' = [\fa_-\fs\fb_+\frt]$ is in $\mathbf{\Gamma}_{\eps/100,R}$. Moreover, $[\fa_-\fa_+]$ and $[\fa_-(\fs\fb_+\frt)]$ form a $(100,\eps/100)$-tame swap pair of bigons, so using swapping we can build a surface $F_1$ out of pants in $\pants$ so $\gamma_0$ and $\gamma'$ are components of $\partial F_1$. Similarly, $[\fb_-\fb_+]$ and $[(\frt \fa_-\fs)\fb_+]$ are a $(100,\eps/100)$-tame swap pair of bigons, so swapping gives us a surface $F_2$ built out of pants in $\pants$ so $\gamma'$ and $\gamma_1$ are components of $\partial F_2$. By gluing $F_1$ and $F_2$ along $\gamma'$, we obtain a bridge of $(\eps,R)$-good pants between $\gamma_0$ and $\gamma_1$.
\end{proof}

\begin{figure}
\includegraphics[scale=0.08]{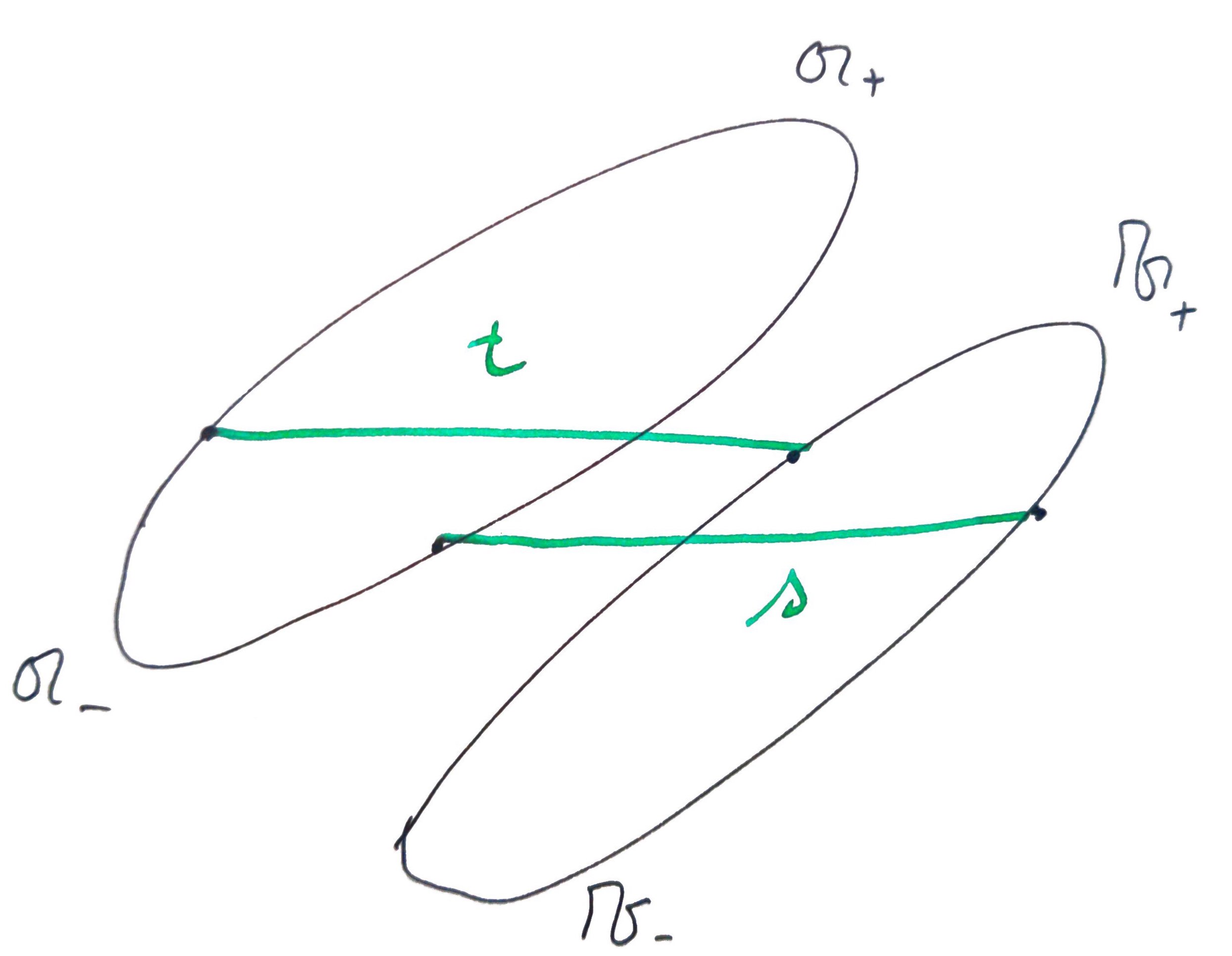}
\caption{Proof sketch of Proposition \ref{conn}: joining $\gamma_0$ and $\gamma_1$ via segments $\fs$ and $\frt$.}
\label{swapfig}
\end{figure}

\begin{thm}
The measure $\mu_{\eps,R}$ is irreducible.
\end{thm}

\begin{proof}
Let $\mu_{\eps,R} = \mu_0 + \mu_1$ be a nontrivial decomposition. Let $\gamma_0\in \supp\partial\mu_0$ and $\gamma_1\in\supp\partial\mu_1$. In view of Proposition \ref{conn}, there are pants $\pi_0,\ldots,\pi_n$ in $\pants = \supp\mu_{\eps,R}$ so that $\gamma_0$ is a cuff of $\pi_0$, $\gamma_1$ is a cuff of $\pi_n$ and $\pi_i$ may be glued to $\pi_{i+1}$. This means there is a curve $\gamma$, which is a cuff of some $\pi_i$, so that $\gamma\in\supp\mu_1$ and $\gamma^{-1}\in\supp\mu_2$. This means $\mu_{\eps,R}$ is irreducible.
\end{proof}

We conclude the section by providing a proof of Theorem \ref{irred}. The regluing of surfaces featured in this proof provides inspiration for the construction of non-equidistributing surfaces in Section 7.

We start with the following lemma about pants decompositions.

\begin{lem}\label{double}
Let $S$ be a surface with a pants decomposition $P$. Then, $S$ has a double cover $\hat{S}$ to which the pants in $P$ lift homeomorphically to pants with nonseparating cuffs.
\end{lem}

\begin{figure}
\includegraphics[scale=0.1]{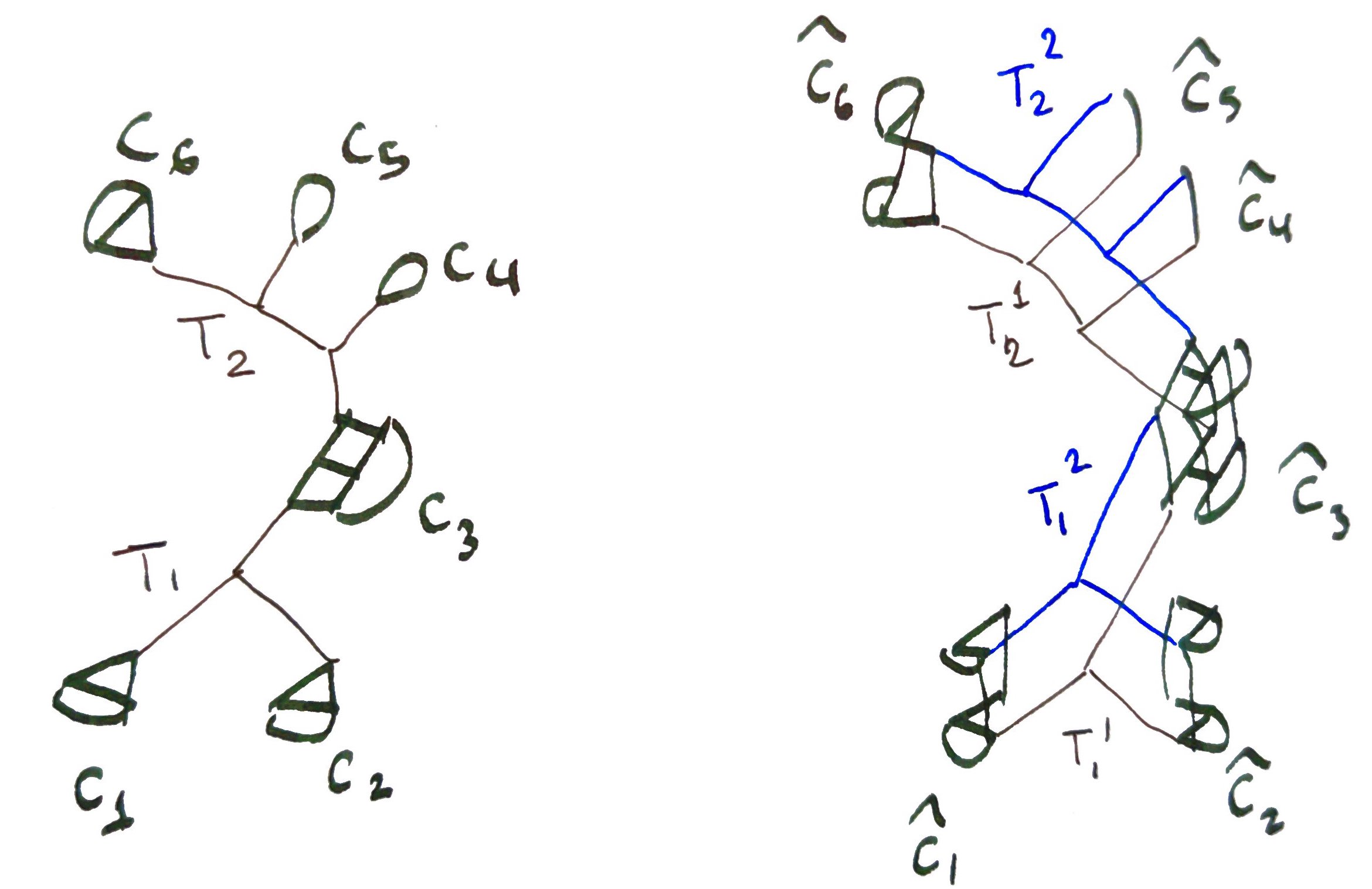}
\caption{Proof of Lemma \ref{double}. On the left, we have the dual graph $X$ to the pants decomposition $P$ of $S$. On the right, we have the double cover $\hat{X}$. $T_i^1$ and $T_i^2$ are the lifts of $T_i$ to $\hat{X}$, which are nonseparating in $\hat{X}$.}
\label{doublecover}
\end{figure}

\begin{proof}
Let $X = X(S,P)$ be the 3-regular graph whose vertices are pants in $P$ and the edges are the cuffs shared by pants in $P$. Conversely, given a 3-regular graph $X$, we can obtain a surface $S$ with pants decompostion $P$ by drawing pants-shaped tubes around a neighborhood of each vertex in $X$ and identifying their cuffs along the edges of $X$. Thus, given a degree $d$ finite cover $\hat{X}$ of $X(S,P)$, we get a surface $\hat{S}$ with pants decomposition $\hat{P}$, which is a degree $d$ cover of $S$
. Moreover, the pants in $P$ lift homeomorphically to pants in $\hat{P}$.

A cuff is separating in $S$ if and only if its corresponding edge in $X(S,P)$ is separating. Thus, our task is to show that $X$ has a double cover $\hat{X}$ that only has nonseparating edges.

To do so, let $F= \sqcup_{i=1}^n T_i \se X$ be the graph-theoretic forest consisting of all separating edges of $X$, where the $T_i$ are disjoint trees. We also write $X - F = \sqcup_{j=1}^m C_j$, where $C_j$ are disjoint connected components. For each $j$, we take a double cover $d_j:\hat{C}_j\to C_j$. This gives us a double cover $d:\sqcup_{j=1}^m \hat{C}_j\to \sqcup_{j=1}^m C_j$.

Note that each $\hat{C_j}$ consists of nonseparating edges. If some $\hat{C_j}$ had a separating edge $e$, it would have another separating edge $e'$, the image of $e$ under the nontrivial deck transformation $\hat{C_j}\to\hat{C_j}$. Thus, $\hat{C_j}-\{e,e'\}$ consists of \emph{three} components, otherwise one of $e$ or $e'$ would not be separating. In particular, the inverse image of the (connected) set $C_j-d_j(e)$ under $d_j$ would consist of three components, contradicting the fact that $d_j$ is a double cover.

For each $i$, we attach a copy of $T_i$ to each of the two lifts that $\partial T_i$ has in $\sqcup_{j=1}^m \hat{C}_j$. As a result, we get a double cover $D:\hat{X}\to X$ which extends $d$. (See Figure \ref{doublecover}.) The trees $T_i\se X$ have nonseparating lifts to $\hat{X}$, as both of their lifts are bounded by the same subset of the $\{\hat{C}_j\}_{j=1
}^m$.
\end{proof}

\begin{figure}
\includegraphics[scale=0.07]{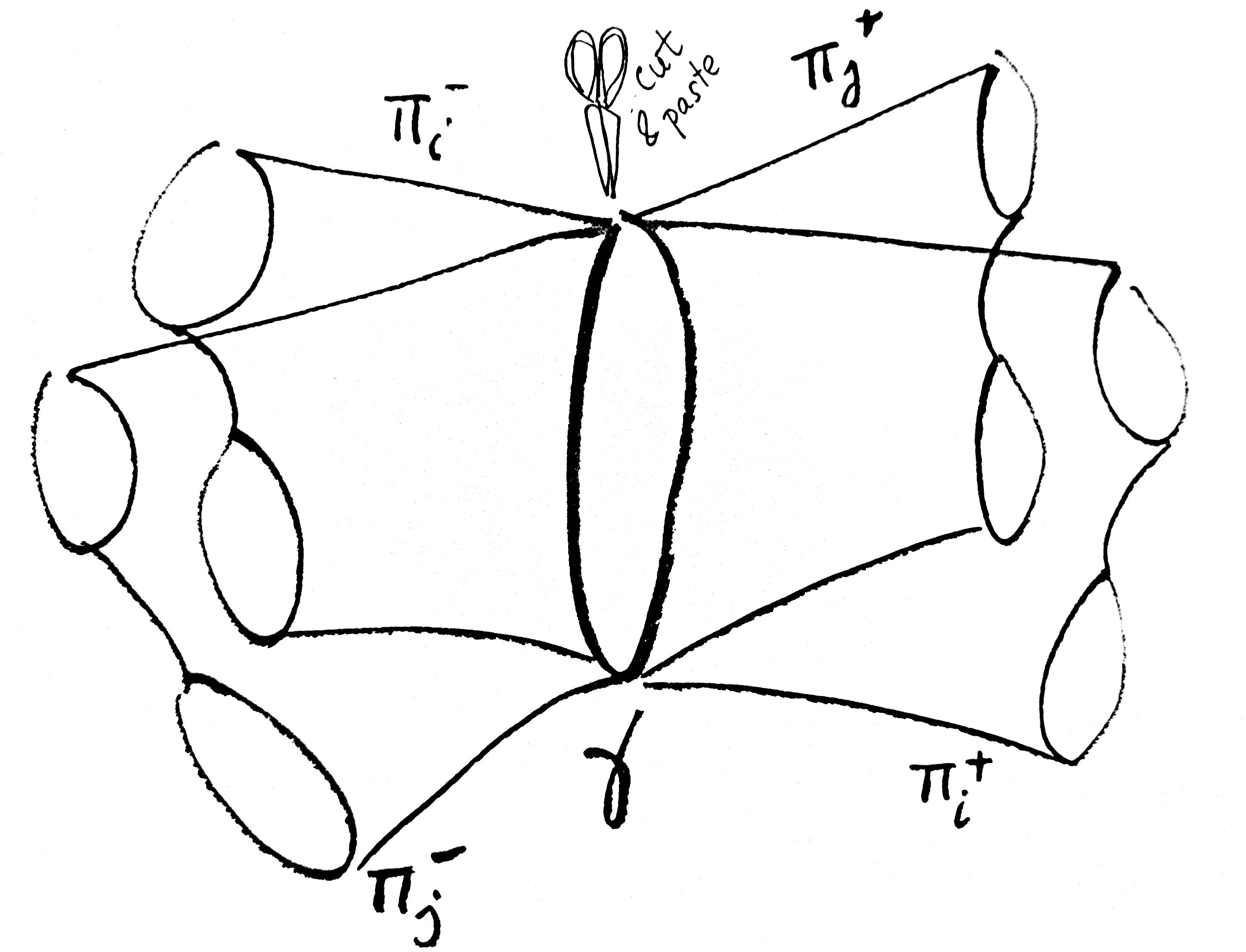}
\caption{Regluing $\hat{S}_i$ to $\hat{S}_j$ with a good gluing and reducing the number of components of $\hat{S}$ while making sure the result is still nearly Fuchsian.}
\label{reglue}
\end{figure}

\begin{proof}[Proof of Theorem \ref{irred}]
Suppose the closed oriented essential $(1+O(\eps))$-quasifuchsian surface $S$ we build out of $\mu_{\eps,R}$ has $r$ components:
\[
S = S_1\sqcup \cdots \sqcup S_r.
\]

We take a finite covering $\hat{S} = \hat{S}_1\sqcup \cdots \sqcup\hat{S}_r$ of $S$ of degree $r$ so that the good pants lift homeomorphically and and with the property that if a cuff $\gamma \in \curves$ appears in $S_j$, then it appears at least $r$ times in $\hat{S}_j$, for $1\leq j \leq r$. In view of Lemma 4.5, we can take a double cover of $\hat{S}$ to which good pants lift homeomorphically to good pants with nonseparating cuffs. This lets us assume all cuffs in $\hat{S}$ are nonseparating.

The surface $\hat{S}$ is built out of a multiple $N\mu_{\eps,R}$ of $\mu_{\eps,R}$, where $N = 2r$. Below we will explain how we can make $r$ cut and paste operations on cuffs of $\hat{S}$ in order to get a \emph{connected} closed $(1+O(\eps))$-quasifuchsian surface also built out of $N\mu_{\eps,R}$.

For $\gamma\in \curves$, let $\bPi_k(\gamma)$ be the ends of pants in $\tpants(\gamma)$ that lie in $\hat{S}_k$. We can divide $\bPi_k(\gamma) = \bPi_k^-(\gamma) \sqcup \bPi_k^+(\gamma)$ that induce a negative and positive orientation on $\gamma$. 

Consider the graph $X$ whose vertices $v_1,\ldots,v_r$ are in correspondence with the components $\hat{S}_1,\ldots,\hat{S}_r$ of $\hat{S}$. Two vertices $v_i$ and $v_j$ in $X$ are connected by an edge when $\hat{S}_i$ and $\hat{S}_j$ share a common cuff $\gamma\in\curves$ so there are pants $\pi_i^-\in \bPi_i^-(\gamma)$ and $\pi_j^-\in \bPi_j^-(\gamma)$ satisfying
\[
| \ft_{\gamma} \pi_i^- - \ft_{\gamma} \pi_j^- | < \frac{\eps}{R}.
\]

The irreducibility of $\mu_{\eps,R}$ comes in to play to show

\begin{claim*}
$X$ is connected.
\end{claim*}

\begin{proof}
Suppose by contradiction that $X_1 \se X$ is a connected component that is not empty and not all of $X$. Let $X_2 = X - X_1$.

We may decompose $N\mu_{\eps,R}$ as $N\mu_{\eps,R} = \hat{\mu}_1 + \hat{\mu}_2$, where each $\hat{\mu}_i$ has support in the pants that make up the components of $\hat{S}$ corresponding to the vertices of $X_i$. Moreover, the way in which the $\hat{S}$ was constructed tells us that $\hat{\mu}_i = N\mu_i$, where each $\mu_i$ is also a nonnegative integer-valued measure on $\pants$.

In particular, we have a nontrivial decomposition $\mu_{\eps,R} = \mu_1+\mu_2$. By the irreducibility of $\mu_{\eps,R}$ there is a cuff $\gamma \in \curves$ appearing in the good pants decompositions of some connected components of $\hat{S}$ corresponding to vertices in both $X_1$ and $X_2$.

Let $\bPi^-_{X_i}(\gamma)$ denote the ends of pants in $\pants^-(\gamma)$ coming from pants that lie in the components of $\hat{S}$ corresponding to the vertices of $X_i$, for $i = 1$ or $2$. Let $F_i := \ft_{\gamma} ( \bPi^-_{X_i}(\gamma))$.

We claim $N_{\eps/2R} (F_1) \cup N_{\eps/2R} (F_2) = \N^1 (\sqrt{\gamma})$. If we let $F = \ft_{\gamma} (\pants^-(\gamma))$, this is equivalent to saying $N_{\eps/2R} (F) = \N^1(\sqrt{\gamma})$. The latter follows from the equidistribution of feet (Theorem \ref{kw}), which tells us
\[
0 = \frac{\Ft(\N^1(\sqrt{\gamma}) - F)}{C_{\eps,R,\gamma}} \geq (1-\delta) \,\lambda\lef( N_{-\delta} (\N^1(\sqrt{\gamma}) - F) \ri),
\]
and so we see $N_{\delta}(F)$ has full measure. Since $\delta = e^{-qR} < \eps/2R$, we get the desired conclusion.

As $\N^1(\sqrt{\gamma})$ is connected, we conclude $N_{\eps/2R}(F_1) \cap N_{\eps/2R}(F_2) \neq\emptyset$. Thus we may find pants $\pi_1^-\in \bPi_{X_1}^-(\gamma)$ and $\pi_2^- \in \bPi_{X_2}^-(\gamma)$ so that
\[
\dist\, (\ft_{\gamma} \pi_1^- ,\ft_{\gamma} \pi_2^-) < \frac{\eps}{R}.
\]

This implies there is an edge between some vertex of $X_1$ and some vertex of $X_2$, contradicting the assumption that $X_1$ is a proper connected component of $X$.
\end{proof}

Let $T\se X$ be a maximal tree. For an edge $e$ of $T$ between two vertices $v_i$ and $v_j$, we select cuffs $\gamma(e, \hat{S}_i)$ lying in $\hat{S}_i$ and $\gamma(e,\hat{S}_j)$ in $\hat{S}_j$ with the following properties. As cuffs in $\curves$ they are the same $\gamma_e\in\curves$ and we may find $\pi_i^- \in \bPi^-_i(\gamma_{e})$ and  $\pi_j^- \in \bPi^-_j(\gamma_{e})$ so that $|\ft_{\gamma_{e}} \pi_i^- - \ft_{\gamma_{e}} \pi_j^-| < \eps/R$.

As both $\hat{S}_i$ and $\hat{S}_j$ are $(\eps,R)$-well built from pants $\pants$, both $\pi_i^-$ and $\pi_j^-$ are $(\eps,R)$-well glued to some $\pi_i^+$ and $\pi_j^+$ in $\pants^+(\gamma_{e})$, respectively. Namely,
\[
|\ft_{\gamma_{e}} \pi_i^- - \tau(\ft_{\gamma_{e}} \pi_i^+) | < \frac{\eps}{R}
\aand
|\ft_{\gamma_{e}} \pi_j^- - \tau(\ft_{\gamma_{e}} \pi_j^+) | < \frac{\eps}{R}.
\]
Therefore, $\pi_i^-$ may be $(2\eps,R)$-well glued to $\pi_j^+$ and $\pi_j^-$ may be $(2\eps,R)$-well glued to $\pi_i^+$.

From our hypothesis that each cuff originally in each component $S_k$ appears at least $r$ times in each $\hat{S}_k$, we can assume that all the $\gamma(e,S_k)$ are distinct as curves in the surfaces they lie in. This means we can perform all these regluings at once, without worrying about one regluing interfering in another. Thus we obtain a \emph{connected} closed $(1+O(\eps))$-quasifuchsian surface built out of $N\mu_{\eps,R}$.
\end{proof}

\section{Barycenters of the good pants are equidistributed}\label{equid}

\begin{figure}
\includegraphics[scale=0.1]{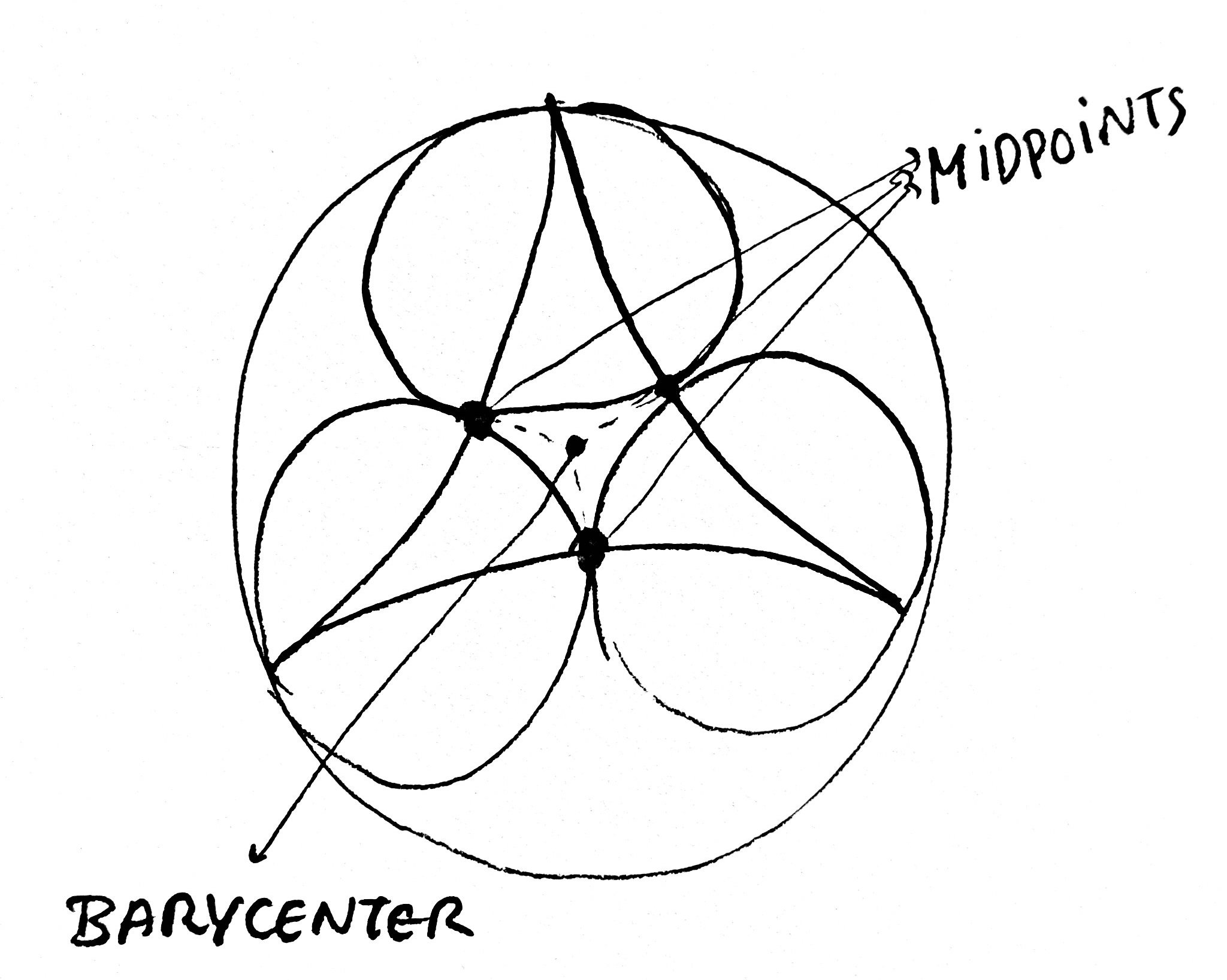}
\includegraphics[scale=0.1]{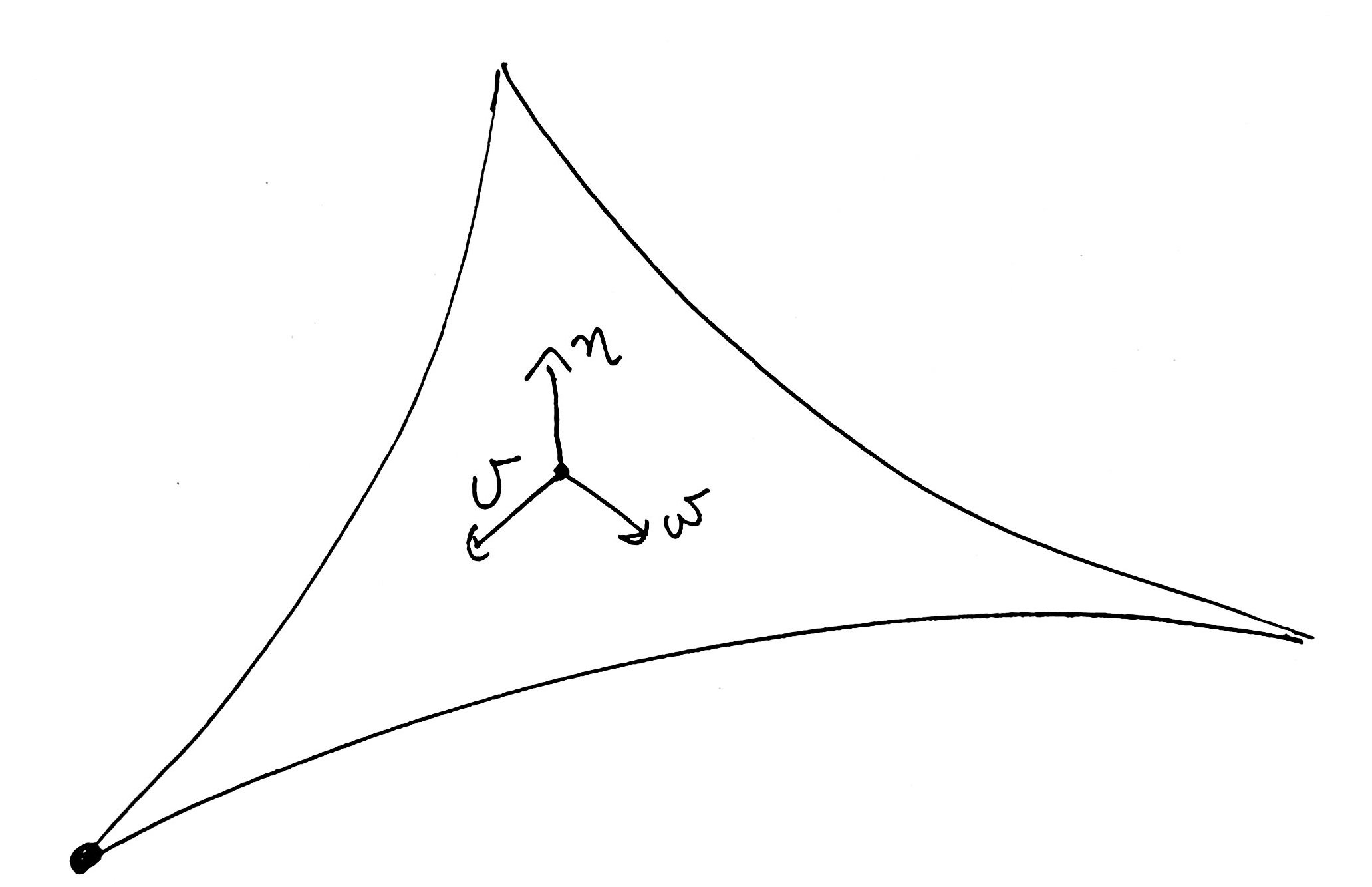}
\caption{Left: midpoints and barycenter of an ideal triangle. Right: one of the three framed barycenters of an ideal triangle.}
\end{figure}

Let $T\se \bH^3$ be an oriented ideal triangle. There are three horocycles based on the vertices of $T$ that are pairwise tangent, with their tangency points lying in $\partial T$. The points where the horocycles meet $\partial T$ are called the \emph{midpoints} of the edges of $T$.

The geodesic rays from the midpoints of $T$ towards the opposite vertices meet at the \emph{barycenter} $b(T)$ of $T$. The \emph{framed barycenters} of $T$ are the frames $(v,w,n)$ based at $b(T)$, where $v$ points away from a side of $T$, $n$ is normal to $T$ and $v\times w = n$.

The barycenter of an ideal triangle $T\se M$ is the projection onto $M$ of the barycenter of a lift of $T$ to $\bH^3$. The framed barycenters of $T\se M$ are the projections to $\Fr M$ of the framed barycenters of a lift of $T$ to $\bH^3$.

A good pants $\pi\in\pants$ has a pleated structure consisting of two ideal triangles, as in Figure \ref{twist}. Its \emph{barycenters} are the framed barycenters of these ideal triangles.

We let $\beta_{\eps,R}$ be the weighted uniform probability measure supported on the barycenters of the pants in $\pants$. In this section, we will show

\begin{thm}[Equidistribution of barycenters]\label{bary}
For $\epsilon \to 0$ and $R(\epsilon)\to \infty$ fast enough,
\[
\beta_{\eps,R(\eps)} \wkstar \nu_{\Fr M},
\]
where $\nu_{\Fr M}$ is the probability volume measure on $\Fr M$.
\end{thm}

In other words, the barycenters of the good pants equidistribute in $\Fr M$ as $\eps\to 0$ and $R(\eps)\to\infty$. 
This will be used in Section 6 to show that the connected surface $S_{\eps,R}$ built out of $N= N(\eps,R,M)$ copies of each pants in $\pants$ equidistributes as $\eps\to 0$ and $R\to\infty$. This will follow from the fact that the unit tangent bundle of each pair of pants (outside of the pleats) may be obtained from the barycenters via the right action of a set $\Delta \se \PSL_2 \bR$.

\begin{figure}
\includegraphics[scale=0.09]{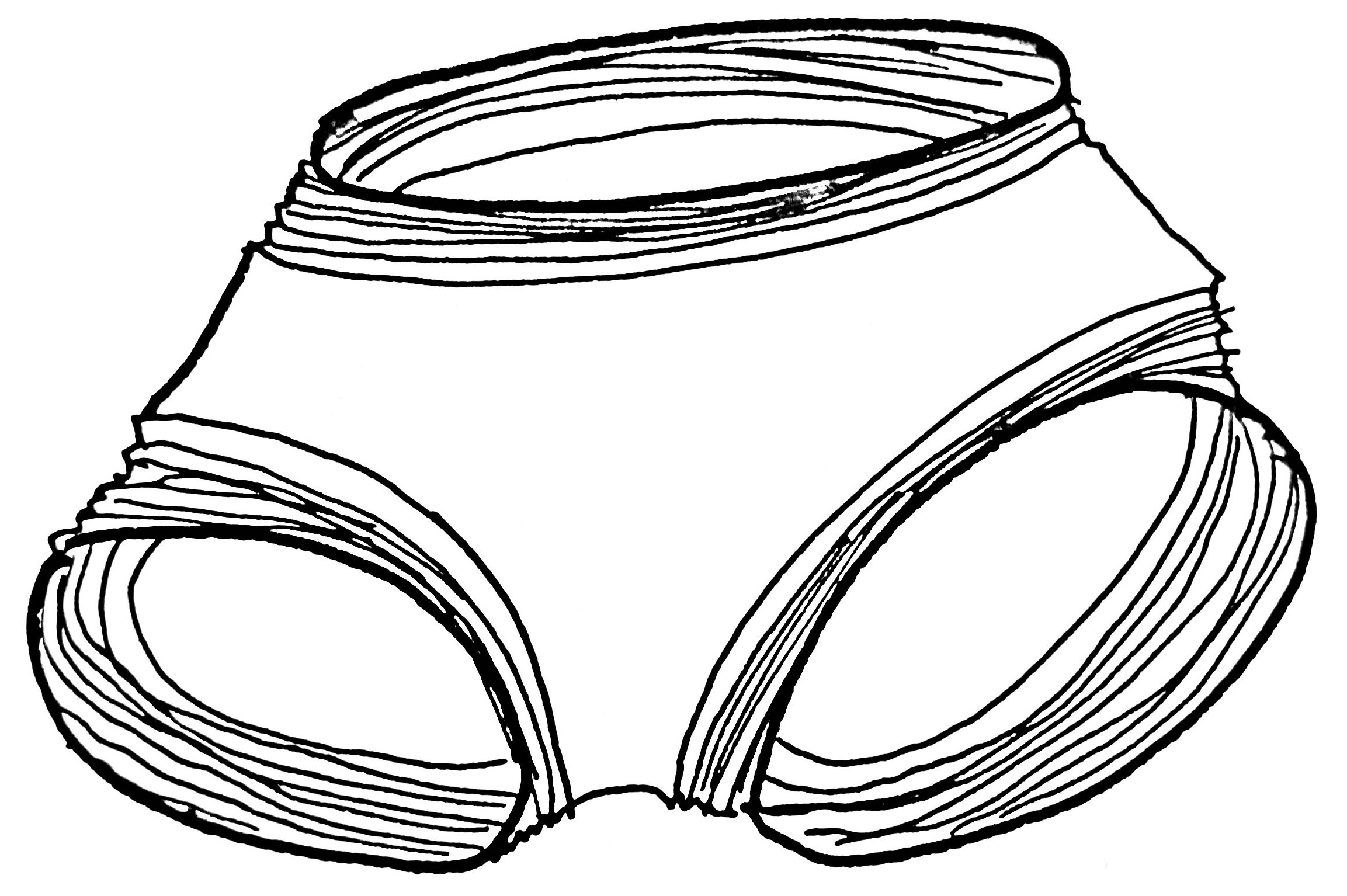}
\caption{Pleated structure of a pair of pants consisting of two ideal triangles.}
\label{twist}
\end{figure}

\subsection{Outline of the proof}

We will show the equidistribution of barycenters, Theorem 5.1, in three steps.

First, we will prove that the feet of all pants in $\pants$, seen as points in $\Fr M$, equidistributes as $\eps\to 0$. Precisely, a foot $f$ of $\pi = [(f,C_i)] \in \tpants$ is associated to the frame $(v,f,v\times f)$, where $v$ is the unit tangent vector to the $\gamma\in\curves$ homotopic to $f(C_i)$. (With this identification, we can realize $\N^1(\gamma)$ as a subset of $\Fr M$.) We let $\phi_{\eps,R}$ be the weighted uniform probability measure on $\Fr M$ supported on the feet of pants in $\pants$. We will show

\begin{lem}[Equidistribution of feet in $\Fr M$]\label{equidftfm}
For $\epsilon \to 0$ and $R(\epsilon)\to \infty$ fast enough,
\[
\phi_{\eps,R(\eps)} \wkstar \nu_{\Fr M}.
\]
\end{lem}

The proof of Lemma 5.2 will use the fact that the feet are well-distributed in the unit normal bundle of a given good curve (due to Kahn-Wright \cite{KW2}, in a modified version), as well as the fact that the good curves themselves are asymptotically almost surely well-distributed in $\T^1 M$ (due to Lalley \cite{L}).

Let $a_t = \diag(e^{t/2},e^{-t/2})$ and $k\in \SO_2$ be the ninety-degree rotation bringing the first vector in a frame to the second, fixing the third.
The second step of the proof is to observe that the right action\footnote{We choose an origin $o\in \Fr M$ and identify $\Fr M \cong \PSL_2 \bC$ by sending $go$ to $g$. We say that the right action $R_h$ of an element $h\in G$ on $go\in \Fr M$ is given by $R_h (go) = gho$. This is an antihomomorphism $R:G\to \Aut G$.} $v_R:= R_{a_{R/2} k a_{\log(\sqrt{3}/2)}}$ of the element
\[
a_{R/2} \,k\, a_{\log(\sqrt{3}/2)} \in \PSL_2 \bC,
\]
brings the feet of a pants $\pi$ to frames that are very close to the framed barycenters of the triangles of the pleated structure of $\pi$. 

We call the images of the feet of $\pants$ under $v_R$ the \emph{approximate barycenters} of the pants in $\pants$. In Lemma \ref{ab}, we show that the distances in $\Fr M$ between the approximate barycenters and the actual barycenters of pants in $\pants$ go to zero uniformly as $\eps\to 0$.

Let $\beta^a_{\eps,R}$ be the (weighted) uniform probability measure on the approximate barycenters of the pants in $\pants$. We will show that these approximate barycenters equidistribute in $\Fr M$, namely

\begin{prop}[Equidistribution of approximate barycenters]\label{equidab}
For $\eps\to 0$ and $R(\eps)\to\infty$ fast enough,
\[
\beta^{a}_{\eps,R(\eps)} \wkstar \nu_{\Fr M}.
\]
\end{prop}

To conclude, we use Lemmas \ref{ab} and \ref{equidab} to show the main theorem of the section -- the \emph{actual} barycenters of the pants equidistribute.

\begin{figure}
\includegraphics[scale=0.11]{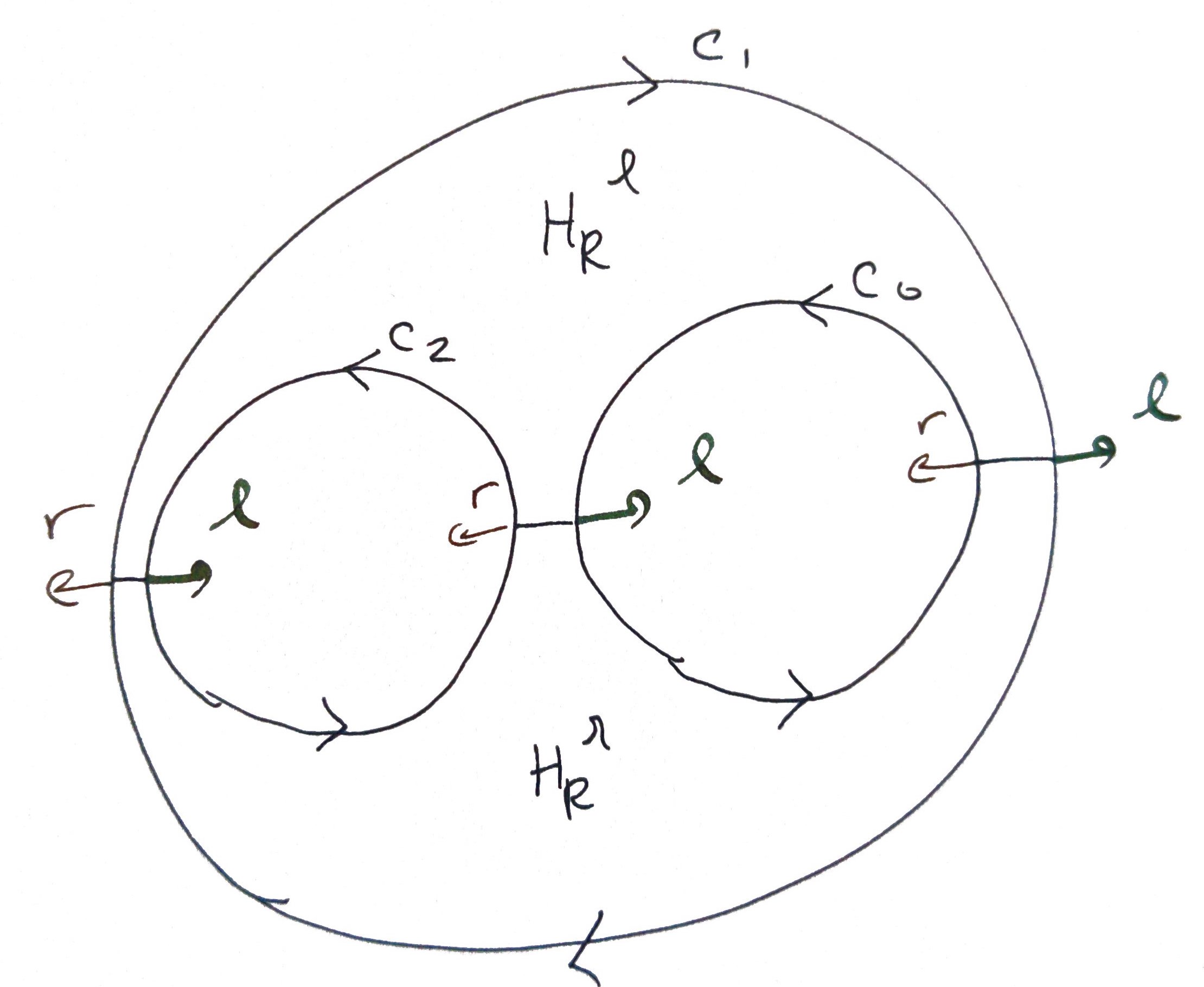}
\caption{
The pants $P_R$ divided into left and right hexagons, with its left and right feet.}
\label{leftandright}
\end{figure}

\subsection{Left and right} In this subsection we will do some bookkeeping that will be useful to carry out the rest of the proof.

Let $P_R$ be the oriented planar hyperbolic pair of pants whose cuffs have size $2R$, as defined in Section 3. The cuffs of $P_R$ are named $C_0$, $C_1$ and $C_2$, as in Figure \ref{leftandright}. As defined before, each cuff $C_i$ has two \emph{feet} in $\N^1 (C_i)$, which are unit vectors in the direction of the short orthogeodesics incident to $C_i$. The \emph{left} foot of $C_i$ points away from $C_{i-1}$ and the \emph{right} foot points away from $C_{i+1}$.

We can cut $P_R$ along its short orthogeodesics to obtain two right-angled hexagons $H_R^{\ell}$ and $H_R^r$. The \emph{left} right-angled hexagon $H_R^{\ell}$ of $P_R$ is the one so that a traveller going around $\partial H^{\ell}_R$ in the direction given by the orientation of $P_R$ sees the cuffs in the cyclic order $(C_0\,C_2\,C_1)$. The \emph{right} right-angled hexagon is the other one (associated to the cyclic order $(C_0\,C_1\,C_2)$).

As before, let $v_R$ be the right action of $a_{R/2} k a_{\log(\sqrt{3}/2)}\in \PSL_2\bR$. Observe that the image of a left foot of $P_R$ under $v_R$ falls inside $H_R^{\ell}$. Similarly, the image of a right foot under $v_R$ falls in $H_R^r$.

\begin{figure}
\includegraphics[scale=0.08]{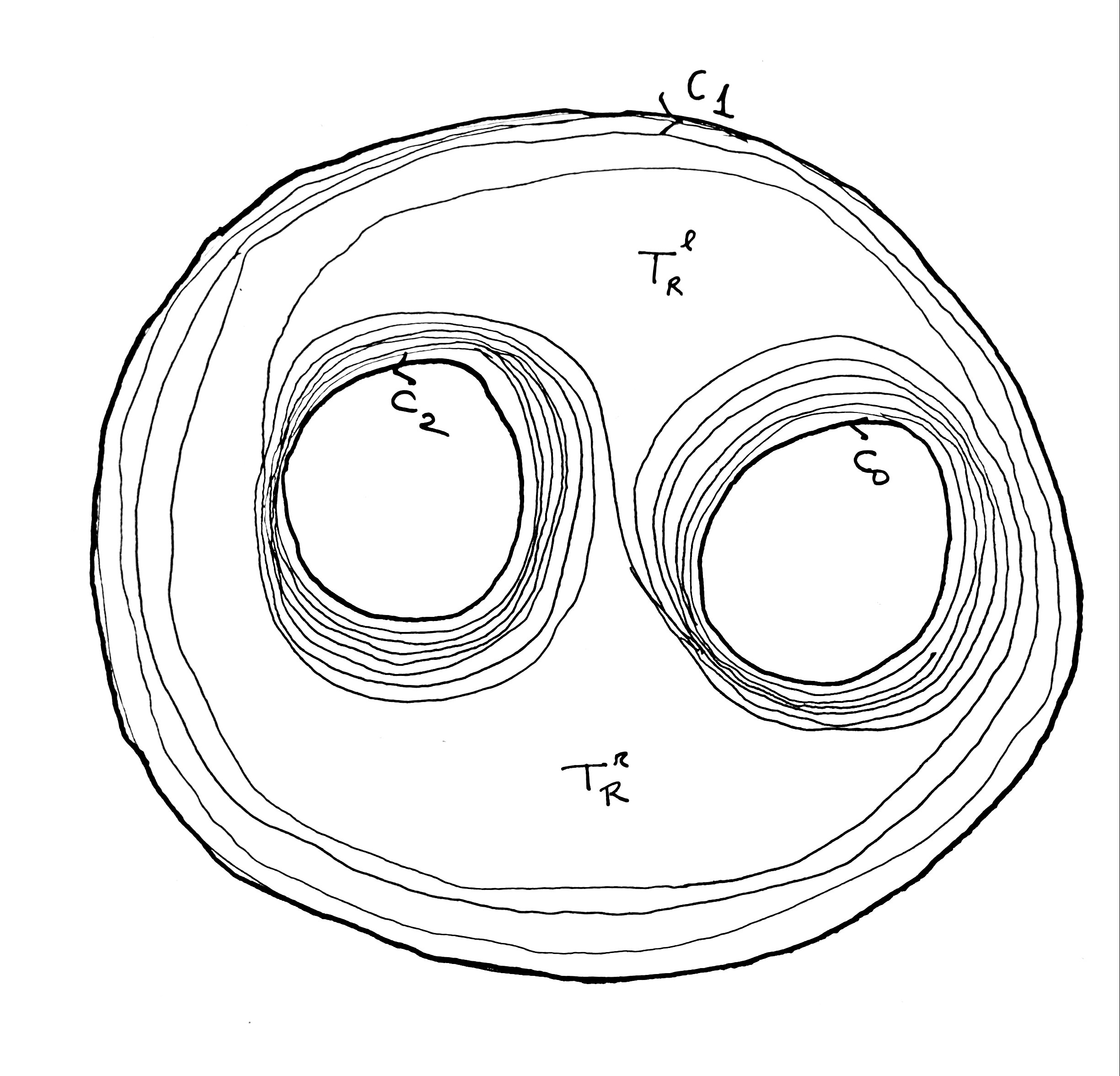}
\caption{
Spinning the hexagons of $P_R$ into two ideal triangles.}
\label{spun}
\end{figure}

We can turn the right-angled hexagons $H_R^{\ell}$ and $H_R^r$ into ideal triangles $T_R^{\ell}$ and $T_R^r$ by spinning their vertices around the cuffs, following their orientation. See Figure \ref{spun}.

Let $\pi\in \pants$ and  $f\in\pi$ be a pleated representative (so $f(P_R)$ is made out of two ideal triangles). We call $f(T^{\ell}_R)$ the \emph{left triangle} $T^{\ell}(\pi)$ of $\pi$ and $f(T^r_R)$ the \emph{right triangle} $T^r(\pi)$ of $\pi$. Note that these are well-defined as they do not depend on the choice of pleated representative in $\pi$.

Now let $\pi\in\tpants$ and $(f,C_i)\in \pi$ be a pleated representative. The \emph{left barycenter} of $\pi$, denoted $\bary^{\ell}(\pi)$, is the framed barycenter of $T^{\ell}(\pi)$ associated to the side $f(C_i)$. Similarly, the \emph{right barycenter} of $\pi$, denoted $\bary^r(\pi)$, is the framed barycenter of $T^r(\pi)$ associated to the side $f(C_i)$.

\subsection{Equidistribution of feet in $\Fr M$}

The goal of this subsection is to prove Lemma \ref{equidftfm}, in other words, that
\[
\phi_{\eps,R(\eps)} \wkstar \nu_{\Fr M}
\]
as $\eps \to 0$. To do so, we will use the fact that the feet of pants are well-distributed along a good curve. This is Theorem \ref{kw}, due to Kahn-Wright \cite{KW}, but we will use the modified version of Theorem \ref{ftalongcurve} below. The difference is that Theorem \ref{kw} is stated for counting feet in a subset $B$ of $\N^1(\sqrt{\gamma})$, whereas the counting we will do is weighted by a nonnegative function $g\in L^{\infty}(\N^1(\gamma))\se L^{\infty} (\Fr M)$.

For $\gamma\in \curves$, we let $\lambda^{\gamma}$ denote the probability Lebesgue measure in $\N^1(\gamma)\se \Fr M$. For a bounded function $g$ on a metric space, we let
\[
m_{\delta} (g) (p) = \inf_{B_{\delta}(p)} g \aand M_{\delta} (g)(p) = \sup_{B_{\delta}(p)} g,
\]
where $B_{\delta}(p)$ is the metric ball of radius $\delta$ around $p$.

\begin{thm}[Equidistribution of feet along a curve]\label{ftalongcurve}
There exists $q>0$ depending on $M$ such that for any $\epsilon > 0$, there is $R\geq R_0(\eps)$ so that the following holds. Let $\gamma\in \curves$. If $g \in L^{\infty} \lef( \Fr M \ri)$ is a nonnegative function, then
\[
(1-\delta) \int_{\N^1 (\gamma)} m_{\delta} (g) \,d\lambda^{\gamma} \leq 
\frac{1}{C_{\eps,R,\gamma}} \sum_{\pi\in\pants(\gamma)} \lef( g(\ft^{\ell} \pi) + g(\ft^r \pi) \ri) \leq
(1 + \delta) \int_{\N^1 (\gamma)} M_{\delta} (g) \,d\lambda^{\gamma},\tag{$\star$}
\]
where $\delta = e^{-qR}$,
\[
C_{\eps,R,\gamma} = \frac{2\pi c_{\eps} \eps^4 \ell(\gamma) e^{4R-\ell(\gamma)}}{\vol M}
\]
and $c_{\eps}\epszero 1$.
\end{thm}

\begin{proof}[Proof from Theorem \ref{kw}]
Let $g\in L^{\infty}(\Fr M)$. Since the measure $\lambda^{\gamma}$ is supported on $\N^1(\gamma)$, we may assume $g\in L^{\infty}(\N^1(\gamma))$. Let $h(n):= g(n) + g(n+\bh\bl(\gamma))$. Since $h$ is invariant under $n\mapsto n+\bh\bl(\gamma)$, $h$ descends to a function $\check{h}\in L^{\infty}(\N^1(\sqrt{\gamma}))$ so that $h\circ \proj = \check{h}$, where $\proj: \N^1(\gamma)\to\N^1(\sqrt{\gamma})$ is the quotient projection.

Using the shorthand notation
\[
\{f > y\}:= \{n\in \N^1(\sqrt{\gamma}) : f(n)>y\},
\]
note that
\[
N_{-\delta} \{f>y\} = \{m_{\delta}(f)>y\} \aand
N_{\delta} \{f> y\} = \{M_{\delta}(f)> y\}.
\]

Thus, Theorem \ref{kw} gives us
\begin{equation}\label{before}
(1-\delta)\,\lambda(\{m_{\delta}(\check{h})>y \})
\leq \frac{\#\{\pi\in\pants:\check{h}(\ft\pi)>y\}}{C_{\eps,R,\gamma}} \leq (1+\delta)\,\lambda(\{M_{\delta}(\check{h})> y\}),
\end{equation}
where $\lambda$ is the probability Lebesgue measure on $\N^1(\sqrt{\gamma})$.

A basic property of the Lebesgue integral says that for a function $f\in L^{\infty}(X,\mu)$, where $X$ is a space with a measure $\mu$, we have $\int_X f\,d\mu = \int_0^{\|f\|_{\infty}} \mu(\{f>y\}) \,dy$. Thus, if we integrate the inequality (\ref{before}) above with respect to $y$ from $0$ to $\|\check{h}\|_{L^{\infty}(\N^1(\sqrt{\gamma}))}$ and apply this property for $\lambda$ and the counting measure of feet in $\N^1(\sqrt{\gamma})$, we obtain
\[
(1-\delta)\int_{\N^1(\sqrt{\gamma})} m_{\delta}(\check{h})\,d\lambda \leq \frac{1}{C_{\eps,R,\gamma}} \sum_{\pi\in\pants(\gamma)} \check{h}(\ft \pi) \leq (1+\delta)\int_{\N^1(\sqrt{\gamma})} M_{\delta}(\check{h})\,d\lambda.
\]

Note that $\check{h}(\ft\pi) = g(\ft^{\ell}\pi)+g(\ft^r\pi)$, so the middle term of the inequality is the same as in $(\star)$. On the other hand, $m_{\delta}(\check{h})\circ \proj = m_{\delta}(h)$. Thus, $\int_{\N^1(\sqrt{\gamma})} m_{\delta} (\check{h})\, d\lambda = \int_{\N^1(\gamma)} m_{\delta} (h)\,d\lambda^{\gamma}$. Finally, $\int_{\N^1(\gamma)} m_{\delta} (h)\,d\lambda^{\gamma} \geq 2\int_{\N^1(\gamma)}  m_{\delta} (g)\, d\lambda^{\gamma}$ and similarly $\int_{\N^1(\sqrt{\gamma})} M_{\delta} (\check{h})\, d\lambda \leq 2\int_{\N^1(\gamma)} M_{\delta}(g)\,d\lambda^{\gamma}$. This yields the desired inequality $(\star)$, up to the constant $C_{\eps,R,\gamma}$ absorbing a factor of $2$.
\end{proof}

We can simplify the main statement of the theorem with the following notations. For a measure $\mu$ on a space $X$, and $g\in L^{\infty}(X)$, we let $\mu(g):=\int_X gd\mu$. We define a measure $\phi^{\gamma}_{\eps,R}$ supported on $\N^1(\gamma)$ by
\[
\phi^{\gamma}_{\eps,R} (g) =\frac{1}{C_{\eps,R,\gamma}} \sum_{\pi \in \pants(\gamma)} \lef( g(\ft^{\ell}\pi) + g(\ft^r \pi) \ri),
\]
where $g\in C(\Fr M)$.

Fix a nonnegative function $g\in C(\Fr M)$. The inequality $(\star)$ can be rewritten as
\[
(1-\delta)\, \lambda^{\gamma} \lef( m_{\delta} g \ri) \leq
\phi^{\gamma}_{\eps,R} (g) \leq
(1 + \delta) \,\lambda^{\gamma} \lef (M_{\delta} g \ri)
\]
and we can average it over all $\gamma\in\curves$, yielding
\[
(1-\delta) \frac{1}{\#\curves}\sum_{\gamma\in\curves} \lambda^{\gamma} \lef( m_{\delta} g \ri) \leq 
\frac{1}{\#\curves}\sum_{\gamma\in\curves}\phi^{\gamma}_{\eps,R} (g) \leq
(1 + \delta) \frac{1}{\#\curves}\sum_{\gamma\in\curves}\lambda^{\gamma} \lef (M_{\delta} g \ri).
\]
If we can show that the upper and lower bounds of this inequality are very close to $\nu_{\Fr M}(g)$ and that the middle term is very close to $\phi_{\eps,R} (g)$ as $\eps\to 0$, then it will follow that $\phi_{\eps,R} (g) \epszero \nu_{\Fr M} (g)$. This, in turn, implies Lemma \ref{equidftfm}, using the fact that since $M$ is compact, $C(\Fr M)\se L^{\infty}(\Fr M)$, as well as the fact that if $\phi_{\eps,R} (g) \epszero \nu_{\Fr M} (g)$ for nonnegative functions $g\in C(\Fr M)$, it follows that $\phi_{\eps,R}(g)\epszero \nu_{\Fr M} (g)$ for all $g\in C(\Fr M)$.

Our task is therefore to show the following two lemmas:

\begin{lem}\label{ft1}
For $g\in C(\Fr M)$,
\[
\lef| \frac{1}{\#\curves}\sum_{\gamma\in\curves} \lambda^{\gamma} \lef( m_{\delta} g \ri) - \nu_{\Fr M} (g) \ri| \aand \lef| \frac{1}{\#\curves}\sum_{\gamma\in\curves} \lambda^{\gamma} \lef( M_{\delta} g \ri) - \nu_{\Fr M} (g) \ri|
\] 
converge to zero as $\eps\to 0$. (Recall $\delta = e^{-qR(\eps)}$ goes to zero as $\eps\to 0$.)
\end{lem}

\begin{lem}\label{ft2}
For $g\in C(\Fr M)$,
\[
\lef| \frac{1}{\#\curves}\sum_{\gamma\in\curves}\phi^{\gamma}_{\eps,R} (g) - \phi_{\eps,R} (g) \ri| \epszero 0.
\]
\end{lem}

\begin{proof}[Proof of Lemma \ref{ft1}]
For $g\in C(\Fr M)$, we define a function $\hat{g} \in C(\T^1 M)$ via
\[
\hat{g}(p,v) = \frac{1}{2\pi} \int_{S^1(v)} g(p,v,\theta)\,d\theta,
\]
where $S^1(v)$ is the circle in $\T^1_p M$ orthogonal to $v$. For $\gamma\in\curves$, we let $d\gamma$ be the probability length measure of $\gamma$ on $\T^1 M$, in other words, for $h\in C(\T^1 M)$.
\[
\int_{\T^1 M} h\,d\gamma = \frac{1}{\ell(\gamma)} \int_0^{\ell(\gamma)} h(\gamma(t),\gamma'(t)) \,dt.
\]
Note that for $g\in C(\Fr M)$,
\[
\int_{\T^1 M} \hat{g} \,d\gamma = \int_{\Fr M} g \, d\lambda^{\gamma}.
\]

Let $\Prob_{\eps}$ be the uniform probability measure on $\curves$. In Theorem II of \cite{L}, Lalley showed that, if $h\in C(\T^1 M)$ and $\eta>0$, then
\[
\Prob_{\eps} \lef( \lef| \int_{\T^1 M} h\, d \gamma - \nu_{\T^1 M} (h) \ri| > \eta \ri) \epszero 0.
\]
In other words, if $g\in C(\Fr M)$, then
\[
\Prob_{\eps} \lef( \lef| \lambda^{\gamma}(g) - \nu_{\Fr M} (g) \ri| > \eta \ri) \epszero 0.
\]

Let $\curves^{\geq \eta}$ be the $\gamma\in\curves$ so that $|\lambda^{\gamma}(g) - \nu_{\Fr M}(g)| \geq \eta$ and let $\curves^{<\eta} := \curves - \curves^{\geq\eta}$. Then,

\begin{align*}
\lef|\frac{1}{\#\curves} \sum_{\gamma\in\curves} \lambda^{\gamma}(g) - \nu_{\Fr M} (g) \ri| &\leq
\frac{1}{\#\curves} \lef( \sum_{\gamma\in\curves^{<\eta}} \lef|\lambda^{\gamma}(g)- \nu_{\Fr M}(g)\ri| +\sum_{\gamma\in\curves^{\geq\eta}} \lef|\lambda^{\gamma}(g)- \nu_{\Fr M}(g)\ri| \ri) \\
&\leq \eta + 2\|g\|_{L^{\infty} (\Fr M)} \Prob_{\eps} (\curves^{>\eta}).
\end{align*}

As $\eta>0$ was arbitrary, this shows that
\[
\lim_{\eps\to 0} \frac{1}{\#\curves} \sum_{\gamma\in\curves} \lambda^{\gamma}(g) = \nu_{\Fr M } (g).
\]

Finally, since $g$ is continuous and $\delta\to 0$ as $\eps\to 0$, we also conclude that
\[
\frac{1}{\#\curves} \sum_{\gamma\in\curves} \lambda^{\gamma}\lef( m_{\delta} g\ri) \aand \frac{1}{\#\curves} \sum_{\gamma\in\curves} \lambda^{\gamma}\lef( M_{\delta} g\ri)
\]
converge to $\nu_{\Fr M}(g)$ as $\eps\to 0$, which concludes the proof of Lemma \ref{ft1}.
\end{proof}

\begin{proof}[Proof of \ref{ft2}]

To begin, for $\gamma\in\curves$ and for $g\in C(\Fr M)$, we define
\[
\Ft_{\gamma}(g) := \sum_{\pi\in\pants(\gamma)} \lef( g(\ft^{\ell} \pi) + g(\ft^r \pi) \ri).
\]

Note that $\#\tpants = \sum_{\gamma\in\curves}\#\pants^- (\gamma)$, or equivalently, $2\#\tpants = \sum_{\gamma\in\curves} \#\pants(\gamma)$. This is because each oriented good curve corresponds to a distinct element of $\curves$, while each $\#\pants(\gamma)$ counts the ends of pants having either $\gamma$ or its orientation reversal as a cuff.

In particular,
\begin{align*}
\phi_{\eps,R}(g) &= \frac{1}{2\#\tpants} \sum_{\pi\in\tpants} \lef( g(\ft^{\ell} \pi) + g(\ft^r \pi) \ri) \\
&=  \frac{1}{2\#\tpants} \sum_{\gamma\in\curves} \sum_{\pi\in \pants^-(\gamma)} \lef( g(\ft^{\ell} \pi) + g(\ft^r \pi) \ri)  \\
&= \frac{1}{4\#\tpants} \sum_{\gamma\in\curves} \Ft_{\gamma} (g).
\end{align*}

Moreover,
\begin{align*}
\frac{1}{\#\curves}\sum_{\gamma\in\curves} \phi_{\eps,R}^{\gamma}(g) - \phi_{\eps,R} (g) &=
\frac{1}{\#\curves} \sum_{\gamma\in\curves} \frac{\Ft_{\gamma}(g)}{C_{\eps,R,\gamma}} - \frac{1}{4\#\tpants} \sum_{\gamma\in\curves} \Ft_{\gamma}(g) \\
&= \sum_{\gamma\in\curves} \Ft_{\gamma}(g) \lef(\frac{1}{\#\curves C_{\eps,R,\gamma}} - \frac{1}{4\#\tpants}  \ri).
\end{align*}

Thus, we obtain
\begin{align*}
\lef| \sum_{\gamma\in\curves}\frac{1}{\#\curves} \phi_{\eps,R}^{\gamma}(g) - \phi_{\eps,R} (g) \ri| &\leq
\frac{\#\curves}{\#\tpants} \sup_{\gamma\in\curves}  |\Ft_{\gamma}(g)| \lef|\frac{1}{\#\curves C_{\eps,R,\gamma}} - \frac{1}{4\#\tpants}  \ri|.
\end{align*}
Using the fact that
\[
|\Ft_{\gamma}(g)| \leq 2 \#\pants(\gamma)\, \|g\|_{L^{\infty}(\Fr M)},
\]
we have
\begin{align}\label{bound}
\lef| \sum_{\gamma\in\curves}\frac{1}{\#\curves} \phi_{\eps,R}^{\gamma}(g) - \phi_{\eps,R} (g) \ri| &\leq
\|g\|_{L^{\infty}(\Fr M)} \sup_{\gamma\in\curves} \frac{2\#\pants(\gamma)}{C_{\eps,R,\gamma}} \lef| 1 - \frac{C_{\eps,R,\gamma}}{4\#\tpants/\#\curves} \ri| \\
&\leq \|g\|_{L^{\infty}(\Fr M)} \, (1+\delta) \sup_{\gamma\in\curves}\lef| 1 - \frac{C_{\eps,R,\gamma}}{4\#\tpants/\#\curves} \ri|,
\end{align}
where $\delta = e^{-qR}$ and in the last step we have used the equidistribution of feet (Theorem \ref{ftalongcurve}).

Now it remains for us to show that
\[
(\star):= \sup_{\gamma\in\curves} \lef| 1 - \frac{C_{\eps,R,\gamma}}{4\#\tpants/\#\curves} \ri|
\]
tends to 0 as $\eps\to 0$,
which we will also achieve with the equidistribution of feet and some algebraic manipulations.

Recall that $C_{\eps,R,\gamma} = (\vol M)^{-1} 2\pi c_{\eps} \eps^4 \ell(\gamma) e^{4R - \ell{\gamma}}$, where $c_{\eps}\to 1$ as $\eps\to 0$. We define a new constant,
\[
C_{\eps,R} := \frac{2\pi c_{\eps} 2R e^{2R}}{\vol M},
\]
that is very close to $C_{\eps,R,\gamma}$ but does not depend on $\gamma$.

Note that
\begin{align*}
(\star) = \sup_{\gamma\in\curves} \lef| \frac{\frac{4\#\tpants/\#\curves}{C_{\eps,R}} - \frac{C_{\eps,R,\gamma}}{C_{\eps,R}}}{\frac{4\#\tpants/\#\curves}{C_{\eps,R}}} \ri| \leq \sup_{\gamma\in \curves} \frac{\lef|1 - \frac{4\#\tpants/\#\curves}{C_{\eps,R}} \ri| + \lef| 1 - \frac{C_{\eps,R,\gamma}}{C_{\eps,R}}\ri| }{\frac{4\#\tpants/\#\curves}{C_{\eps,R}}}.
\end{align*}

It suffices to show that, as $\epsilon\to 0$, the right hand side goes to 0.

As $\gamma$ is an $(\eps,R)$-good curve, we have
\begin{align}\label{nogamma}
\lef| 1 - \frac{C_{\eps,R,\gamma}}{C_{\eps,R}} \ri| = \lef| \frac{\ell(\gamma)\,e^{2R-\ell(\gamma)} - 2R}{2R} \ri| &\leq \frac{|\ell(\gamma)||e^{2R-\ell(\gamma)} - 1| + |\ell(\gamma)-2R|}{2R} \nonumber \\
&\leq \lef( 1 + \frac{\eps}{2R}\ri) (e^{2\eps}- e^{-2\eps}) + \frac{\eps}{R},
\end{align}
where the last term in the inequality, let us call is $\omega_1(\eps)$, goes to 0 as $\eps\to 0$.

The final task is now to show
\[
\lim_{\eps\to 0} \lef|1 - \frac{4\#\tpants/\#\curves}{C_{\eps,R}} \ri| = 0.
\]

From the estimate \ref{nogamma} and equidistribution of feet, we have
\[
\lef| 1 - \frac{2\# \pants(\gamma)}{C_{\eps,R}} \ri| \leq \frac{\lef| 1 -  \frac{C_{\eps,R}}{C_{\eps,R,\gamma}} \ri| + \lef| 1 - \frac{2\#\pants(\gamma)}{C_{\eps,R,\gamma}} \ri|}{\frac{C_{\eps,R}}{C_{\eps,R,\gamma}}} \leq 2 ( \omega_1 (\eps) + \delta) =: \omega_2(\eps).
\]
Above, $\omega_2(\eps)\to 0$ as $\eps\to 0$. In other words,
\[
1 - \omega_2(\eps) \leq \frac{2\# \pants(\gamma)}{C_{\eps,R}} \leq 1 + \omega_2(\eps).
\]
Averaging this inequality over $\gamma$ and recalling that $\sum_{\gamma\in\curves} \#\pants(\gamma) = 2\# \tpants$, we obtain
\[
1 - \omega_2(\eps) \leq \frac{4\#\tpants/\#\curves}{C_{\eps,R}} \leq 1 + \omega_2(\eps).
\]
Thus we have all we needed to show that $(\star)\epszero 0.$
\end{proof}

% The equidistribution of feet along a curve, Theorem \ref{ftalongcurve}, is what allows us to argue that the right hand side goes to zero as $\eps\to 0$. Say $a(\eps)\sim b(\eps)$ if $a(\eps)/b(\eps) \to 1$ as $\eps \to 0$. Then, Theorem \ref{ftalongcurve} applied to $g = 1_{\Fr M}$ says that
% \[
% 1-\delta \leq \frac{2\#\pants(\gamma)}{C_{\eps,R,\gamma}} \leq 1+\delta,
% \]
% where $\delta = e^{-qR}$, which implies
% \[
% 2\#\pants(\gamma) \sim C_{\eps,R,\gamma}
% \]
% for any $\gamma\in\curves$, as well as
% \[
% \frac{\#\tpants}{\#\curves} \sim \frac{1}{\#\curves}\sum_{\gamma\in\curves} C_{\eps,R,\gamma}.
% \]
% But from the definition of $C_{\eps,R,\gamma}$, we have
% \[
% C_{\eps,R,\gamma} \sim \frac{1}{\#\curves}\sum_{\gamma\in\curves} C_{\eps,R,\gamma} \sim \til{c}_{\eps} \eps^4 e^{2R}R,
% \]
% where $\til{c}_{\eps}$ is bounded in $\eps$. Thus,
% \[
% 2\#\pants(\gamma) \sim C_{\eps,R,\gamma} \sim \frac{\#\tpants}{\#\curves},
% \]
% which allows us to conclude that the right hand side of \ref{bound} goes to zero as $\eps\to 0$.\end{proof}

To wrap up this section, we have proved Lemmas \ref{ft1} and \ref{ft2}, which is what we needed to show that the feet of all pants equidistribute in $\Fr M$, namely
\[
\phi_{\eps,R(\eps)} \wkstar \nu_{\Fr M}.
\]

\subsection{Approximate barycenters of pants}

As before, we let $v_R$ be the right action on $\Fr M\simeq \PSL_2\bC$ of the element
\[
a_{R/2}\, k\, a_{\log(\sqrt{3}/2)} \in \PSL_2 \bC,
\]
where $a_t = \diag(e^{t/2},e^{-t/2})$ and $k\in \SO_2$ is the ninety-degree rotation bringing the first frame to the second.

We defined \emph{left} and \emph{right approximate barycenter} of an end of pants $\pi\in\tpants$ respectively by
\[
\ab^{\ell} (\pi) = v_R (\ft^{\ell} \pi) \aand
\ab^r (\pi) = v_R (\ft^r \pi).
\]

This subsection is dedicated to proving that the aproximate barycenters are indeed close to barycenters. Precisely,

\begin{lem}\label{ab}
Let $\pi\in \tpants$ and $s = \ell$ or $r$. Then,
\[
\dist_{\Fr M} \lef( \ab^{s}\pi,\bary^s \pi \ri) \leq \omega(\eps),
\]
where $\omega(\eps)\to 0$ as $\eps\to 0$.
\end{lem}

\begin{proof}

\begin{figure}
\includegraphics[scale=0.1]{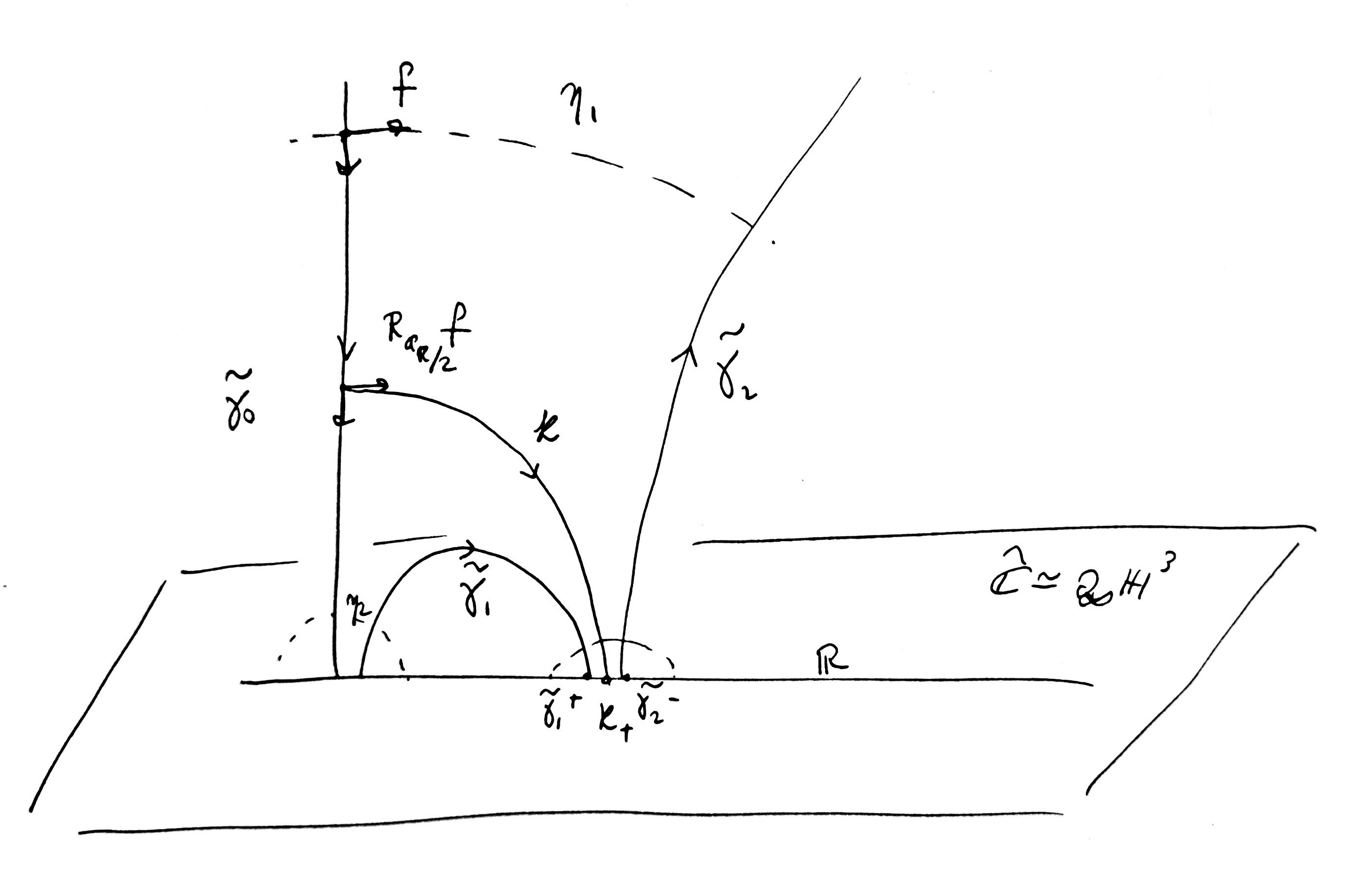}
\caption{Lifting the pants $\pi$ and its left foot
to $\bH^3$.}
\label{kappafig}
\end{figure}

Throughout the proof, we will let $\omega(\eps)$ denote any quantity that goes to zero as $\eps\to 0$.

Let $\pi = [(f,C_i)] \in \tpants$. Pick a representative $(f,C_i)$ with geodesic cuffs and let $f(C_j) = \gamma_{j-i}$, for $i,j \in \bZ/3$. Without loss of generality and to be explicit, we assume $f$ to be orientation-preserving.

Lift $\gamma_0$ to the geodesic $\til{\gamma}_0$ from $\infty$ to $0$ in $\hat{\bC}\simeq \partial_{\infty}\bH^3$. Lift the left foot $\ft^{\ell}\pi$ to the frame based $f$ at $e^{R/2} i$, whose first vector points at the direction of $\gamma$ and whose second vector points at the positive direction of the real line $\bR\se \hat{\bC}$.

Let $\kappa$ be the ray given by
\[
\kappa(t) := R_{a_t} R_{k} R_{a_{R/2}} f
\]
for $t\geq 0$.
The left approximate barycenter $\ab^{\ell}\pi$ lifts to the framed barycenter $\kappa(\log(\sqrt{3}/2))$ of the triangle with vertices $(\infty, 0,\kappa_+)$ associated to the side $(\infty,0) = \til{\gamma}_0$.

Choose lifts $\til{\gamma}_1$ and $\til{\gamma}_2$ of the other cuffs of $\pi$ so they are connected to $\til{\gamma}_0$ by lifts of the short orthogeodesics, as in Figure \ref{kappafig}. Note that $\til{\gamma}_1$ and $\til{\gamma}_2$ lie, respectively, in geodesic planes $P_1$and $P_2$ with $\infty$ in their boundary that make an angle of $\omega(\eps)$ with each other and the plane $P_0$ that contains $\til{\gamma}_0$ and $(\kappa(t))_{t\geq 0}$.

The left triangle of $\pi$ lifts to the triangle with vertices $(\infty,\til{\gamma}_1^-,\til{\gamma}_2^-)$ in this picture.

We let
\[
\sigma(t) := R_{a_t} R_k R_{a_{\bh\bl(\gamma)/2}} f,
\]
for $t\geq 0$, 
be a lift of the orthogeodesic ray from $\gamma_0$ to itself. Then, $\sigma^+$ lies in the annulus $E$ of $\hat{\bC}$ so that $\til{\gamma}_1^+,\til{\gamma}_2^- \in \partial E$. Since the geodesic planes $P_1$ and $P_2$ containing $\til{\gamma}_1$ and $\til{\gamma}_2$ make a small angle with each other, it follows that $\sigma^+$ is within distance $\omega(\eps)$ in $\hat{\bC}$ of $\til{\gamma}_1^+$ and $\til{\gamma}_2^-$.

On the other hand, since $R_{a_{R/2}} f$ and $R_{a_{\bh\bl(\gamma)/2}} f$ are at a distance $O(\eps)$ of each other in $\Fr \bH^3$, it follows that $\sigma^+$ and $\kappa^+$ are at a distance $O(\eps)$ in $\hat{\bC}$.

We conclude that the vertices $(\infty,\til{\gamma}_1^-,\til{\gamma}_2^-)$  of the left triangle of $\pi$ are within distance $\omega(\eps)$ of the vertices $(\infty,0,\kappa^+)$ of the triangle whose framed barycenter associated to $(\infty,0)$ is $\ab^{\ell} \pi$. This means that all the framed barycenters of these triangles are within $\omega(\eps)$ of each other in $\Fr \bH^3$. Thus,
\[
\dist_{\Fr M} (\ab^{\ell}\pi,\bary^{\ell}\pi) \leq \omega(\eps),
\]
as desired.

The proof follows in the same way for the right barycenters.
\end{proof}

\subsection{Equidistribution of approximate  barycenters in $\Fr M$}

In this subsection, we will show that the probability uniform measure $\beta^a_{\eps,R}$ supported on the approximate barycenters of $\pi\in\tpants$ equidistributes as $\eps\to 0$ and $R(\epsilon)\to\infty$. 

We first claim that the equidistribution of feet on $\Fr M$, Lemma \ref{equidftfm}, still holds if we move the feet by the right action of $a_{R/2}$, i.e.,

\begin{lem}
As $\eps\to 0$ and $R(\eps)\to\infty$,
\[
(R_{a_{R/2}})_* \phi_{\eps,R} \wkstar \nu_{\Fr M}.
\]
\end{lem}

\begin{proof}
We want to verify that the arguments of the original proof of Lemma \ref{equidftfm} still hold when we change $g$ to $g_R:= g\circ R_{a_{R/2}}$. What makes this not obvious is that the function $g_R$ is not fixed as we take our $\eps\to 0$ limit.

Let $g\in C(\Fr M)$ be nonnegative.
Along each $\gamma\in\curves$, the Lebesgue probability measure $\lambda^{\gamma}$ is invariant under the right action of $a_t \in \PSL_2 \bC$. Therefore, the equidistribution of feet along a curve, Lemma \ref{ftalongcurve} gives us
\[
(1-\delta)\lambda^{\gamma}\lef(m_{\delta}g \ri)
\leq \phi^{\gamma}_{\eps,R} (g_R)
\leq (1+\delta) \lambda^{\gamma} (M_{\delta} g),
\]
where as usual $\delta= e^{-qR}$. Following logic of the proof of Lemma 5.2, we average this inequality along curves on $\curves$, and we want to show that the upper and lower bounds are close to $\nu_{\Fr M}(g)$, which is exactly Lemma 5.6, and that the term in the middle is close to $\phi_{\eps,R} (g_R)$, which is Lemma 5.7 with $g_R$ instead of $g$. In fact, the proof of Lemma 5.7 remains largely the same with $g_R$; the one thing to note is that $g_R$ has the same $L^{\infty}$ norm as $g$ in $\Fr M$.
\end{proof}

We also observe
\begin{lem}
Suppose $\nu_i$ are probability measures in $\Fr M$ so that $\nu_i \wkstar \nu_{\Fr M}$ as $i\to\infty$. Then, given $h\in \PSL_2 \bC$, we have
\[
(R_h)_* \nu_i \wkstar \nu_{\Fr M}
\]
\end{lem}

\begin{proof}
Suppose $g\in C(\Fr M)$. Then, 
\[
(R_h)_* \nu_i (g) = \nu_i (g\circ R_h^{-1}).
\]
Thus $(R_h)_*\nu_i (g) \to \nu_{\Fr M}(g\circ R_h^{-1})$ as $i\to\infty$. As $\nu_{\Fr M}$ is invariant under the right action of $\PSL_2\bC$, we conclude.
\end{proof}

By construction,
\[
\beta^{a}_{\eps,R} = (R_{k a_{\log (\sqrt{3}/2)}})_* (R_{a_{R/2}})_* \phi_{\eps,R}.
\]
Combining the two lemmas above, we conclude that $\beta^{a}_{\eps,R} \wkstar \nu_{\Fr M}$ as $\eps\to 0$ and $R(\eps)\to \infty$.

\subsection{Conclusion: equidistribution of barycenters in $\Fr M$}

Finally, as a corollary of the equidistribution of the approximate barycenters, we obtain the equidistribution of the actual barycenters, which is the main theorem of the section. For $g\in C(\Fr M)$, we have

\[
\beta^a_{\eps,R}(g) - \beta_{\eps,R}(g) = 
\frac{1}{2\#\tpants} \sum_{\pi \in \tpants}\sum_{s\in \{\ell,r\}} \lef( g(\ab^s \pi) - g(\bary^s \pi)   \ri)
\]

As $g$ is uniformly continuous and
\[
\dist_{\Fr M} (\ab^s \pi,\bary^s \pi) \leq \omega(\eps),
\]
where $\omega(\eps)\to 0$ as $\eps \to 0$, we conclude that
\[
|\beta^a_{\eps,R}(g) - \beta_{\eps,R}(g)| \epszero 0.
\]
Thus, since $\beta^a_{\eps,R}(g) \to \nu_{\Fr M}(g)$, we conclude that $\beta_{\eps,R}(g) \to \nu_{\Fr M}(g)$ as $\eps \to 0$ and $R(\eps)\to\infty$.

\section{Equidistributing surfaces}

Let $S_{\epsilon,R}$ be the connected, closed, $\pi_1$-injective and $(1+O(\epsilon))$-quasifuchsian surface made out of $N(\epsilon,R)$ copies of each pants in $\pants$, as explained in Section 2. Let $\nu_{S_{\epsilon,R}}$ be their probability area measures on $\Gr M$. Note that these measures are also the probability area measure of the possibly disconnected surface built out of one copy of each $\pi\in\pants$. Using the fact that the barycenters of good pants are well distributed, we will show

\begin{thm} As $\epsilon\to 0$ and $R(\epsilon)\to\infty$ fast enough,
\[
\nu_{S_{\epsilon,R(\epsilon)}} \wkstar \nu_{\Gr M},
\]
where $\nu_{\Gr M}$ is the probability volume measure of $\Gr M$.
\end{thm}

Outside of the pleating lamination, we may define the unit tangent bundle $\T^1 S_{\eps,R}$ of $S_{\eps,R}$. This can be seen as a three-dimensional submanifold of $\Fr M$, where $(p,v)\in \T^1 S_{\eps,R}$ is included in $\Fr M$ as the frame $(p,v,w,v\times w)$, where $w$ is the image of $v$ under the ninety-degree rotation $k\in \SO_2$ described in the last section.

Let $\nu_{\T^1 S_{\eps,R}}$ be the probability volume measure of $\T^1 S_{\eps,R}$ on $\Fr M$. We show

\begin{claim}
As $\eps\to 0$ and $R(\eps)\to\infty$,
\[
\nu_{\T^1 S_{\eps,R(\eps)}} \wkstar \nu_{\Fr M}.
\]
\end{claim}

\begin{proof}
This proof is similar to pages 23-26 of \cite{La}. 

Let $\Delta \se \PSL_2 \bR$ be the set so that $R_{\Delta} (b)$ is the unit tangent bundle of the ideal triangle in $M$ with $b\in \Fr M$ as a framed barycenter (for any $b\in \Fr M$). Let $\nu_{\PSL_2 \bR}$ be the probability Haar measure on $\PSL_2 \bR$.

Thus, given $g\in C(\Fr M)$,

\[
\nu_{\T^1 S_{\eps, R}} (g) = \int_{\Fr M} \frac{1}{\nu_{\PSL_2 \bR}(\Delta)}\int_{\Delta} g(R^{-1}_t b) \,d\nu_{\PSL_2 \bR}(t)\, d\beta_{\eps,R}.
\]

By Fubini's theorem,
\[
\nu_{\T^1 S_{\eps, R}} (g) = \frac{1}{\nu_{\PSL_2 \bR}(\Delta)}\int_{\Delta} \beta_{\eps,R} (g\circ R^{-1}_t) \,d\nu_{\PSL_2\bR} (t).
\]

From the equidistribution of the barycenters and the $\PSL_2 \bC$-invariance of $\nu_{\Fr M}$, the integrand $\beta_{\eps,R}(g\circ R_t^{-1})$ converges to $\nu_{\Fr M} (g)$. Using the dominated convergence theorem, we conclude that
\[
\nu_{\T^1 S_{\eps,R}} (g) \epszero \frac{1}{\nu_{\PSL_2 \bR}(\Delta)}\int_{\Delta}  \nu_{\Fr M} (g) \,d\nu_{\PSL_2\bR} = \nu_{\Fr M} (g).
\]
\end{proof}

For $g\in C(\Fr M)$, we let $\til{g}\in C(\Gr M)$ be the function defined by
\[
\til{g}(p,P) = \frac{1}{2\pi} \int_{0}^{2\pi} g(p,r_{\theta} f)\,d\theta,
\]
where $r_{\theta}\in \PSL_2 \bR$ is the rotation of degree $\theta$ and $f$ is the frame whose first two vectors span the oriented plane $P$.

Fubini's theorem tells us that for $g\in C(\Fr M)$,
\[
\nu_{\T^1 S_{\eps,R}} (g) = \nu_{S_{\eps,R}} (\til{g}).
\]
Moreover, any $h\in C(\Gr M)$ is of the form $h=\til{g}$, where $h(p,P) = g(p,f)$ for any frame $f$ whose first two vectors span $P$.

Thus Claim 6.2 implies Theorem 6.1, and we can conclude that the connected surfaces made out of $N(\eps,R,M)$ copies of each pants in $\pants$ equidistribute in $\Gr M$ as $\eps\to 0$ and $R(\eps)\to\infty$.

\section{Non-equidistributing surfaces}

Let $\sG$ be a set containing a representative of each commensurability class of closed immersed totally geodesic surfaces in $M$. Let $(\sG_k)_{k\ge 1} \seq \sG$ be an increasing sequence of finite subsets, so that $\bigcup_{k\geq 1} \sG_k = \sG$. (In the case when $\sG$ is finite, it suffices to take $\sG_k = \sG$ for all $k\geq 1$.)
Kahn and Marković \cite{KME} proved that given $k\geq 1$, $\eps>0$ small enough and $R = R(\eps,k)$ large enough, each $T\in \sG_k$ has a finite cover $\hat{T}$ which admits a pants decomposition of pants in $\pants$ that are all glued via $(\eps,R)$-good gluings. (This fact was used to prove the Ehrenpreis conjecture.) By possibly passing to a further double cover, we can assume the cuffs of each $\hat{T}$ are all nonseparating, as explained in the proof of Theorem \ref{irred}.

For each $T\in \sG_k$, let $T^d$ be a cover of $\hat{T}$ of degree $d = d(T,\eps,R)$. We may choose this cover so that $T^d$  also admits a pants decomposition, denoted $\Pi_T$, by pants in $\pants$ that are glued via $(\eps,R)$-good gluings.

Let $\hat{S}_{\eps,R}$ be the connected, closed, $\pi_1$-injective and $(1+O(\epsilon))$-quasifuchsian surface produced in the previous section. We may assume that $\hat{S}_{\eps,R}$ is built out of $N(\eps,R,k)\geq \#\sG_k$ copies of each $\pi\in\pants$.

For each $T\in\sG_k$, choose a curve $\gamma\se T^d$ that arises as a boundary of a pants in $\Pi_T$. Let $\pi_T^- \in \pants^- (\gamma_T)$  and $\pi_T^+ \in \pants^+(\gamma_T)$ be pants in $\Pi_T$ that are $(\eps,R)$-well glued along $\gamma_T$. Namely,
\[
\lef|\ft \pi^-_T - \tau\lef(\ft \pi^+_T\ri) \ri| < \frac{\eps}{R},
\]
where as before $\tau(x) = x + 1 + i\pi$.

As argued in the proof of Theorem \ref{irred}, since $\hat{S}_{\eps,R}$ is built out of $N$ copies of \emph{each} $\pi \in \pants$, there is a pants $p^-_T\in \pants^-(\gamma_T)$ in $\hat{S}_{\eps,R}$ so that
\[
|\ft \pi^-_T - \ft p^-_T | < \frac{\eps}{R}.
\]
(Note we could take $p^-_T$ to be another copy of $\pi^-_T$ if $N>1$.)
On the other hand, $p^-_T$ is $(\eps,R)$-well glued to a pants $p^+_T$ also from $\hat{S}_{\eps,R}$, i.e.,
\[
\lef|\ft p^-_T - \tau \lef( \ft p^+_T \ri) \ri| < \frac{\eps}{R}.
\]

Putting these inequalities together, we have that
\[
\lef|\ft p^-_T - \tau\lef(\ft \pi^+_T\ri) \ri| < 2\frac{\eps}{R} \aand 
\lef|\ft \pi^-_T - \tau\lef(\ft p^+_T\ri) \ri| < 2\frac{\eps}{R}.
\]
In other words, we may cut along $\gamma_T$ and reglue $p^-_T$ to $\pi^+_T$ and $\pi^-_T$ to $p^+_T$ in a $(2\eps,R)$-good way. We call this reglued surface $S_{\eps,R,\bd}$, where $\bd = (d(T,\eps,R))_{T\in\sG_k}$ is a vector keeping track of the degrees of each cover $T^d\to \hat{T}$.

The regluings are done along nonseparating curves, so each $S_{\eps,R,\bd}$ is closed, oriented, \emph{connected} and $(1+O(\eps))$-quasifuchsian. (The connectedness uses the fact that the regluings were done along nonseparating cuffs.) As usual, we let $\nu(S_{\eps,R,\bd})$ denote the probability area measure of $S_{\eps,R,\bd}$ on the Grassmann bundle $\Gr M$. Recall that $\nu_{\Gr M}$ denotes the Haar measure on $\Gr M$ and $\nu_T$ denotes the area measure of $T$ on $\Gr M$.

\begin{figure}
\includegraphics[scale=0.16]{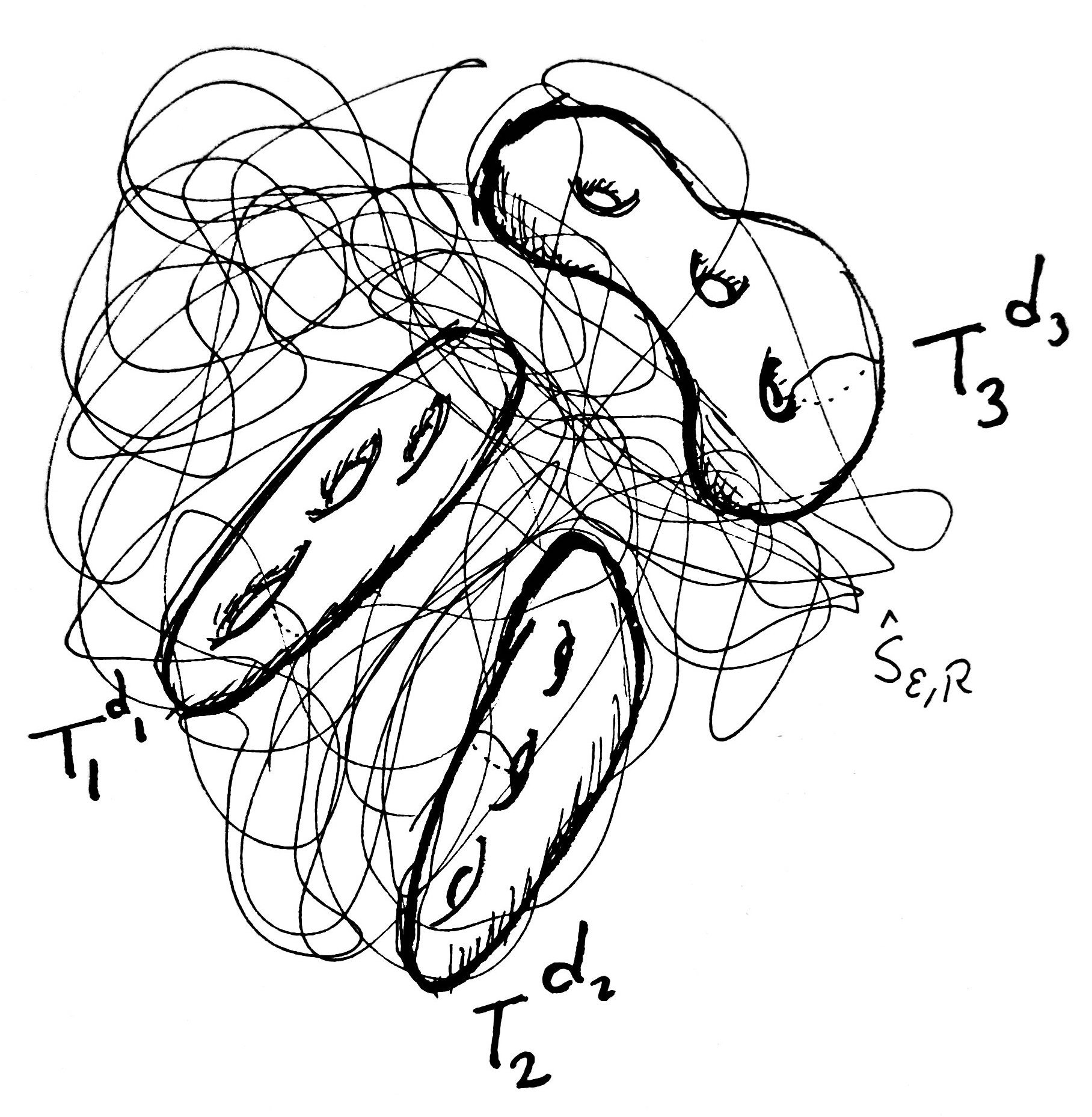}
\caption{The family of surfaces $S_{\eps,R,\bd}$, which can accumulate on the totally geodesic surfaces $T$ by appropriately choosing the degrees $d$ of their covers.}
\end{figure}

\begin{prop}
The weak-* limit points, as $\eps\to 0$ and $k\to\infty$, of the measures $\nu(S_{\eps,R(\eps,k),\bd(\eps,k)})$ on $\Gr M$ consist of all measures $\nu$ of the form
\[\tag{$\star$}
\nu = \alpha_M \nu_{\Gr M} + \sum_{T\in\sG} \alpha_T \nu_T.
\]
\end{prop}

\begin{proof}
Let $g_T$ denote the genus of a totally geodesic surface $T$ and let $g_{\eps,R}$ denote the genus of $\hat{S}_{\eps,R}$.

We can write the area measure $\nu(S_{\eps,R,\bd})$ as
\[
\nu(S_{\eps,R,\bd}) = \frac{1}{\sum_{T\in \sG_k} 2\pi(g_T-1)d(T) + 2\pi(g_{\eps,R} - 1)} \lef( \sum_{T\in \sG_k} 2\pi(g_T-1) d(T) \nu_T + 2\pi(g_{\eps,R} - 1)\nu(\hat{S}_{\eps,R}) \ri).
\]

Recall that $\nu(\hat{S}_{\eps,R(\eps)})\wkstar \nu_{\Gr M}$ as $\eps\to 0$. Thus, by making $d(T,\eps,R(\eps,k))$ grow appropriately fast for each $T$, we can make $\nu(S_{\eps,R,\bd})$ converge to any given measure of the form $(\star)$ as $\eps\to 0$ and $k\to\infty$.
\end{proof}

This, together with Theorem 1.2, which was proved in Section 4, completes the proof of Theorem 1.1.

\end{document}